\documentclass[english,11pt]{smfart}

\usepackage{appendix}
\usepackage{etex}

\usepackage{amsbsy}
\usepackage{amsmath,amsfonts,amssymb,amsthm,mathrsfs,mathtools}

\usepackage{bm}

\usepackage[a4paper,vmargin={3cm,3cm},hmargin={3.5cm,3.5cm}]{geometry}
\usepackage[font=sf, labelfont={sf,bf}, margin=1cm]{caption}
\usepackage{graphicx}
\usepackage{epsfig}
\usepackage{latexsym}
\usepackage{xcolor}
\usepackage{ae,aecompl}
\usepackage{soul,framed}
\usepackage{comment}

\usepackage{xcolor}
\usepackage[pdfpagemode=UseNone,bookmarksopen=false,colorlinks=true,urlcolor=blue,citecolor=blue,citebordercolor=blue,linkcolor=blue]{hyperref}
\usepackage{smfhyperref}
\usepackage[capitalize]{cleveref}

\usepackage{pstricks}
\usepackage{enumerate}
\usepackage{tikz,animate,media9}						
\usepackage{todonotes}
\usepackage{pifont}
\usepackage{bm,marvosym}
\usepackage{algorithm}
\usepackage{algorithmic}

\usepackage{caption}
\usepackage{subcaption}

\usepackage{bbm}

\usepackage{tcolorbox}

\definecolor{aleacolor}{rgb}{0.16,0.59,0.78}

\hypersetup{
	breaklinks,
	colorlinks=true,
	linkcolor=aleacolor,
	urlcolor=aleacolor,
	citecolor=aleacolor}

\newcount\colveccount
\newcommand*\colvec[1]{
	\global\colveccount#1
	\begin{pmatrix}
		\colvecnext
	}
	\def\colvecnext#1{
		#1
		\global\advance\colveccount-1
		\ifnum\colveccount>0
		\\
		\expandafter\colvecnext
		\else
	\end{pmatrix}
	\fi
}

\newcommand{\ndN}{\mathbb{N}}
\newcommand{\ndZ}{\mathbb{Z}}

\newcommand{\ndR}{\mathbb{R}}
\newcommand{\ndC}{\mathbb{C}}
\newcommand{\ndT}{\mathbb{T}}
\newcommand{\ndF}{\mathbb{F}}

\renewcommand{\Pr}[1]{\mathbb{P}(#1)}

\newcommand{\Prb}[1]{\mathbb{P}\left(#1\right)}

\newcommand{\Ex}[1]{\mathbb{E}[#1]}

\newcommand{\Exb}[1]{\mathbb{E}\left[#1\right]}

\newcommand{\one}{{\mathbbm{1}}}

\newcommand{\convdis}{\,{\buildrel d \over \longrightarrow}\,}

\newcommand{\convp}{\,{\buildrel p \over \longrightarrow}\,}

\newcommand{\convas}{\,{\buildrel a.s. \over \longrightarrow}\,}

\newcommand{\eqdist}{\,{\buildrel d \over =}\,}

\newcommand{\He}{\mathrm{H}}

\newcommand{\cA}{\mathscr{A}}
\newcommand{\cB}{\mathscr{B}}
\newcommand{\cC}{\mathscr{C}}
\newcommand{\cD}{\mathscr{D}}
\newcommand{\cE}{\mathscr{E}}
\newcommand{\cF}{\mathscr{F}}
\newcommand{\cG}{\mathscr{G}}
\newcommand{\cH}{\mathscr{H}}

\newcommand{\cJ}{\mathscr{T}}

\newcommand{\cL}{\mathscr{L}}
\newcommand{\cM}{\mathscr{M}}

\newcommand{\cT}{\mathcal{T}}
\newcommand{\cU}{\mathscr{U}}

\newcommand{\V}{\mathrm{V}}
\newcommand{\F}{\mathrm{F}}
\newcommand{\E}{\mathrm{E}}

\newcommand{\Inn}{\mathrm{Inn}}

\newcommand{\mfB}{\mathfrak{B}}

\newtheorem{theorem}{Theorem}[section]

\newtheorem{corollary}[theorem]{Corollary}
\newtheorem{proposition}[theorem]{Proposition}
\newtheorem{lemma}[theorem]{Lemma}

\newtheorem{definition}[theorem]{Definition}
\newtheorem{remark}[theorem]{Remark}

\numberwithin{equation}{section}

\title{\textbf{First-passage percolation on random simple triangulations}}
\date{}

\author{Benedikt Stufler}

\address[Benedikt Stufler]{Vienna University of Technology}
\email{benedikt.stufler at tuwien.ac.at}

\begin{document}

\vspace {-0.5cm}

\begin{abstract}
We study first-passage percolation on random simple triangulations and their dual maps with independent identically distributed link weights. Our main result shows that the first-passage percolation distance concentrates in an $o_p(n^{1/4})$ window around a constant multiple of the graph distance. 
\end{abstract}


\maketitle

\section{Introduction}

The first-passage percolation distance $d_{\mathrm{fpp}}^G$ on a connected graph $G$ assigns an independent copy of a random link-weight $\iota>0$ to each edge of the graph. The weights on the edges are interpreted as lengths, and the first-passage percolation distance between any two points is given by the minimal sum of weights along joining paths. 

Our main result studies the first-passage percolation metric on random simple triangulations. That is, planar maps without multi-edge or loops, with all faces having degree $3$.

\begin{theorem}
	\label{te:main0}
	Let $\cT_n$ denote the uniform simple triangulation with $n+1$ vertices. Suppose that $\iota$ has finite exponential moments and that there exists a constant $\kappa>0$ such that $\Pr{\iota \ge \kappa} = 1$. Then there exists a constant~$c_{\mathrm{fpp}}^\cT>0$ such that
	\[
	n^{-1/4} \sup_{x,y \in \V(\cT_n)} \left| d_{\mathrm{fpp}}^{\cT_n}(x,y) - c_{\mathrm{fpp}}^\cT d_{\cT_n}(x,y) \right| \convp 0
	\]
	as $n$ tends to infinity.
\end{theorem}

Here $d_{\cT_n}(x,y)$ denotes the graph distance between points $x$ and $y$ of $\cT_n$. The constant $c_{\mathrm{fpp}}^\cT$ depends on the distribution of $\iota$, and determining an explicit expression appears to be a severely difficult problem. However, if the distribution of $\iota$ is explicit, simulations of random simple triangulations may be used to obtain  estimates.\footnote{The author's program \texttt{simtria} provides a highly optimized and efficient sampler for random simple triangulations. It supports multithreading and uses efficient memory management in order to facilitate  large scale simulations. The program is open source and available for free on the github: \url{https://github.com/BenediktStufler/simtria}.}

Similar results for bounded link-weights have been established in~\cite{zbMATH07144469} for  triangulations with loops and multi-edges, and in~\cite{lehericy2019firstpassage} for quadrangulations and unrestricted planar maps. The present work follows closely the proof strategy of~\cite{zbMATH07144469}, see below for more detailed comments on this. Concentration results of this sort are also known for members of the universality class of the Brownian tree. In particular,~\cite{zbMATH06729837} shows such a result for first-passage percolation distances on random outerplanar maps, and ~\cite[Thm. 7.1]{zbMATH06653517} for  random graph from subcritical graph classes like series-parallel-graphs and cacti graphs. Simple triangulations were also shown to admit the Brownian map as Gromov--Hausdorff--Prokhorov scaling limit~\cite{zbMATH06812193}, after rescaling distances by roughly $n^{-1/4}$.

We additionally prove a concentration result for first-passage percolation on random $3$-connected cubic planar graphs. Cubic planar graphs and related classes have received increasing attention in probabilistic and combinatorial literature. The asymptotic number of connected and $2$-connected cubic planar graphs was determined in~\cite{zbMATH05122852, zbMATH07213288} via singularity analysis methods and combinatorial decompositions. The shape of random cubic planar graphs was studied in~\cite{zbMATH06639396}, with a focus on the number of perfect matchings. The  work~\cite{stufler2022uniform} established a Uniform Infinite Cubic Planar Graph of as quenched local limit of random cubic planar graphs.

A  graph is called $3$-connected, if it is connected, has at least $4$ vertices,  and removing any pair of vertices does not disconnected the graph. In \cite{10.5565/PUBLMAT6612213}  the typical number of triangles in $3$-connected cubic planar graphs was determined. Further research directions include $4$-regular planar graphs~\cite{zbMATH06827273}, cubic graphs on general orientable surfaces~\cite{zbMATH06841874}, and cubic planar maps~\cite{coreprep}.

In particular, the work~\cite{zbMATH06556653} studies the asymptotic geometric shape of cubic planar maps by determining  the geodesic two- and three-point functions, after assigning independent random lengths with an exponential distribution to each edge. 

Our second main result is concerned with  first-passage percolation distances  on the model of $3$-connected cubic planar graphs:

\begin{theorem}
	\label{te:main}
	Let $\cC_n$ denote a uniformly selected $3$-connected cubic planar graph with $n$ labelled vertices. Suppose that $\iota$ has finite exponential moments and that there exists a constant $\kappa>0$ such that $\Pr{\iota \ge \kappa} = 1$. Then there exists a constant $c_{\mathrm{fpp}}^\cC>0$ such that
	\[
		n^{-1/4} \sup_{x,y \in \V(\cC(_n))} \left| d_{\mathrm{fpp}}^{\cC_n} (x,y) - c_{\mathrm{fpp}}^\cC d_{\cC_n}(x,y) \right| \convp 0
	\]
	as $n \in 2\ndN$ tends to infinity.
\end{theorem}

We emphasize that critical classes like cubic planar graphs or unrestricted planar graphs do not fall into the setting of subcritical graph classes considered in~\cite{zbMATH06653517}. The main results of the present work play a role in the proof of the main result in~\cite{stufler2022scaling} which establishes the Brownian map as scaling limit of random connected (and not necessarily $3$-connected) cubic planar graphs. Facilitating this application is also the reason why we go the extra mile and treat unbounded light-tailed link weights. See also the forthcoming independent proof~\cite{emt2022} of the scaling limit of cubic planar graphs with similar intermediate results on first-passage percolation distances.

The proof of Theorem~\ref{te:main} makes use of Whitney's theorem, which states that each $3$-connected planar graph has precisely two embeddings into the $2$-sphere. Hence, uniform $3$-connected cubic planar graphs are distributed like uniform $3$-connected cubic planar maps. The dual map construction yields a bijection between this class of planar maps and $3$-connected triangulations. We will show that the first-passage percolation distance on the dual of a random simple triangulation concentrates around a constant multiple of the graph distance in the triangulation. If the link-weights are constants equal to $1$, then the dual distance is equal to the graph distance on the associated $3$-connected cubic planar map. Hence applying this result twice  yields Theorem~\ref{te:main}.

Our proof strategy for simple triangulations follows closely the steps in~\cite{zbMATH07144469} that treats first-passage percolation  on  unrestricted triangulations (admitting loops and multi-edges) and their duals. The main challenge is that additional constraints arise in the skeleton decomposition of simple triangulations. We also have to do without the bijection between type I triangulations and type I triangulations of the $0$-gon (see~\cite[Fig. 2]{zbMATH07144469}), since simple triangulations may not contain loops. Some adaptions are also necessary to treat unbounded link-weights.  Moreover, the numbers of various classes of simple triangulation differ by at least exponential factors from the numbers of unrestricted triangulations. Thus, although it is natural to expect simple triangulations to behave in a similar way as unrestricted triangulations, we really need to go through the details.

\subsection*{Notation}
We  let $\mathbb{N} = \{1, 2, \ldots\}$ denote the collection of positive integers.  All unspecified limits are taken as $n$ becomes large, possibly taking only values in a subset of the natural numbers.  We use $\convdis$ and $\convp$ to denote convergence in distribution and probability. Almost sure convergence is denoted by $\convas$. Equality in distribution is denoted by $\eqdist$.  An event holds with high probability, if its probability tends to $1$ as $n \to \infty$.
For any integer $k \ge 0$ we let $[x^k]f(x) = a_k$ denote the $k$th coefficient of a power series $f(x) = \sum_{i \in \ndZ} a_i x^i$.

\section{Preliminaries}

\subsection{Planar maps}

A \emph{planar map} is an embedding of a connected multi-graph onto the $2$-dimensional sphere, so that edges of the multi-graph are represented by arcs that may only intersect at their endpoints.  The \emph{faces} of the planar map are the connected components created when removing it from the $2$-sphere.  All faces are required to be homeomorphic to  open discs.

Planar maps are viewed up to orientation-preserving homeomorphism. This way, the number of planar maps with a given number of edges is finite. In order to eliminate symmetries, we distinguish and orient a \emph{root edge}. All planar maps we consider in this work are rooted in this way. The origin of the root edge is called the \emph{root vertex}. A planar map is called \emph{simple}, if it has no multi-edges or loops. Using the stereographic projection, we may equivalently draw planar maps in the plane.

Any edge can be thought of as a pair of \emph{half-edges} with opposing directions. The half-edges delimiting a face are its \emph{boundary}. We say the boundary is \emph{simple}, if it is a simple graph-theoretic cycle. This way, the degree of the face is equal to the length of the cycle. 
The number of half-edges on the boundary of a face is its \emph{degree}. This way, any edge that has both sides incident to the same face is counted twice. 

\subsection{Triangulations}

A \emph{triangulation} is a planar map that only has faces of degree $3$. A triangulation without restrictions is called a \emph{type I} triangulation. If it has no loops, it is called a \emph{type II} triangulation. For triangulations this is equivalent to being non-separable. \emph{Simple triangulations} with no loops and no multi-edges are also called~\emph{type III} triangulations. For triangulations with at least four vertices, this is equivalent to being $3$-connected. The reader should beware that this terminology is not used uniformly throughout the literature. See for example the pioneering work by Tutte~\cite{zbMATH03169204} who additionally requires a simple triangulation to have no separating $3$-cycles.

A \emph{triangulation with a boundary} is a planar map $t$ where all faces have degree $3$, except for the face to the right of the root-edge, which is required to be simple. The \emph{boundary} $\partial t$ is boundary of this face and is also called the \emph{bottom cycle}. Letting $p \ge 1$ denote its length, we also say $t$ is a \emph{triangulation of the $p$-gon}. An \emph{inner vertex} of $t$ is a vertex that does not lie on the bottom cycle. The \emph{height} of a vertex $v \in \V(t)$ of $t$ is the graph distance between $v$ and the bottom cycle. That is, the minimal length of a path that originates in $v$ and ends at any of the vertices of the bottom cycle.

\subsection{The dual map construction}

The \emph{dual map} $M^\dagger$ of a planar map $M$ is the  ``red'' planar map constructed by placing a red vertex inside each face of $M$ and then adding for each edge $e$ of $M$ a red edge between the red vertices inside the two faces adjacent to $e$. These faces may be identical, and in this case the corresponding red edge is a loop. The root edge of the dual map $M^\dagger$ is the red edge corresponding to the root-edge of $M$, oriented in a canonical way.
For any vertex $v$ and face $f$ of $M$ we write $v \triangleleft f$ if $v$ is adjacent to $f$.

Let $\iota>0$ denote a positive random variable. The $\iota$-first-passage percolation metric $d_{\mathrm{fpp}}$  on the vertex set $\V(M)$ of a planar map $M$ assigns a weight to each edge according to an independent copy of $\iota$. The distance between any two points is then given by the minimal sum of weights along joining paths. We let $d^\dagger_{\mathrm{fpp}}$ denote the $\iota$-first-passage percolation distance on the dual map $M^\dagger$.

\section{$3$-connected cubic planar graphs and $3$-connected triangulations}

Restricting the dual map construction to $3$-connected triangulations yields a bijection between this class and $3$-connected cubic planar maps. Furthermore, Whitney's theorem ensures that each $3$-connected cubic planar map has precisely two embeddings into the $2$-sphere. Hence, if $\cT_n$ denotes a uniform triangulation with $n+1$ vertices (and consequently $3(n-1)$ edges and $2(n-1)$ faces) then the dual map $\cT_n^\dagger$ is distributed like the uniform random $3$-connected planar graph $\cC_{2(n-1)}$. 

In order to prove Theorem~\ref{te:main} it suffices to show:

\begin{theorem}
	\label{te:main1}
	 Suppose that $\iota$ has finite exponential moments and that there exists a constant $\kappa>0$ with $\Pr{\iota \ge \kappa} = 1$. Then there exists a constant $c_{\mathrm{fpp}}^\dagger>0$ such that
\begin{align*}
	n^{-1/4} \sup_{\substack{u,v \in \cT_n\\ u\triangleleft f, v\triangleleft g }} \left | c_{\mathrm{fpp}}^\dagger d_{\cT_n}(u,v) -  d_{\mathrm{fpp}}^\dagger(f,g)\right | \convp 0.
\end{align*}
	as $n$ tends to infinity.
\end{theorem}

Indeed, Theorem~\ref{te:main} follows by applying Theorem~\ref{te:main1} twice, once for $\iota$ and once for constant link-weights.

\section{The skeleton decomposition of simple triangulations}

\subsection{Enumerative properties}

As shown by~\cite{zbMATH03169204}, the set $\ndT_{n}$ of triangulations with $n+3$ vertices and hence $2(n+1)$ faces and $3(n+1)$ edges has cardinality
\begin{align}
	\label{eq:tn0}
	\#\ndT_n = \frac{2(4n+1)!}{(n+1)!(3n+2)!}.
\end{align}
For any integer $p \ge 3$, we let $\ndT_{n,p}$ denote the set of all simple (type III) triangulations of the $p$-gon with $n \ge 0$ inner vertices, and hence $3n + 2p - 3$ edges. In particular, $\ndT_n = \ndT_{n,3}$. As determined by~\cite{zbMATH03217678}, 
\begin{align}
	\label{eq:tnp}
	\#\ndT_{n,p} = \frac{2 (2p-3)! (4n + 2p-5)!}{(p-3)!(p-1)!n!(3n+2p - 3)!}.
\end{align}
In \cite[Sec. 2.1]{MR2013797} the following asymptotics were determined. As $n \to \infty$,
\begin{align}
	\label{eq:tnpasymp}
\#\ndT_{n,p} \sim C(p) n^{-5/2} (27/256)^{-n}
\end{align}
with
\begin{align}
	\label{eq:cpasymp}
	C(p) = \frac{(2p-3)!}{3 \sqrt{6\pi} (p-3)!(p-1)!} (16/9)^{p-2}.
\end{align}
As $p \to \infty$,
\begin{align}
	\label{eq:cpasymptotoo}
	C(p) \sim \frac{9 \sqrt{3}}{2048 \pi \sqrt{2}} (9/64)^{-p} \sqrt{p}.
\end{align}
For $p \ge 3$ we set
\begin{align}
	\label{eq:zpdef}
	Z(p) = \sum_{n \ge 0} (27/256)^n \#\ndT_{n,p}.
\end{align}
The sum converges by Equation~\eqref{eq:tnpasymp}. Its exact value was determined in~\cite[Prop. 2.4]{MR2013797}
\begin{align}
	\label{eq:zpexpr}
	Z(p) =  \frac{81(2p-4)!}{128(p-2)!p!} (16/9)^p.
\end{align}
We now define the simple Boltzmann triangulation with a given perimeter:
\begin{definition}[Boltzmann triangulations with a given perimeter]
	\label{def:boltzmann}
	For any integer $p \ge 3$ the \emph{simple Boltzmann triangulation with perimeter $p$} is a random element of $\bigcup_{n \ge 0} \ndT_{n,p}$ that, for all $n \ge 0$, assumes a triangulation $T \in \ndT_{n,p}$ with probability $(27/256)^n / Z(p)$.
\end{definition}

In other words, the simple Boltzmann triangulation has a random number of internal vertices determined by the probability generating function $Z(p)^{-1} \sum_{n \ge 0} (27/256)^n \#\ndT_{n,p} x^n$, and conditional on having size $n$ it is uniformly distribution over $\#\ndT_{n,p}$.

It will be notationally convenient to additionally set
\begin{align}
	Z(2) = 1.
\end{align}
By Equation~\eqref{eq:zpexpr} and standard arguments it follows that
\begin{align}
	\label{eq:zseries}
	\sum_{p \ge 2} Z(p)x^p = \frac{1}{512} \left((9-64 x)^{3/2}+288 x-27\right).
\end{align}
The well-known asymptotic \[
\binom{2n}{n} \sim \frac{4^n}{\sqrt{\pi n}}
\] for the central binomial coefficient implies
\begin{align}
	\label{eq:zpasymp}
	Z(p) \sim \frac{2}{\sqrt{\pi}} p^{-5/2} (64/9)^{p-2} 
\end{align}
as $p \to \infty$. 

The following observation is a version of~\cite[Lem. 1]{zbMATH07144469} for simple triangulations.
\begin{lemma}
	\label{le:eet}
	There exists a constant $c>0$ such that for all $n \ge 1$ and $p \ge 3$
	\[
		{\# \ndT_{n,p}} \le { c C(p) n^{-5/2} (256/27)^{n}}.
	\]
	Furthermore, for each $\alpha>0$ there is a constant $c(\alpha)>0$ with
	\[
			{\# \ndT_{n,p}} \ge { c(\alpha) C(p) n^{-5/2} (256/27)^{n}}.
	\]
	uniformly for all $n \ge 1$ and $p \le \alpha \sqrt{n}$.
\end{lemma}
\begin{proof}
	We use $\Theta(1)$ to denote a term that is bounded away from zero and infinity uniformly in $n$ and $p$. Likewise, we let $O(1)$ denote a term that is bounded from above uniformly in $n$ and $p$. By Equations~\eqref{eq:tnp}, \eqref{eq:cpasymp}
	\begin{align*}
		\frac{\# \ndT_{n,p}}{ C(p) n^{-5/2} (256/27)^{n}} &= \frac{\Theta(1)}{n^{-5/2} (16/9)^p (256/27)^{n} } \frac{(4n + 2p-5)!}{n!(3n+2p-3)!}.
	\end{align*}
	Stirling's formula yields
	\begin{align*}
		\frac{(4n + 2p-5)!}{n!(3n+2p-3)!} &= \Theta(1) \left( \frac{n + p/2}{n(n + 2p/3)} \right)^{1/2} \frac{ (4n + 2p - 5)^{4n + 2p-5}}{n^n (3n+2p-3)^{3n+2p-3}}. 
	\end{align*}
We may write
\begin{align*}
	\frac{ (4n + 2p - 5)^{4n + 2p-5}}{n^n (3n+2p-3)^{3n+2p-3}} = \frac{ (4n)^{4n+2p-5}}{n^n (3n)^{3n+2p-3}} \frac{(1 + \frac{p}{2n} - \frac{5}{4n})^{4n + 2p-5}}{(1 + \frac{2p}{3n} -\frac{1}{n})^{3n + 2p-3}}
\end{align*}
and
\begin{align*}
	\frac{ (4n)^{4n+2p-5}}{n^n (3n)^{3n+2p-3}} = \Theta(1/n^2) (256/27)^n (16/9)^p.
\end{align*}
Combining these equations yields
\begin{align*}
	\frac{\# \ndT_{n,p}}{ C(p) n^{-5/2} (256/27)^{n}} &= \Theta(1) \left( \frac{1 + \frac{p}{2n}}{1 + \frac{2p}{3n}} \right)^{1/2}  \frac{(1 + \frac{p}{2n} - \frac{5}{4n})^{4n + 2p-5}}{(1 + \frac{2p}{3n} -\frac{1}{n})^{3n + 2p-3}}
\end{align*}
The function $x \mapsto (1+x) / (1 + 2x/3)$ is bounded away from zero and infinity on the interval $[0, \infty[$. Hence 
\begin{align*}
	\left( \frac{1 + \frac{p}{2n}}{1 + \frac{2p}{3n}} \right)^{1/2} = \Theta(1)
\end{align*}
and
\begin{align}
	\frac{\# \ndT_{n,p}}{ C(p) n^{-5/2} (256/27)^{n}} &= \Theta(1)   \frac{(1 + \frac{p}{2n} - \frac{5}{4n})^{4n + 2p-5}}{(1 + \frac{2p}{3n} -\frac{1}{n})^{3n + 2p-3}}.
\end{align}
It is elementary that uniformly for all $n \ge 1$ and $3 \le p \le \alpha \sqrt{n}$ this expression is bounded from below by a constant that only depends on $\alpha$. This verifies the claimed lower bound. In order to show the upper bound, note that for all $n \ge 1$ and $p \ge 3$

\begin{align*}
	\frac{(1 + \frac{p}{2n} - \frac{5}{4n})^{4n + 2p-5}}{(1 + \frac{2p}{3n} -\frac{1}{n})^{3n + 2p-3}} &= \left( \frac{(1 + \frac{2p -3}{4n} - \frac{2}{4n})^{4 + \frac{2p-3}{n} - \frac{2}{n}}}{(1 + \frac{2p-3}{3n} )^{3 + \frac{2p-3}{n}}} \right)^n \\
	&\le  \left( \frac{(1 + \frac{2p -3}{4n} )^{4 + \frac{2p-3}{n} }}{(1 + \frac{2p-3}{3n} )^{3 + \frac{2p-3}{n}}} \right)^n.
\end{align*}
Consider the function
\[
	f(x) =  \frac{(1 + x/4 )^{4 + x }}{(1 + x/3 )^{3 + x}} .
\]
Note that $f(0) = 1$ and 
\[
	\frac{\mathrm{d}}{\mathrm{d}x} \log(f(x)) = \log \left( \frac{3(4+x)}{4(3+x)} \right) \le 0.
\]
Hence $f(x) \le 1$ for $x \ge 0$. Consequently,
\begin{align}
	\frac{\# \ndT_{n,p}}{ C(p) n^{-5/2} (256/27)^{n}} &= O(1) f \left( \frac{2p-3}{4n} \right)^n = O(1).
\end{align}

\end{proof}

\subsection{The skeleton decomposition of simple triangulations of the cylinder}

\label{sec:app}

We adapt Krikun's skeleton decomposition of type II triangulations~\cite{zbMATH02213763, 2005math.....12304K} to type III triangulations.  In the parts that do not depend on the type we follow closely the terminology from~\cite[Sec. 2.2]{zbMATH07144469}, where an adaption to type I triangulations was made.

\begin{definition}[Simple triangulation of the cylinder]
A \emph{simple triangulation of the cylinder} $\Delta$ of height $r \ge 1$ is a rooted simple planar map such that all faces have degree $3$ except for a so-called \emph{bottom face} and a \emph{top face}. We require the following conditions to be satisfied:
\begin{enumerate}
	\item The boundaries of these two distinguished faces are  simple graph-theoretic cycles, and hence have length at least $3$.
	\item The root edge lies on the boundary of the bottom face, so that the bottom face lies to its right.
	\item Every vertex on the boundary of the top face is at graph distance exactly $r$ of the boundary of the bottom face. 
	\item Every edge on the top face is incident to a face of degree $3$ with the third vertex at distance $r-1$ from the boundary of the bottom face.
\end{enumerate}   
\end{definition}

We let $\partial \Delta$ denote the boundary of the bottom face, and $\partial^*\Delta$ the boundary of the top face. Let $p \ge 3$ and $q \ge 3$ denote their lengths.  For $1 \le j \le r$ let $B_j(\Delta)$ denote the union of all faces of $\Delta$ whose boundary contains a vertex at graph distance strictly less than $j$ from the bottom cycle. The \emph{hull} $B_j^\bullet(\Delta)$ is the union of $B_j(\Delta)$ and all connected components of its complement, except the one containing the top face. This way, $B_j^\bullet(\Delta)$ is a triangulation of the cylinder of height $j$. We let $\partial_j \Delta = \partial^* B_j^\bullet(\Delta)$ denote its top cycle. We also set $\partial_0 \Delta =  \partial \Delta$ and $\partial_r \Delta = \partial^* \Delta$. For ease of reference, we always assume that $\Delta$ is drawn in the plane such that the unbounded face is the top face. This way, we may orient all cycles $\partial_j \Delta$, $0 \le j \le r$, in clock-wise direction.

\begin{figure}[t]
	\centering
	\begin{minipage}{\textwidth}
		\centering
		\includegraphics[width=1.0\linewidth]{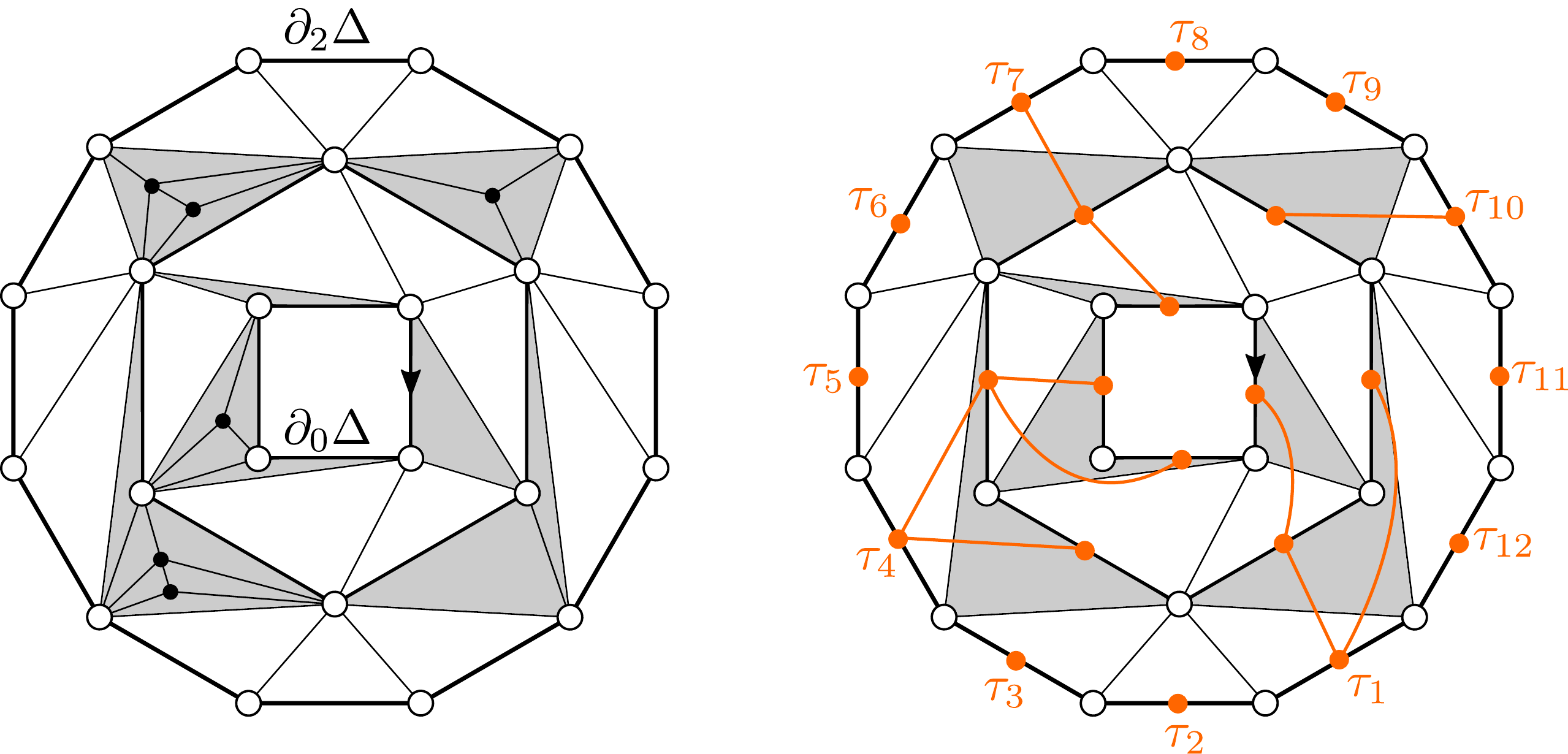}
		\caption{Skeleton decomposition of a simple triangulation of the cylinder. White faces represent downward triangles.}
		\label{fi:skeleton}
	\end{minipage}
\end{figure}

For $1 \le j \le r$, each edge of $\partial_j \Delta$ is incident to exactly one face whose third vertex lies on $\partial_{j-1} \Delta$. These faces are called the \emph{downward triangles} at height $j$. We let $\E_{\mathrm{d}}(\Delta)$ denote the collection of all edges of the cycles $\partial_0 \Delta, \ldots, \partial_r \Delta$. For $1 \le j \le r$ we may consider the vertices on $\partial_{j-1} \Delta$ corresponding to downward triangles at height $j$. This way, for each edge $e'$ of $\partial_{j-1} \Delta$ we may walk along $\partial_{j-1} \Delta$ in  counter-clockwise starting from the  ``middle'' of $e'$ until we encounter for the first time a vertex corresponding to a downward triangle at height $j$. We say the unique edge $e$ of $\partial_j \Delta$ is the parent of $e'$. The parent relation yields a forest structure on $\E_{\mathrm{d}}(\Delta)$, with the $q$ roots of the forest corresponding to the edges on $\partial_r \Delta$. We let $\tau_1, \ldots, \tau_q$ denote the trees of the forest ordered according to the clock-wise orientation of $\partial_r \Delta$, starting with the unique tree that contains the root-edge of $\Delta$. This way, $\tau_1$ has height $r$ and all other trees have height $\le r$. Furthermore, $\tau_1$ has a marked leaf at height $r$, corresponding to the root edge of $\Delta$. See Figure~\ref{fi:skeleton} for an illustration.

Note  that since $\Delta$ is simple, for all $1 \le j \le r$ it may never happen that only a single edge on $\partial_j \Delta$ has offspring and the others don't, because this would entail the presence of a multi-edge in $\Delta$. See Figure~\ref{fi:exception} for an example. Since the tree $\tau_1$ always has height $r$, this is equivalent to stating that $\tau_1$ has no vertex of height less than $r$ who is parent to all vertices of the next generation.

The forest $(\tau_1,\ldots, \tau_q)$ encodes the configuration of downwards triangles. Informally speaking, if we remove the downwards triangles from $\Delta$, what remains are the bottom cycle, the top cycle, and a collection of~\emph{slots}. Specifically, given an edge $e$ of $\partial_j \Delta$ for $1 \le j \le r$ with $c_e \ge 1$ children, the associated slot is bounded by the children of $e$ (who belong to $\partial_{j-1} \Delta$ and two \emph{vertical} edges between $\partial_j \Delta$ and $\partial_{j-1} \Delta$. The slot has a canonical root-edge, given by the vertical edge on its boundary that is incident to the downward triangle of $e$. We orient this root edge so that it points from $\partial_{j-1} \Delta$ to $\partial_j \Delta$. Thus, the slot corresponding to $e$ is  a simple triangulation of the $(c_e + 2)$-gon. Note that the boundary of the slot is always simple, since $\Delta$ is simple. It will be notationally convenient to also assign a ``slot'' to $e$ in the case $c_e = 0$, given by the edge of $e$'s downward triangle that is incident to the origin of~$e$, when orienting $e$ in the same direction as $\partial_j \Delta$.

\begin{figure}[t]
	\centering
	\begin{minipage}{\textwidth}
		\centering
		\includegraphics[width=0.25\linewidth]{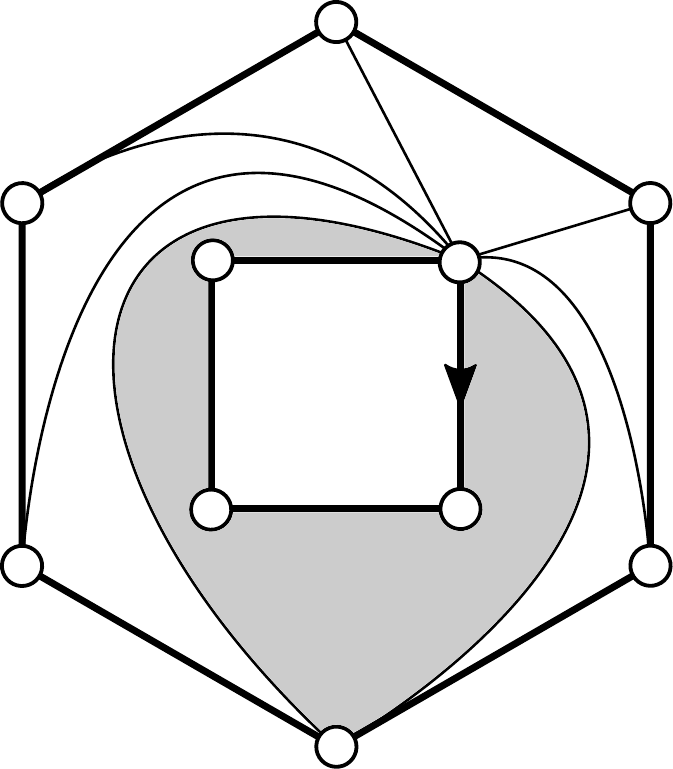}
		\caption{An example of a forbidden configuration of downward triangles that creates a double edge.}
		\label{fi:exception}
	\end{minipage}
\end{figure}

\begin{definition}[Admissible forests]
	\label{eq:defadmissible}
	We say a forest $\cF$ with a distinguished vertex to be $(p,q,r)$-admissible for integers $p,q \ge 3$ and $r \ge 1$ if the following conditions are met:
	\begin{enumerate}
		\item The forest is an ordered sequence $(\tau_1, \ldots, \tau_q)$ of planted plane trees.
		\item Each generation of the forest has at least $3$ vertices.
		\item All trees have height at most $r$.
		\item The forest has exactly $p$ vertices with height $r$.
		\item The distinguished vertex is a leaf of $\tau_1$ with height $r$.
		\item For all $1 \le j < r$, no vertex of $\tau_1$ with height $j$ is parent to all vertices of $\cF$ with height $j+1$.
	\end{enumerate}
We write $\cF^*$ for the set of vertices of $\cF$ with height less than $r$.
\end{definition}

For any vertex $v \in \cF$ we let $c_v$ denote the number of children of $v$. Thus, there is a bijection between simple triangulations of the cylinder of height $r$ with bottom and top cycle having lengths $p$ and $q$, and pairs $(\cF, (M_v)_{v \in \cF^*})$ of a $(p,q,r)$-admissible forest $\cF$ and a family $(M_v)_{v \in \cF^*}$ such that for each $v \in \cF^*$ we have that  $M_v$ is a simple triangulation of the $(c_v+2)$-gon if $c_v >0$, and a place-holder value if $c_v = 0$. Including the place-holder values for leaves with height less than $r$ will be notationally convenient.

\begin{remark}
For type II triangulations, vertices as described in the fifth condition in Definition~\ref{eq:defadmissible} are permitted, but the corresponding slot is subject to constraints in order to avoid loops~\cite[Lem. 2]{zbMATH02213763}. As argued above, for type III  triangulations we have no constraints on the decorations, since any such vertex would entail the presence of multi-edges and hence cannot exist in the forest associated to a type III triangulation of the cylinder.
\end{remark}

\begin{definition}[Left-most geodesics]
	Let $v$ be a vertex of $\partial_j \Delta$ for some $1 \le j \le r$. The \emph{left-most geodesic} from $v$ to the bottom cycle $\partial \Delta$ is defined as follows. Consider the clock-wise ordering of the edges incident to $x$, starting from the unique edge on $\partial_j \Delta$ that starts at $x$ and points in the same direction as the clock-wise orientation of $\partial_j \Delta$. The first edge of the left-most geodesic is the last edge in this ordering that ends at a vertex from $\partial_{j-1} \Delta$. From there we proceed inductively.
\end{definition}

Let $u,v$ be vertices on the top cycle $\partial^* \Delta$. Let $\cF'$ denote the subforest of the skeleton $\cF$ of $\Delta$ consisting of all trees whose root corresponds to an edge of $\partial^* \Delta$ that lies on the path from $u$ to $v$ in the clockwise direction of $\partial^* \Delta$. Let $\cF''$ denote the forest obtained by removing $\cF'$ from $\cF$. It is easy to see that for each integer $1 \le k \le r$ the left-most geodesics starting from $u$ and $v$ merge before or at step $k$ if and only if $\cF'$ or $\cF''$ have height less than $k$.

\subsection{Skeletons of random triangulations}

For all integers $j \ge 1$, we define the \emph{ball} $B_j(t)$ of a simple triangulation $t$ of the $p$-gon  as the submap consisting of all faces of $t$ that are incident to a vertex at distance strictly less than $j$ from the bottom cycle $\partial t$.  Hence any vertex $v$ of $B_j(t)$ has distance at most $j$ from $\partial t$. 

We say a pair $\bar{t} = (t, o)$ of a triangulation $t$ of the $p$-gon and a vertex $o$ of $t$ is a \emph{pointed triangulation of the $p$-gon}. Suppose that $o$ has height (that is, distance from $\partial t$) strictly larger than $j$. We construct the \emph{hull} $B_j^\bullet(\bar{t})$ by adding to $B_j(t)$ all connected components of the complement of $B_j(t)$, except for the component containing the marked vertex $o$. Since we assume the marked vertex to have height strictly larger than $j$, the hull $B_j^\bullet(\bar{t})$ is a triangulation of the cylinder of height $j$. Note that $B_j^\bullet(B_r^\bullet(\bar{t})) = B_j^\bullet(\bar{t})$ for $1 \le j < r$.

For any integer $p \ge 3$ we let $\cT_n^{(p)}$ be drawn uniformly at random from the set $\ndT_{n,p}$ of simple triangulations of the $p$-gon with $n$ inner vertices. We let $\overline{\cT}_n^{(p)}$ denote the pointed simple triangulation obtained by distinguishing a uniformly selected inner vertex of $\cT_n^{(p)}$. The hull $B_r^\bullet(\overline{\cT}_n^{(p)})$ is well-defined if the marked vertex has height at least $r+1$. If this condition is not met, we set $B_r^\bullet(\overline{\cT}_n^{(p)})$ to some place-holder value.
	
Let $\Delta$ be a triangulation of the cylinder with height $r \ge 1$, with $p, q \ge 3$ denoting the lengths of the bottom and top cycle. Let $N$ denote the total number of vertices of $\Delta$ and let $\cF = (\tau_1, \ldots, \tau_q)$ denote the $(p,q,r)$-admissible forest associated to $\Delta$. Let $(M_v)_{v \in \cF^*}$ denote the simple triangulations with a boundary filling in the slots of $\Delta$. For each $v \in \cF^*$ we let $\Inn(M_v)$ denote the number of inner vertices of $M_v$. In case $c_v=0$, we set $\Inn(M_v) = 0$.

The following lemma is analogous to~\cite[Lem. 2]{zbMATH07144469}.
\begin{lemma}
	\label{le:brtntria}
	We have
	\[
		\lim_{n \to \infty} \Prb{ B_r^\bullet(\overline{\cT}_n^{(p)}) = \Delta } = \frac{(64/9)^{-q} C(q)}{(64/9)^{-p}C(p)}   \prod_{v \in \cF^*} \theta(c_v) \frac{(256/27)^{-\Inn(M_v) }}{Z(c_v +2)}
	\]
	with $c_v$ denoting the outdegree of the vertex $v$ in the forest $\cF$, and $(\theta(k))_{k \ge 0}$ the probability weights of a probability distribution with mean
	\[
		\sum_{k \ge 1} k \theta(k) = 1
	\]
	and probability generating function
	\[
	g_\theta(x) = \sum_{k =0}^{\infty}x^k \theta(k) =  1 - \left(1 + \frac{1}{1-x} \right)^{-2}.
	\]
	The probability weights satisfy the asymptotic
	\[
		\theta(k) \sim \frac{3}{2 \sqrt{\pi}} k^{-5/2}
	\]
	as $k \to \infty$.
\end{lemma}
\begin{proof}
	For ease of notation, we set
	\[
		\alpha = 64/9
	\]
	and
	\[
		\rho= 256/27.
	\]
	The event $B_r^\bullet(\overline{\cT}_n^{(p)}) = \Delta $ means that $\cT_n^{(p)}$ is equal to result of inserting an arbitrary triangulation of the $q$-gon in the outer face (that is, the face whose boundary is the top cycle) of $\Delta$, and that the uniformly marked vertex of $\cT_n^{(p)}$ got drawn from the inner vertices of this $q$-gon. Thus, for sufficiently large $n$
	\begin{align}
		\label{eq:rapdasarmas}
		\Prb{ B_r^\bullet(\overline{\cT}_n^{(p)}) = \Delta } = \frac{\#\ndT_{n-(N-p),q}}{\#\ndT_{n,p}} \left( 1 - \frac{N-p}{n} \right).
	\end{align}
	By~\eqref{eq:tnpasymp}, it follows that
	\begin{align}
		\lim_{n \to \infty} \Prb{ B_r^\bullet(\overline{\cT}_n^{(p)}) = \Delta }  = \frac{C(q)}{C(p)}  \rho^{p - N}.
	\end{align}
	The number $N$ of vertices of $\Delta$ satisfies
	\[
		N = \#\E_{\mathrm{d}}(\Delta) + \sum_{v \in \cF^*} \Inn(M_v),
	\]
	with $\E_{\mathrm{d}}(\Delta)$ denoting the combined collection of edges  of the cycles $\partial_j \Delta$ for $0 \le j \le r$. Moreover, 
	\[
		\#\E_{\mathrm{d}}(\Delta) = \sum_{i=1}^q \# \tau_i = q + \sum_{v \in \cF^*} c_v.
	\]
	Thus,
	\begin{align}
		\label{eq:xzibit}
				\lim_{n \to \infty} \Prb{ B_r^\bullet(\overline{\cT}_n^{(p)}) = \Delta }  = \frac{C(q)}{C(p)}  \rho^{p - q} \prod_{v \in \cF^*} \rho^{-\Inn(M_v) - c_v}.
	\end{align}
	Furthermore,
	\[
		\sum_{v \in \cF^*} (c_v -1) = p - q.
	\]
	Hence, multiplying~\eqref{eq:xzibit} with $1 = (\alpha / \rho)^{p-q- \sum_{v \in \cF^*} (c_v-1)}$ yields 
	\begin{align*}
	\lim_{n \to \infty} \Prb{ B_r^\bullet(\overline{\cT}_n^{(p)}) = \Delta }  &= \frac{\alpha^{-q} C(q)}{\alpha^{-p}C(p)}   \prod_{v \in \cF^*} \alpha^{-c_v +1}\rho^{-\Inn(M_v) -1 } \\
	&= \frac{\alpha^{-q} C(q)}{\alpha^{-p}C(p)}   \prod_{v \in \cF^*} \theta(c_v) \frac{\rho^{-\Inn(M_v) }}{Z(c_v +2)},
	\end{align*}
where we set for all $k \ge 0$
\begin{align}
	\label{eq:deftheta}
	\theta(k) = \frac{1}{\rho} \alpha^{-k+1} Z(k+2).
\end{align}
Using Equation~\eqref{eq:zpexpr}, it follows by elementary calculations  that
\begin{align}
	g_\theta(x) = \sum_{k =0}^{\infty}x^k \theta(k) =  1 - \left(1 + \frac{1}{1-x} \right)^{-2}.
\end{align}
Furthermore,
\[
	g_\theta(1)  = 1
\]
and
\[
g_\theta'(1) = 1. 
\]
Hence $g_\theta$ is the density function of a probability distribution with mean $1$.
\end{proof}

\begin{remark}
	The distribution $\theta$ for the type III case considered here turns out to be identical to the offspring distribution in the type II case~\cite{zbMATH02213763} and the type I case~\cite{zbMATH07144469}. However, we are faced with different constraints on the trees.
\end{remark}

Recall that $\cT_n$ denotes a uniform triangulation with $n+1$ vertices. In the work~\cite{MR2013797} the type III Uniform Infinite Planar Triangulation (UIPT) was constructed and shown to be the local limit of large uniform type III triangulations:
\begin{proposition}[\cite{MR2013797}]
	\label{pro:convvvv}
In the local topology,
\begin{align}
	\label{eq:convtinfty}
	\cT_n \convdis \cT_\infty
\end{align}
as $n\to \infty$. 
\end{proposition}
 Furthermore,~\cite{MR2013797} showed that the UIPT is almost surely~\emph{one-ended}, that is, deleting any finite number of vertices of $\cT_\infty$ may leave us with several connected components, but exactly one of them is infinite.

The degree $d(\cT_\infty)$ of the origin of the root-edge of $\cT_\infty$ was determined in~\cite[Lem. 4.1]{MR2013797} to satisfy
\begin{align}
	\label{eq:infp}
	\Pr{d(\cT_\infty)= p} =  \frac{2(2p-3)!}{(p-3)!(p-1)!} \left( \frac{3}{16} \right)^{p-1}
\end{align}
for each $p \ge 3$. Using the asymptotics for the central binomial coefficient, it follows that
\begin{align}
	\label{eq:infpasymptotic}
	\Pr{d(\cT_\infty)= p}  \sim  \frac{1}{\sqrt{2 \pi}} \sqrt{p} \left( \frac{3}{4} \right)^{p-1}
\end{align}
as $p \to\infty$. 

If we condition $\cT_n$ on having root degree $p$, and then delete the origin of the root-edge, we obtain a uniform simple triangulation of the $p$-gon with $n$ vertices in total. The bottom cycle is oriented in a clock-wise way. The root-edge is chosen canonically to be the one that originates from the destination of the original root-edge of $\cT_n$ and points in the direction of the bottom cycle. 

Equation~\eqref{eq:infp} ensures that root of the UIPT assumes any integer $p \ge 3$ with positive probability. Hence the same construction is possible for the UIPT:

\begin{definition}[Type III Uniform Infinite Triangulation of the $p$-gon]
	 For each integer $p \ge 3$ we let $\cT_\infty^{(p)}$ denote the result of conditioning the root vertex of the type III UIPT $\cT_\infty$ on having degree $p$, and then deleting it.
\end{definition}
The root-edge of $\cT_\infty^{(p)}$ is chosen canonically like in the case for finite triangulations. It follows directly from Equation~\eqref{eq:convtinfty} that the map  $\cT_\infty^{(p)}$ is the local limit of the uniform triangulation of the $p$-gon:  
\begin{corollary}
	\label{co:cotnptinfty}
	We have
	\[
		\cT_n^{(p)} \convdis \cT_\infty^{(p)}
	\]
	in the local topology as $n \to \infty$. 
\end{corollary}

Since the UIPT is almost surely one-ended, it immediately follows that $\cT_\infty^{(p)}$ is almost surely one-ended as well.

All neighbourhoods of the root-edge of $\cT_\infty^{(p)}$ are finite, hence all neighbourhood of the bottom cycle of $\cT_\infty^{(p)}$ are finite as well. It follows that all connected components of the complement of $\cT_\infty^{(p)}$ have a finite boundary. Since  $\cT_\infty^{(p)}$ is almost surely one-ended, it follows that exactly one of them is infinite. 

This allows us to construct for all integers $r \ge 1$ the hull $B_r^\bullet(\cT_\infty^{(p)})$ by adding to $B_r(\cT_\infty^{(p)})$ all connected components of its complement, except the unique infinite component. By construction, the hull $B_r^\bullet(\cT_\infty^{(p)})$ is finite.

By Lemma~\ref{le:brtntria}, it follows that
\begin{align}
 \Prb{ B_r^\bullet(\cT_\infty^{(p)}) = \Delta } = \frac{(64/9)^{-q} C(q)}{(64/9)^{-p}C(p)}   \prod_{v \in \cF^*} \theta(c_v) \frac{(256/27)^{-\Inn(M_v) }}{Z(c_v +2)}.
\end{align}
Summing over the countably many possible configurations~$(M_v)_{v \in \cF^*}$ we obtain:

\begin{corollary}
	\label{co:formula}
	For all $\cF \in \ndF_{p,q,r}$ we have
	\begin{align*}
		\Prb{ \cF \text{ is associated to }B_r^\bullet(\cT_\infty^{(p)}) } = \frac{(64/9)^{-q} C(q)}{(64/9)^{-p}C(p)}   \prod_{v \in \cF^*} \theta(c_v).
	\end{align*}
\end{corollary}
We let $\mathbb{P}_{p,r}$ denote the law of $B_r^\bullet(\cT_\infty^{(p)})$ on the collection $\ndC_{p,r}$ of all triangulations of the cylinder of height $r$ with a bottom cycle of length $p$. We let $\mathbf{P}_{p,r}$ denote the law of the associated forest from the collection 
\[
	\ndF_{p,r} := \bigcup_{q=1}^\infty \ndF_{p,q,r}.
\]
By construction, $\mathbf{P}_{p,r}$ is a probability measure, and hence
\begin{align}
	\sum_{q=3}^\infty \sum_{\cF \in \mathbb{F}_{p,q,r}} \frac{(64/9)^{-q}C(q)}{(64/9)^{-p}C(p)} \prod_{v \in \cF^*} \theta(c_v) = 1.
\end{align}
Compare with~\cite[Lem. 3]{zbMATH07144469}, where an analogous equation was verified for type I triangulation via a direct calculation instead. 

\subsection{A comparison principle}
\label{sec:comparision}

Let $1 \le r < s$ and $p \ge 3$ be integers, and let  $\Delta \in \mathbb{C}_{p,s}$.  Let $q$ denote the length of the cycle $\partial_r \Delta$. We may view $\Delta$ as the result of gluing a triangulation $\Delta'' \in \ndC_{q, s-r}$ at the top cycle of a triangulation $\Delta \in \ndC_{p, r}$. From Corollary~\ref{co:formula} it follows that 
\[
	\mathbb{P}_{p,s}(\Delta) = \mathbb{P}_{p,r}(\Delta') \mathbb{P}_{q,s-r}(\Delta'').
\]
We let $L_r^{(p)}$ denote the length of the top cycle of $B_r^\bullet(\cT_\infty^{(p)})$.

If we  delete the origin $o$ of the root-edge of $\cT_\infty$, we obtain a  simple triangulation of the polygon whose length equals the degree of $o$ in $\cT_\infty$. We described this degree  $d(\cT_\infty)$ in Equation~\eqref{eq:infp}.  The bottom cycle is oriented in a clock-wise way. The new root-edge of the bottom-cycle  is chosen canonically to be the one that originates from the destination of the original root-edge of $\cT_\infty$ and points in the direction of the bottom cycle. 

It will be notationally convenient to set $\cT_\infty^{(0)} = \cT_\infty$ and interpret $B_{r}^\bullet(\cT_\infty^{(0)})$ as a  ``triangulation of the cylinder with a bottom cycle of length~$0$.'' Its associated forest is defined to be the one corresponding to $B_{r}^\bullet(\cT_\infty^{(0)}) \setminus \{o\}$, with the ordering of the trees determined by the root-edge on the bottom cycle. We let $L_r$ denote the length of its top cycle. Hence  $L_r$ is distributed like the length of the top cycle of $B_{r-1}^\bullet( \cT_\infty^{(D)})$ for a random independent integer $D$ that is distributed like the root degree $d(\cT_\infty)$ described in Equation~\eqref{eq:infp}. Here, we set $L_1 = D$ to cover the case $r=1$.

By a slight abuse of notation, we let $B_s^\bullet(\cT_\infty^{(p)}) \setminus B_r^\bullet(\cT_\infty^{(p)})$  denote the triangulation of the cylinder of height $s-r$ that contains all faces from $B_s^\bullet(\cT_\infty^{(p)})$, except those that also lie in $B_r^\bullet(\cT_\infty^{(p)})$. We view $B_s^\bullet(\cT_\infty^{(p)}) \setminus B_r^\bullet(\cT_\infty^{(p)})$ as rooted at the edge of the bottom cycle that corresponds to the root of the first tree in the skeleton of $B_r^\bullet(\cT_\infty^{(p)})$.
Thus, conditional on $L_r^{(p)} = q$ (with $p = 0$ or $p \ge 3$), the triangulation of the cylinder $B_s^\bullet(\cT_\infty^{(p)}) \setminus B_r^\bullet(\cT_\infty^{(p)})$ is distributed according to $\mathbb{P}_{q, s-r}$ and is independent from $B_r^\bullet(\cT_\infty^{(p)})$.

We let $Y = (Y_r)_{r \ge 0}$ denote a Bienaym\'e--Galton--Watson process with offspring distribution $\theta$ (from Lemma~\ref{le:brtntria}) that starts with $k$ individuals under the probability measure $\mathbb{P}_k$. Thus, the probability generating function of $Y_r$ under $\mathbb{P}_1$ is the $r$th iterate $g_\theta^{(r)}$ of $g_\theta$. By induction, it follows as in~\cite[Eq. (16)]{zbMATH07144469} that
\begin{align}
	\label{eq:nowayjose}
	\mathbb{E}_1[x^{Y_r}] = g_\theta^{(r)}(x) = 1 - \left(r + \frac{1}{\sqrt{1-x}} \right)^{-2}.
\end{align}
Using singularity analysis~\cite{MR2483235}, it follows that
\begin{align}
	\label{eq:nowsubexp}
	\mathbb{P}_1(Y_r = k) \sim \frac{3r}{2\sqrt{\pi}} k^{-5/2}
\end{align}
as $k \to \infty$. Furthermore,
\begin{align}
	\mathbb{P}_q(Y_{r} = p) = [x^p] \left(g_\theta^{(r)}(x)\right)^q,
\end{align}
and in particular
\begin{align}
	\label{eq:trivbound00435}
	\mathbb{P}_q(Y_{r} = 0) = \left(g_\theta^{(r)}(0)\right)^q = \left(1 - \frac{1}{(r+1)^2}\right)^q.
\end{align}

We let $\ndF_{p,q,r}$ denote the collection of all $(p,q,r)$-admissible forests. We define set  $\ndF'_{p,q,r}$ of all pointed forests satisfying 
\begin{enumerate}[\qquad i)]
	\item The forest is an ordered sequence $(\tau_1, \ldots, \tau_q)$ of planted plane trees.
	\item Each generation of the forest has at least $3$ vertices.
	\item All trees have height at most $r$.
	\item The forest has exactly $p$ vertices with height $r$.
	\item The distinguished vertex is a leaf with height $r$.
	\item For all $1 \le j < r$, no vertex of  height $j$ is parent to all vertices of $\cF$ with height $j+1$.
\end{enumerate}
 We let $\ndF''_{p,q,r}$ denote the collection of all unmarked forests satisfying all requirements except v). 	Naturally, there is a $1$ to $q$ correspondence between $\ndF_{p,q,r}$ and $\ndF'_{p,q,r}$, with any forest from $\ndF_{p,q,r}$  corresponding to its $q$ cyclically permuted versions. 
 Furthermore, there is a $p$ to $1$ correspondence between $\ndF'_{p,q,r}$ and $\ndF''_{p,q,r}$, with any forest from $\ndF''_{p,q,r}$ corresponding to its $p$ versions with a marked leaf at height $r$. 

Let $a \in ]0,1[$ and let $N_r^{(a)}$ be uniformly distributed over $\{\lfloor a r^2 \rfloor +1, \ldots, \lfloor a^{-1} r^2 \rfloor \}$. Let $\tau_1, \tau_2, \ldots$ denote independent copies of a $\theta$-Bienaym\'e--Galton--Watson tree. For each $j \ge 0$ and any tree $\tau$ we let $[\tau]_j$ denote the tree truncated at height $j$. That is, we delete all vertices with height strictly larger than $j$. 

For all $1 \le r <s$ we let $\cF_{r,s}^{(0)}$ denote the skeleton of $B_s(\cT_\infty^{(0)}) \setminus B_r(\cT_\infty^{(0)})$. We also write $\widetilde{\cF}_{r,s}^{(0)}$ for the unmarked forest obtained by forgetting the marked vertex of $\cF_{r,s}^{(0)}$    and applying a uniform random cyclic permutation to the $L_s$ trees of~$\cF^{(0)}_{r,s}$. Thus, conditional on $L_r=p$ and $L_s=q$, $\widetilde{\cF}_{r,s}^{(0)}$ is uniformly distributed over $\ndF_{p,q,s-r}''$.

The following result is similar to the corresponding bound for type III triangulations~\cite[Lem. 4]{zbMATH07144469}. However, the proof is more complicated since, contrarily to type I triangulations, for type III triangulations we do not have a bijection between arbitrary triangulations and triangulations of the $1$-gon.
\begin{lemma}
	\label{le:analem}
	There exists $C_0>0$ such that for all $\alpha \ge 0$ and integers $r \ge 1$ and $q \ge 3$
	\begin{align}
		\label{eq:1pac}
		\Pr{L_r = q} \le C_0/r^2
	\end{align}
	and
	\begin{align}
		\label{eq:2pac}
		\Pr{L_r > \alpha r^2} \le C_0 \exp(- \alpha/5).
	\end{align}
\end{lemma}
\begin{proof}
	Let us first treat the case $r \ge 2$.
Using Corollary~\ref{co:formula} and the discussed correspondences between $\ndF_{p,q,r-1}$, $\ndF_{p,q,r-1}'$, and $\ndF_{p,q,r-1}''$, it follows that
	\begin{align}
		\label{eq:ineqalrq}
		\Pr{L_r = q} &= 	\sum_{p \ge 3} \Pr{d(\cT_\infty) = p} \Pr{L_{r-1}^{(p)} =q} \\
				&= \sum_{p \ge 3} \Pr{d(\cT_\infty) = p} \sum_{\cF \in \ndF_{p,q,r-1}}  \frac{(64/9)^{-q} C(q)}{(64/9)^{-p}C(p)}   \prod_{v \in \cF^*} \theta(c_v) \nonumber\\
				&= \sum_{p \ge 3} \Pr{d(\cT_\infty) = p}   \frac{q^{-1}(64/9)^{-q} C(q)}{p^{-1}(64/9)^{-p}C(p)}  \sum_{\cF \in \ndF_{p,q,r-1}''} \prod_{v \in \cF^*} \theta(c_v) \nonumber\\
				&\le   \sum_{p \ge 3} \Pr{d(\cT_\infty) = p}   \frac{q^{-1}(64/9)^{-q} C(q)}{p^{-1}(64/9)^{-p}C(p)}  \mathbb{P}_q(Y_{r-1} = p).\nonumber
		\end{align}
	The inequality sign in the last line is due to the restrictions on the forests from $\ndF''_{p,q,r-1}$.
	By Equation~\eqref{eq:cpasymptotoo},
		\begin{align}
			\label{eq:dontforgetme}
			q^{-1}(64/9)^{-q} C(q) \sim \frac{9 \sqrt{3}}{2048 \pi \sqrt{2}}  \frac{1}{\sqrt{q}}
		\end{align}
	as $q \to \infty$. 
	Moreover,
	\[
		\mathbb{P}_q(Y_{r-1} = p) = [x^p] (g_\theta^{(r-1)}(x))^q.
	\]
	By Equations~\eqref{eq:cpasymp}, \eqref{eq:infp}  and standard calculations
	\begin{align*}
		f(x) &:= \sum_{p \ge 3}   \frac{\Pr{d(\cT_\infty) = p} }{p^{-1}(64/9)^{-p}C(p)}  x^p \\
		&= \frac{1024 \sqrt{6 \pi } (2-x) x^3}{(3 x-4)^2}.
	\end{align*}
	Note that $f$ is analytic on $\{z \in \ndC \mid |z| < 4/3\}$, and the probability generating function $g_\theta^q$ is analytic at least on $\{ z \in \ndC \mid |z| < 1\}$.  Hence 
	\begin{align}
		\sum_{p \ge 3}   \frac{\Pr{d(\cT_\infty) = p} }{p^{-1}(64/9)^{-p}C(p)}  \mathbb{P}_q(Y_{r-1} = p) = \frac{1}{2\pi i} \int_\gamma \frac{g_\theta^q(z) h(1/z)}{z} \,\,\mathrm{d}z
	\end{align}
	for $\gamma: [0,1] \to \ndC, t \mapsto R \exp(2\pi i t)$ for arbitrary $3/4 < R < 1$. Note that in the domain $\{z \in \ndC \mid |z| < 1\}$ the function $f(1/z)$ has poles exactly in $0$ and $3/4$. Hence, by the residue theorem and standard calculations
	\begin{align*}
		 \frac{1}{2\pi i} \int_\gamma \frac{g_\theta^q(z) h(1/z)}{z} \,\,\mathrm{d}z &= \mathrm{Res}_{0} \left(\frac{g_\theta^q(z) h(1/z)}{z}\right) + \mathrm{Res}_{3/4} \left(\frac{g_\theta^q(z) h(1/z)}{z}\right).
	\end{align*}
	Equation~\eqref{eq:nowayjose} and standard calculations yield
	\begin{multline*}
		\mathrm{Res}_{0} \left(\frac{g_\theta^q(z) h(1/z)}{z}\right)  \\= -\frac{256 \sqrt{\frac{2 \pi }{3}} q \left(6 q+17 (r-1)^3+42 (r-1)^2+16
			(r-1)-6\right) \left(1-\frac{1}{r^2}\right)^q}{9 (r-1)^2 r^2 (r+1)^2}
	\end{multline*}
and
	\begin{align*}
	\mathrm{Res}_{3/4}  = \frac{16384 \sqrt{\frac{2 \pi }{3}} q \left(  1 - \frac{1}{(1+r)^2}\right)^{q-1}}{9 (r+1)^3}.
\end{align*}
Note that
\[
	\mathrm{Res}_{0} \left(\frac{g_\theta^q(z) h(1/z)}{z}\right)  < 0.
\]
Hence, using Inequality~\eqref{eq:ineqalrq} and Equation~\eqref{eq:dontforgetme} it follows that
\begin{align}
	\label{eq:fromthis453}
	\Pr{L_r = q} &\le C' \sqrt{q} r^{-3} \left(  1 - \frac{1}{(1+r)^2}\right)^{q-1} \\
	&\le C'' \frac{1}{r^2} \sqrt{\frac{q}{r^2}} \exp\left(-\frac{q}{4r^2} \right). \nonumber
\end{align}
for  constants $C', C''>0$ that do not depend on $r$ or $q$. This readily verifies the existence of a constant $C_0$ such that Inequality~\eqref{eq:1pac} holds for all $q \ge 3$ and $r \ge 2$. We may without loss of generality assume that $C_0>1$, hence~\eqref{eq:1pac} also holds for $r=1$.

As for Inequality~\eqref{eq:2pac}, we obtain for $r \ge 2$
\begin{align*}
	\Pr{L_r > \alpha r^2} &\le C'' \sum_{q > \alpha r^2} \frac{1}{r^2} \sqrt{\frac{q}{r^2}} \exp\left(-\frac{q}{4r^2} \right) \\
	&\le C'' \frac{1}{r^2} \int_{\alpha r^2}^\infty  \sqrt{\frac{x}{r^2}} \exp\left(-\frac{x}{4r^2} \right) \,\,\mathrm{d}x  \\
	&\le C''' \exp(-\alpha/5),
\end{align*}
for some constant $C'''>0$ that does not depend on $\alpha$. As for $r=1$, since 
\[
	L_1 \eqdist d(\cT_\infty),
\]
we get from~\eqref{eq:infpasymptotic} that
\[
	\Pr{L_1 > \alpha} = O(\sqrt{\alpha} (3/4)^\alpha) = O(\exp(-\alpha/5))
\]
since $\log(3/4) \approx -0.287 < -1/5$. This completes the proof.
\end{proof}

We prove the following bound using analogous arguments as for~\cite[Prop. 5]{zbMATH07144469}.
\begin{proposition}
	\label{pro:comparison}
	For each $a \in]0,1[$ there exists $C_1>0$ such that for all large enough integers $r$, all integers $s >r$,   all integers $p,q \in \{\lfloor a r^2 \rfloor +1, \ldots, \lfloor a^{-1} r^2 \rfloor \}$, and all forests $\cF \in \mathbb{F}_{p,q,s-r}''$ we have
	\[
		\Prb{ \widetilde{\cF}_{r,s}^{(0)} = \cF} \le C_1 \Prb{ ([\tau_1]_{s-r}, \ldots, [\tau_{N_r^{(a)}}]_{s-r}) = \cF  }.
	\]
\end{proposition}
\begin{proof}
	We have
	\begin{align}
		\label{eq:pulseaudio}
		\Prb{ ([\tau_1]_{s-r}, \ldots, [\tau_{N_r^{(a)}}]_{s-r}) = \cF  } &= \Pr{N_r^{(a)} = p} \Prb{ ([\tau_1]_{s-r}, \ldots, [\tau_p]_{s-r}) = \cF  } \\
		&= \frac{1}{\lfloor a^{-1} r^2 \rfloor - \lfloor a r^2 \rfloor } \prod_{v \in \cF^*} \theta(c_v), \nonumber
	\end{align}
	with $\cF^*$ denoting the collection of vertices of $\cF$ with height strictly less than $s-r$. Let $\cF^\circ \in \ndF_{p,q,s-r}$ be any element corresponding to $\cF$ up to cyclic permutation and forgetting about the marked vertex. By Corollary~\ref{co:formula}
	\begin{align*}
		\Pr{{\cF}_{r,s}^{(0)} = \cF^\circ \mid L_r = p } &= \mathbf{P}_{p,s-r}(\cF^\circ) \\
		&= \frac{(64/9)^{-q} C(q)}{(64/9)^{-p}C(p)}   \prod_{v \in \cF^*} \theta(c_v).
	\end{align*}
	As discussed before, there is a $1$ to $q$ correspondence between $\ndF_{p,q,r}$ and $\ndF'_{p,q,r}$, with any forest from $\ndF_{p,q,r}$  corresponding to its $q$ cyclically permuted versions, and  a $p$ to $1$ correspondence between $\ndF'_{p,q,r}$ and $\ndF''_{p,q,r}$. Hence
	\begin{align}
		\label{eq:slim}
				\Pr{\widetilde{\cF}_{r,s}^{(0)} = \cF \mid L_r = p } 
		= \frac{q^{-1}(64/9)^{-q} C(q)}{p^{-1}(64/9)^{-p}C(p)}   \prod_{v \in \cF^*} \theta(c_v).
	\end{align}
	It follows from Equation~\eqref{eq:cpasymptotoo} and the assumptions on $p$ and $q$ that
	\begin{align}
		\label{eq:shady}
		\frac{q^{-1}(64/9)^{-q} C(q)}{p^{-1}(64/9)^{-p}C(p)} = O(\sqrt{p/q}) \le C_2
	\end{align}
	for a constant $C_2>0$ that only depends on $a$. By Lemma~\ref{le:analem}, it follows that
	\begin{align*}
			\Pr{\widetilde{\cF}_{r,s}^{(0)} = \cF } \le \frac{C_0 C_2}{r^2} \prod_{v \in \cF^*} \theta(c_v).
	\end{align*}
By Equation~\eqref{eq:pulseaudio}, it follows that
	\[
\Prb{ \widetilde{\cF}_{r,s}^{(0)} = \cF} \le C_1 \Prb{ ([\tau_1]_{s-r}, \ldots, [\tau_{N_r^{(a)}}]_{s-r}) = \cF  }
\]
for a constant $C_1>0$ that only depends on $a$.
\end{proof}

\begin{corollary}
	\label{co:jimmychoo}
	Let $u_0^{(n)}$ denote a uniform random vertex of the top cycle $\partial^* B_n^\bullet(\cT_\infty^{(0)})$. We may enumerate the vertices of $\partial^* B_n^\bullet(\cT_\infty^{(0)})$ in clockwise order starting at $u_0^{(n)}$ as $u_0^{(n)}, \ldots, u_{L_n-1}^{(n)}$. 
	Let $\delta>0$ and $0<a<1$. For sufficiently small $0<\eta<1/2$ and sufficiently large $n$ it holds with probability at most $\delta$ that simultaneously
	\begin{align}
		\label{eq:jimmy}
		a n^2 \le L_n \le a^{-1} n^2
	\end{align}
	and
	\begin{align}
		\label{eq:choo}
		a n^2 \le L_{n- \lfloor \eta n\rfloor} \le a^{-1} n^2
	\end{align}
	and that the left-most geodesics from $u_0^{(n)}$ and $u_{\lfloor a n^2 / 2\rfloor }^{(n)}$ coalesce before hitting $\partial^* B_{n - \lfloor \eta n\rfloor}^\bullet( \cT_\infty^{(0)})$.
\end{corollary}
\begin{proof}
	By the discussion at the end of Section~\ref{sec:app} it suffices  to bound the probability of Inequalities~\eqref{eq:jimmy} and~\eqref{eq:choo} holding simultaneously with the event that the subforest consisting of the first $\lfloor a n^2/2\rfloor$ trees of $\widetilde{\cF}^{(0)}_{n - \lfloor \eta n\rfloor, n}$ (starting this time from a uniformly selected location with a cyclic ordering of the trees) has  height strictly smaller than $\lfloor \eta n \rfloor$, or that the same holds for the $L_n - 1 - \lfloor a n^2 / 2\rfloor  \ge a n^2/2-2$ trees in the complementary subforest. Applying Proposition~\ref{pro:comparison} for $r = n - \lfloor \eta n \rfloor \in [n/2, n]$, this probability is bounded by \[
		\Prb{ \max_{1 \le i \le a n^2/4} \He(\tau_i) <  \lfloor \eta n\rfloor},
	\]
	up to a multiplicative constant that only depends on $a$. Here $(\tau_i)_{i \ge 1}$ denote independent $\theta$-Bienaym\'{e}--Galton--Watson trees and $\He(\cdot)$ their height. By Equation~\eqref{eq:trivbound00435} it follows that for $n$ large enough (depending on $\eta$)
	\begin{align*}
		 \Prb{ \max_{1 \le i \le a n^2/2} \He(\tau_i) <  \lfloor \eta n\rfloor} &= \mathbb{P}_{\lfloor a n^2/2 \rfloor} (Y_{\lfloor \eta n\rfloor} = 0) \\
		 &= \left(1 - \frac{1}{(\lfloor \eta n\rfloor+1)^2}\right)^{\lfloor a n^2/2 \rfloor} \\
		 &\le \exp\left( -a / (4 \eta^2)\right).
	\end{align*}
	Taking $\eta$ large enough, the total bound is hence smaller than $\delta$ for sufficiently large~$n$.
\end{proof}

\section{Simple half-plane triangulations}

\subsection{The type III Upper Half-Plane Triangulation}

Adapting the arguments from~\cite[Sec. 3.1]{zbMATH07144469}, we construct a triangulation of the upper half-plane $\ndR \times \ndR_{\ge 0}$. 

We let $\bar{\theta}$ with $\bar{\theta}(k) = k \theta(k)$ for $k \ge 1$ denote the size-biased version of $\theta$. We embed a modified version of Kesten's tree with a backwards growing spine into the half-plane. It's vertex set will be precisely the collection of points $(1/2 +i, j)$ for integers $i \in \ndZ$ and $j \in \ndZ_{\ge 0}$. 

The backwards growing spine of the tree consists of the points $(1/2, j)$, $j \ge 0$. For  each $j \ge 1$ the spine vertex at $(1/2, j)$ receives an independent number $m_j$ of children following the size biased distribution $\bar{\theta}$. The coordinates of the children are chosen to be $(1/2 + k, j-1)$ for $\ell_j - m_j \le k \le \ell_j -1$ with $\ell_j$ uniform over $\{1, \ldots, m_j\}$. This way, the unique spine child $(1/2, j-1)$ has  uniform rank. All non-spine vertices of the tree with positive $y$-coordinates created in this way become roots of independent copies of $\theta$-Bienaym\'e--Galton--Watson trees that we truncate when they hit the $x$-axis.

Note that by Equation~\eqref{eq:nowayjose}  a $\theta$-Bienaym\'e--Galton--Watson tree has height at least $r$ with probability $\frac{1}{(1+r)^2}$.  It follows that on both sides of the spine there are infinitely branches that hit the $x$-axis. Hence we may assign the coordinates of the vertices of all  branches in a way so that the vertex set of the entire modified Kesten tree is precisely~$\{ (1/2 +i, j) \mid i, j \in \ndZ, j \ge 0\}$, and so that its edges may be drawn as straight lines that only intersect at their endpoints.

We now construct the upper half-plane simple triangulation. For all $(i,j)$ with $i \in \ndZ$ and $j \in \ndZ_{\ge0}$ we draw a horizontal edge from $(i,j)$ to $(i+1,j)$. If $j \ge 1$ we construct a \emph{downward triangle} containing this edge such that the third vertex $(1/2 +k, j-1)$ is determined by letting $k$ be the minimal integer such that  $(1/2 + k, j-1)$ is a child of $(1/2 + i', j)$ for some $i' > i$. We merge each double-edge created in this way into a single-edge. This way,  edges are straight line-segments that do not cross.

This way, the modified Kesten tree uniquely determines the configuration of downward triangles. As before, we fill all slots by gluing independent copies of simple Boltzmann triangulations with the corresponding perimeters to their boundary. We call the resulting simple triangulation $\cU$ the type III UHPT for Upper Half-Plane Triangulation. We declare the edge from $(0,0)$ to $(1,0)$ as its oriented root-edge, and let $\partial \cU$ denote its (bottom) boundary. See Figure~\ref{fi:uhpt} for an illustration.

\begin{figure}[t]
	\centering
	\begin{minipage}{0.8\textwidth}
		\centering
		\includegraphics[width=1.0\linewidth]{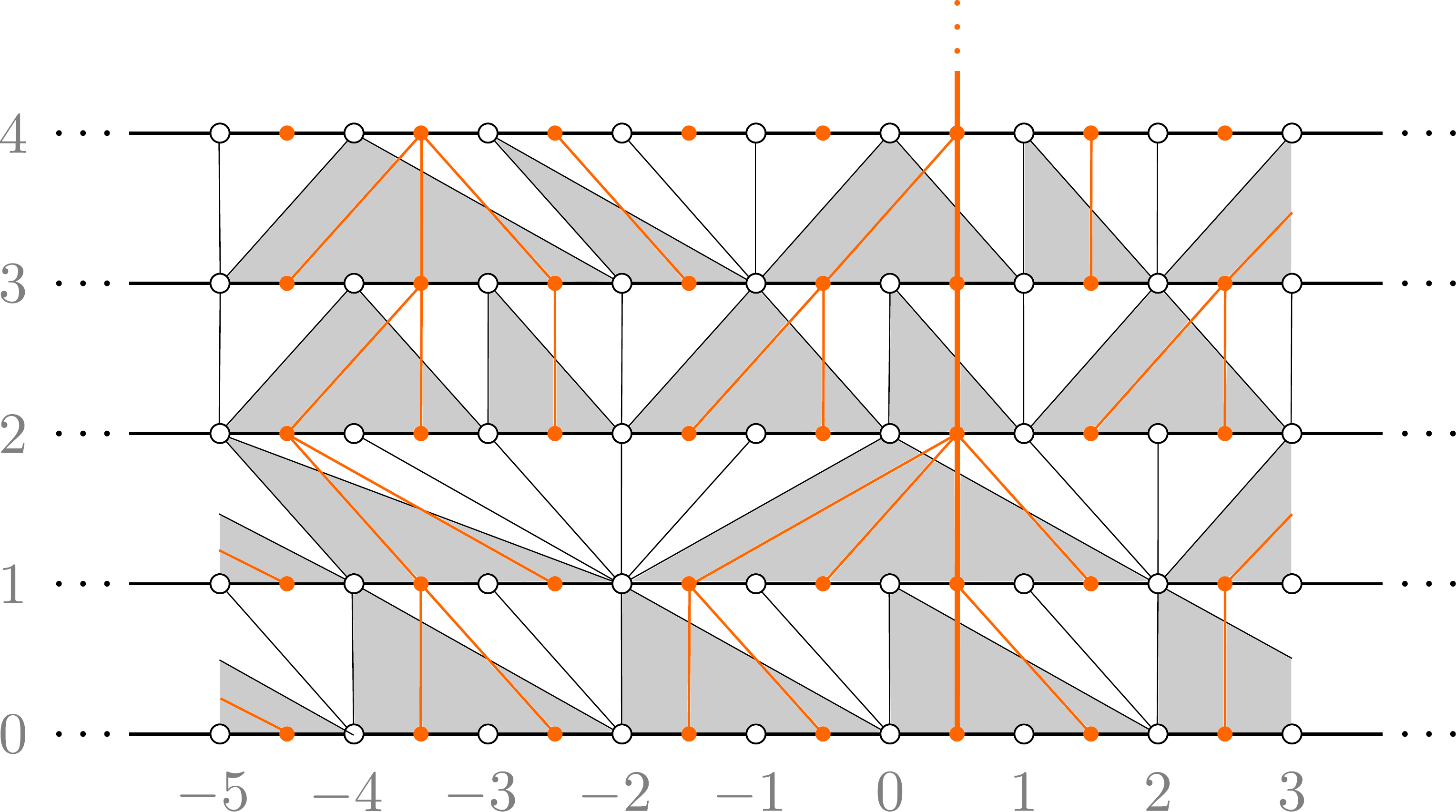}
		\caption{The type III upper half-plane triangulation.}
		\label{fi:uhpt}
	\end{minipage}
\end{figure}

We verify the following local convergence by adapting the proof of a similar convergence~\cite[Prop. 6]{zbMATH07144469} for type I triangulations. See also~\cite{angel2005scaling} for a prior construction and similar convergence for type II triangulations.

\begin{proposition}
	\label{pro:localconv}
	As $p \to \infty$,
	\[
		\cT_\infty^{(p)} \convdis \cU
	\]
	in the local topology.
\end{proposition}
\begin{proof}
We have \begin{align}
L_r^{(p)} \convp \infty
\end{align} as $p \to \infty$, since for each fixed $q \ge 3$ and $r \ge 1$ we have similarly as in Inequality~\eqref{eq:ineqalrq} and using~\eqref{eq:nowsubexp} and~\eqref{eq:dontforgetme} that
	\begin{align*}
		\Pr{L_r^{(p)} = q} &\le \frac{q^{-1}(64/9)^{-q} C(q)}{p^{-1}(64/9)^{-p}C(p)}  \mathbb{P}_q(Y_{r} = p) \\
		&= O(\sqrt{p}) [x^p] (g_\theta^{(r)}(x))^q \\
		&= O(\sqrt{p}) [x^p] (g_\theta^{(r)}(x)) \\
		&= O(1 / p^2).
	\end{align*}
	Here we have used that as $p \to \infty$ \[
	[x^p] (g_\theta^{(r)}(x))^q \sim q [x^p] g_\theta^{(r)}(x),
	\]
	 since $q$ is constant and  $Y_r$ follows the power law in~\eqref{eq:nowsubexp}. 
	 
	 Let $\cB_r(\cT_\infty^{(p)})$ and $\cB_r(\cU)$ denote the planar maps consisting of all faces that are incident to a vertex with graph distance strictly less than $r$ from the root-vertex.  In order to prove $\cT_\infty^{(p)} \convdis \cU$ as $p \to \infty$ it suffices to show that for all $r \ge 1$
	 \begin{align}
	 	\label{eq:toshow14365140}
	 	\cB_r(\cT_\infty^{(p)}) \convdis \cB_r(\cU)
	 \end{align}
	 as random finite planar maps.
	 
	 To this end, for each $r \ge 1$ we write
	 \[
	 \cF_{0,r}^{(p)} = (\cJ_0^{(p)}, \ldots, \cJ_{L_r^{(p)} -1}^{(p)})
	 \]
	 for the skeleton of $B_r^\bullet(\cT_\infty^{(p)})$. Let $(t_{-k},\ldots, t_k)$ be a sequence sequence of finite plane trees  satisfying the following conditions:
	 \begin{enumerate}[\qquad a)]
	 	\item Their height is at most $r$.
	 	\item The tree $t_0$ carries  a marked leaf at height $r$ that belongs to $t_0$.
	 	\item Each generation of the forest has at least $3$ individuals.
	 	\item No generation consists entirely of children of the same parent.
	 \end{enumerate}
	 
	  We are going to show that 
	 \begin{multline}
	 	\label{eq:asdfqwer2345}
	 	\Prb{L_r^{(p)} \ge 4k, \cJ_{L_r^{(p)}-k}^{(p)} = t_{-k}, \ldots, \cJ_{L_r^{(p)}-1}^{(p)} = t_{-1}, \cJ_{0}^{(p)} = t_{0}, \dots, \cJ_{L_r^{(p)}+k}^{(p)} = t_k} \\ \to \Prb{ \Gamma_{(-k,r)} =t_{-k}, \ldots, \Gamma_{(k,r)} = t_k}
	 \end{multline}
 as $p \to \infty$. Here $\Gamma_{i,j}$ denotes the fringe subtree at $(1/2 +i, j)$ in the modified Kesten tree, interpreted as planted plane tree.  
 
 In order to see that this implies~\eqref{eq:toshow14365140}, note that for each $r \ge 1$ and  $\epsilon>0$ we may choose $k \ge 1$ sufficiently large such that there is a collection $\mathbf{F}_k$ of forests of the form $(t_{-k}, \ldots, t_k)$ satisfying conditions a), \ldots, d) with $\Pr{ (\Gamma_{(-k,r)}, \ldots, \Gamma_{(k,r)}) \in \mathbf{F}_k} > 1-\epsilon$ and the following properties: Conditional on this event the ball $\cB_r(\cU)$ is determined by $(\Gamma_{(-k,r)}, \ldots, \Gamma_{(k,r)})$ and the simple triangulations with a boundary inserted at the slots associated to the vertices of these trees.  Likewise, conditional on the event that $L_r^{(p)} \ge 4k$ and   $(\cJ_{L_r^{(p)}-k}^{(p)}, \ldots, \cJ_{L_r^{(p)}-1}^{(p)},\cJ_{0}^{(p)}, \dots, \cJ_{L_r^{(p)}+k}^{(p)}) \in \mathbf{F}_k$, the ball $\cB_r(\cT_\infty^{(p)})$ is also fully determined by these trees and the simple triangulations  of polygons inserted in the corresponding slots. Thus,~\eqref{eq:toshow14365140} follows from~\eqref{eq:asdfqwer2345}.

Now, in order to show~\eqref{eq:asdfqwer2345}, let $(t_{-k},\ldots, t_k)$ be a sequence sequence of finite plane trees satisfying conditions a), \ldots, d). Letting $\V^*(t_i)$ denote the collection of vertices with height strictly less than $r$, we have
	\begin{align}
		\Prb{ \Gamma_{(-k,r)} =t_{-k}, \ldots, \Gamma_{(k,r)} = t_k} = \prod_{i=-k}^k \prod_{v \in \V^*(t_i)} \theta(c_v).
	\end{align}
Here $c_v$ denotes the number of children of a vertex $v$.  For ease of notation, we set $\cF_{(k)} = (t_{-k}, \ldots, t_k)$ and write $\phi_k(\cF) = (\sigma_{\ell -k}, \ldots, \sigma_{\ell -1}, \sigma_0, \ldots, \sigma_k)$ for all $\cF = (\sigma_0, \ldots, \sigma_{\ell}) \in \cF_{p, \ell, r}$ with $\ell \ge 2k+ 1$. Let $m_k$ denote the number of vertices at height $r$ in $\cF_{(k)}$. By Corollary~\ref{co:formula}, it follows that for $p \ge m_k$ the left-hand side of~\eqref{eq:asdfqwer2345} is given by
\begin{align*}
 \sum_{\ell = 2k +1}^\infty \sum_{\substack{\cF \in \ndF_{p,\ell,r} \\ \phi_k(\cF) = \cF_{(k)} }} \frac{(64/9)^{-\ell} C(\ell)}{(64/9)^{-p}C(p)}   \prod_{v \in \cF^*} \theta(c_v).
\end{align*}
 Note that $\phi_k(\cF) = \cF_{(k)}$ already ensures that no generation consists entirely of children of the same parent, and that each generation has at least $3$ individuals. Hence we may rewrite this expression as
\begin{align*}
	\left(\prod_{i=-k}^k \prod_{v \in \V^*(t_i)} \theta(c_v)\right) \sum_{\ell = 4k}^\infty \frac{(64/9)^{-\ell} C(\ell)}{(64/9)^{-p}C(p)}  \sum_{\substack{\sigma_{k+1}, \ldots, \sigma_{\ell - k -1}}}	\left(\prod_{i=k+1}^{\ell -k -1} \prod_{v \in \V^*(\sigma_i)} \theta(c_v)\right),
\end{align*}
with the sum indices $\sigma_{k+1}, \ldots, \sigma_{\ell - k +1}$ ranging over forests of plane trees of height at most $r$ with a total number of vertices at height $r$ equal to $p - m_k$. In order to shorten notation, we set $\varphi(\ell) = (64/9)^{-\ell} C(\ell)$ and
\begin{align*}
	A_p &:= \sum_{\ell = 4k}^\infty \frac{\varphi(\ell)}{\varphi(p)}  \sum_{\substack{\sigma_{k+1}, \ldots, \sigma_{\ell - k -1}}}	\left(\prod_{i=k+1}^{\ell -k -1} \prod_{v \in \V^*(\sigma_i)} \theta(c_v)\right) \\
	&= \sum_{\ell = 4k}^\infty \frac{\varphi(\ell)}{\varphi(p)}  \mathbb{P}_{\ell - (4k)}(Y_r = p - m_k)\\
	&= \sum_{\ell = 0}^\infty \frac{\varphi(\ell+4k)}{\varphi(p)}  \mathbb{P}_{\ell}(Y_r = p - m_k)\\
	&\ge  \sum_{\ell = 3}^\infty \frac{\varphi(\ell)}{\varphi(p)}  \mathbb{P}_{\ell}(Y_r = p - m_k).
\end{align*}
Here we have used that, by Equation~\eqref{eq:cpasymp}, $\varphi$ is monotonically increasing. Writing
\begin{align}
	\tilde{h}(\ell) := \varphi(\ell)\ell^{-1},
\end{align}
it follows trivially that for $0< \epsilon < 1/4$
\begin{align}
	\label{eq:karda}
	A_p \ge (1-\epsilon) \sum_{\ell = \lfloor(1-\epsilon)p\rfloor +1}^\infty \frac{\tilde{h}(\ell)}{\tilde{h}(p)}  \mathbb{P}_{\ell}(Y_r = p - m_k).
\end{align}
By~\eqref{eq:cpasymptotoo} we know that
\begin{align}
	\tilde{h}(\ell) \sim \frac{9 \sqrt{3}}{2048 \pi \sqrt{2}} \frac{1}{\sqrt{\ell}}.
\end{align}
We define for all $\ell \ge 1$
\begin{align}
	\label{eq:htotheizzo}
	h(\ell) = 4^{-\ell} \binom{2\ell}{\ell} \sim \frac{1}{\sqrt{\pi \ell}}.
\end{align}
Thus, for $p$ large enough, it follows that 
\begin{align}
	\label{eq:kardabba}
	A_p \ge (1-2\epsilon) \sum_{\ell = \lfloor(1-\epsilon)p\rfloor +1}^\infty \frac{{h}(\ell)}{{h}(p)}  \mathbb{P}_{\ell}(Y_r = p - m_k).
\end{align}

Next, we will show that
\begin{align}
	\label{eq:kanyetoshow}
	\lim_{p \to \infty} \sum_{\ell=1}^{\lfloor (1-\epsilon) p \rfloor} \frac{h(\ell)}{h(p)}  \mathbb{P}_{\ell}(Y_r = p - m_k) = 0.
\end{align}
Recall that by Lemma~\ref{le:brtntria} the offspring distribution $\theta$ is critical. Hence   $Y_r - \ell$ under $\mathbb{P}_\ell$ is the sum of $\ell$ i.i.d. centred random variables that asymptotically follow the power law in~\eqref{eq:nowsubexp}. Hence, for $\ell \le (1-\epsilon)p$, we have that $\mathbb{P}_{\ell}(Y_r -\ell = p - \ell - m_k)$ is the probability for this centred sum with $\ell$ summands to assume a value 
\[
p - \ell - m_k \ge \ell(\epsilon/(1-\epsilon)) - m_k.
\]
Using results for the big-jump domain of random walks~\cite[Cor. 2.1]{MR2440928}, it follows that
there exists $\ell_0 >0$ and $C_\epsilon >0$ such that uniformly for all sufficiently large $p$ and all $\ell_0 \le \ell \le (1-\epsilon)p$ 
\begin{align}
	\label{eq:kanye1}
\mathbb{P}_{\ell}(Y_r -\ell = p - \ell - m_k) \le C_\epsilon \ell \mathbb{P}_1(Y_r -1 = p - \ell - m_k ).
\end{align}
Furthermore, it is clear that uniformly for all $3 \le \ell \le \ell_0$
\begin{align}
	\label{eq:kanye2}
\mathbb{P}_{\ell}(Y_r  = p  - m_k) \sim \ell  \mathbb{P}_1(Y_r = p - m_k).
\end{align}
as $p\to \infty$. Combining~\eqref{eq:kanye1} with~\eqref{eq:kanye2} and using~\eqref{eq:nowsubexp} it follows that there exists a constant $C_\epsilon'>0$ such that  uniformly for all sufficiently large $p$ and all integers $1 \le \ell \le (1-\epsilon)p$  
\begin{align}
	\label{eq:kanye3}
\mathbb{P}_{\ell}(Y_r = p  - m_k) \le C_\epsilon' \ell p^{-5/2}.
\end{align}
Using Equation~\eqref{eq:htotheizzo} it follows that 
\[
\sum_{\ell=1}^{\lfloor (1-\epsilon) p \rfloor} \frac{h(\ell)}{h(p)}  \mathbb{P}_{\ell}(Y_r = p - m_k) = O(p^{-2}) \sum_{\ell=1}^{\lfloor (1-\epsilon)p \rfloor } \ell^{1/2} \to 0
\]
as $p \to \infty$. This verifies~\eqref{eq:kanyetoshow}.

Combining~\eqref{eq:kanyetoshow} and~\eqref{eq:kardabba}, it follows that
\begin{align}
	\label{eq:almostthere43453}
	\liminf_{p \to \infty} A_p &\ge (1- 2\epsilon) \liminf_{p \to \infty} \sum_{\ell=1}^\infty \frac{h(\ell)}{h(p)} \mathbb{P}_\ell(Y_r = p  - m_k).
\end{align}
The reason for making the switch from $\tilde{h}$ to $h$ is that
\[
	\Pi(x) := \sum_{\ell \ge 1} \tilde{h}(x) = \frac{1}{\sqrt{1-x}} - 1
\]
satisfies, by standard calculations, 
\begin{align}
	\Pi( g_\theta^{(r)}(x) ) - \Pi( g_\theta^{(r)}(0) )  = \Pi(x).
\end{align}
Hence, for all $q \ge 1$
\begin{align}
	\sum_{\ell=1}^\infty {h(\ell)}\mathbb{P}_\ell(Y_r = q) = h(q).
\end{align}
Hence Inequality~\eqref{eq:almostthere43453} simplifies to
\begin{align*}
	\liminf_{p \to \infty} A_p &\ge (1- 2\epsilon) \liminf_{p \to \infty}  \frac{h(p-m_k)}{h(p)} \\
	&= 1- 2\epsilon. 
\end{align*}
Since $\epsilon$ was arbitrary, it follows that
\begin{align}
	\liminf_{p \to \infty} A_p \ge 1.
\end{align}
Thus, we have verified
\begin{align*}
	&\liminf_{p \to \infty} \Prb{L_r^{(p)} \ge 4k, \cJ_{L_r^{(p)}-k}^{(p)} = t_{-k}, \ldots, \cJ_{L_r^{(p)}-1}^{(p)} = t_{-1}, \cJ_{0}^{(p)} = t_{0}, \dots, \cJ_{L_r^{(p)}+k}^{(p)} = t_k} \\ &\ge \Prb{ \Gamma_{(-k,r)} =t_{-k}, \ldots, \Gamma_{(k,r)} = t_k}.
\end{align*}
The sum of the quantities on the right-hand side over possible choice of forests $(t_{-k}, \ldots, t_k)$ equals $1$. Hence this implies~\eqref{eq:asdfqwer2345} and completes the proof.
\end{proof}

\begin{remark}
	\label{re:uhptthesame}
	The difference between the type III and type I UHPT is less pronounced than the differences between the types of $\cT_n$, $\cT_n^{(p)}$, and $\cT_{\infty}^{(p)}$. This is because the skeleton of the type I UHPT almost surely satisfies the additional constraints  on the skeletons in the type III case, and because we recovered the same offspring distribution $\theta$ in the type III case. Hence the type III  UHPT may be constructed from the skeleton of the type I UHPT by contracting each degree $2$ slot into a single edge and by gluing independent simple Boltzmann triangulations with the corresponding perimeters into the remaining slots.
\end{remark}

\subsection{The type III Lower Half-Plane Triangulation}

We construct the type III lower half-plane triangulation in a similar way as the upper half-plane triangulation. It's skeleton has vertex set $\ndZ \times \ndZ_{\le 0}$ and instead of a modified Kesten tree we use a doubly infinite sequence $(\cJ)_{i \in \ndZ}$ of $\theta$-Bienaym\'e--Galton--Watson trees. These trees are embedded so that the root of $\cJ_i$ is $(1/2 +i, 0)$ for all $i \in \ndZ$, and the vertex set of this infinite forest is precisely $(1/2 + \ndZ) \times \ndZ_{\le 0}$. The vertex set of the trees $\cJ_i$ for $i \ge 0$ is precisely $(1/2 + \ndZ_{\ge0}) \times \ndZ_{\le 0}$.

The downward triangles are constructed in the same way as for the UHPT. That is, for all integers $i \in \ndZ$ and $j \in \ndZ_{\le 0}$ we draw a horizontal edge from $(i,j)$ to $(i+1,j)$ and construct a downward triangle whose third vertex has coordinates $(k, j-1)$ with $k$ the smallest integer such that $(1/2 + k, j-1)$ is a child of $(1/2 + i', j)$ for some $i'>i$. Each resulting pair of double edges is merged into a single edge. We glue independent Boltzmann triangulations with the corresponding perimeter into all remaining slots. We consider the resulting type III lower half-plane triangulation (LHPT) $\cL$ as rooted at the oriented edge from $(0,0)$ to $(1,0)$. 

\begin{remark}
	\label{re:lhptthesame} 
	Similar as mentioned in Remark~\ref{re:uhptthesame} for the UHPT, the type III lower half-plane triangulation may be constructed from the skeleton of the type I LHPT by contracting each slot of degree $2$ to a single edge, and by gluing independent simple triangulations with the corresponding perimeters into the remaining slots.
\end{remark}

In order to illustrate the relation between the LHPT and the model $\cT_{\infty}^{(p)}$ we mention the following property:

\begin{proposition}
	\label{pro:lhptlimit}
	Given integers $p \ge 3$ and $r \ge 1$, let $\tilde{B}_r^\bullet(\cT_\infty^{(p)})$ be obtained from the hull ${B}_r^\bullet(\cT_\infty^{(p)})$ by re-rooting at a uniformly selected edge of the top cycle, oriented so that the top face lies to its left. Then
	\[
	{B}_r^\bullet(\cT_\infty^{(p)}) \convdis \cL
	\]
	in the local topology.
\end{proposition}

The proof is by analogous arguments as for Proposition~\ref{pro:localconv}. We omit the  details, since we are not going to make use of Proposition~\ref{pro:lhptlimit} in what follows. The type III LHPT and UHPT satisfy a specific contiguity relation that we are going to describe now.

For each integer $r \ge 1$ we let $\cU_{[0,r]}$ denote the infinite rooted planar map consisting of the first $r$ layers of $\cU$. That is, we only keep those vertices and edges that lie in the strip $\ndR \times [0,r]$. Hence $\cU_{[0,r]}$ is the hull of radius $r$, corresponding to distances from the bottom boundary. Likewise, we let $\cL_{[0,r]}$ denote the rooted planar map consisting of the first $r$ layers of $\cL$. That is, we only keep those vertices and edges that lie in the strip $\ndR \times [-r,0]$. 	Moreover, we set $\cU_r= \ndZ \times \{r\}$ and $\cL_r= \ndZ \times \{-r\}$.

Recall that $\Gamma_{i,r}$ denotes the subtree of descendants of $(1/2 + i, r)$ in the modified Kesten tree in the construction of $\cU$. Thus, $(\Gamma_{i,r})_{i \in \ndZ \setminus\{0\}}$ are independent Bienaym\'e--Galton--Watson trees truncated at height $r$. For all $i \in \ndZ$ we let $\Gamma_{(i,r)}(r)$ denote the collection of vertices of $\Gamma_{i,r}$ with height $r$. We also let $K_r \ge 1$ be the first index $i \ge 1$ such that $\Gamma_{(i,r)}(r) \ne \emptyset$. 

Let $i_r <0$ denote the largest integer $i<0$ such that the tree $\cJ_i$ has height at least $r$. We let $\cJ_{i_r}(r)$ denote the collection of vertices of $\cJ_{i_r}$ at height $r$. 

\begin{proposition}
	\label{pro:omit1}
	Let $\widetilde{\cU}_{[0,r]}$ be obtained by re-rooting $\cU_{[0,r]}$ so that the root-edge is the horizontal edge from $(J_r, r)$ to $(J_r+1,r)$, with an index $J_r$ selected uniformly at random from $\{1, \ldots, K_r\}$. For any non-negative measurable function $f$ on the collection of rooted planar maps we have
	\[
		\Ex{ K_r f(\widetilde{\cU}_{[0,r]})} = \Ex{ \# \cJ_{i_r}(r) f(\cL_{[0,r]})  }
	\]
\end{proposition}
The proof is by  identical arguments as for the corresponding result~\cite[Prop. 8]{zbMATH07144469} for the type I case. No adaptions are necessary.  The reason for this is the following: In order to prove Proposition~\ref{pro:omit1}, it suffices to show that the distribution of the configuration of downward triangles is the same for $\widetilde{\cU}_{[0,r]}$, under the measure having density $K_r$ with respect to $\mathbb{P}$, and for $\cL_{[0,r]}$, under the measure having density $\# \cJ_{i_r}(r)$ with respect to $\mathbb{P}$. This result on the skeletons is exactly what was verified in~\cite[Prop. 8]{zbMATH07144469} in the type I case. By Remark~\ref{re:uhptthesame} and Remark~\ref{re:lhptthesame} we know that the skeletons of the UHPT and LHPT in the type III case are, respectively, the same as the skeletons of the UHPT and LHPT in the type I case. Hence, in order to verify Proposition~\ref{pro:omit1} we may copy the arguments for~\cite[Prop. 8]{zbMATH07144469} word by word, adjusting only the references to the corresponding intermediate results derived so far for type III triangulations. 

Having Proposition~\ref{pro:omit1} at hand, the following result may likewise be verified by identical arguments (without any adaption) as for the corresponding result~\cite[Cor. 9]{zbMATH07144469}:

\begin{corollary}
	For each $\epsilon>0$ we may choose $\delta>0$ small enough so that for each integer $r \ge 1$ and every measurable set $A$ the property $\Pr{\widetilde{\cU}_{[0,r]} \in A} \le \delta$ implies $\Pr{\cL_{[0,r]} \in A} \le \epsilon$.
\end{corollary}

\section{Estimates for distances along the boundary}

The present section adapts results on distances along the boundaries of the UHPT and LHPT presented in~\cite[Sec. 4]{zbMATH07144469} for type I triangulations. As discussed in Remark~\ref{re:uhptthesame} and Remark~\ref{re:lhptthesame}, the configuration of downward triangles (that is, the skeletons) of the type III UHPT and LHPT are identical to the type I case. The only difference between type I and III for these infinite triangulations is that in the type III case we merge each slot of degree $2$ into a single edge, and fill the remaining slots with  independent simple (as opposed to unrestricted) Boltzmann triangulations (see Definition~\ref{def:boltzmann}) with the corresponding perimeter. For this reason, many of the proofs in~\cite[Sec. 4]{zbMATH07144469} require little to no adaption to treat the type III case. Hence we are going to provide proofs for results where small adaptions are needed, and refer the reader to the corresponding parts of~\cite[Sec. 4]{zbMATH07144469} for detailed justifications of results whose proof of the type I case could practically be copied word by word to treat the type III case.

\subsection{Layers of balls in the type III UHPT}

\label{sec:layers}

Let $r \ge 1$ be an integer. As before, we let $\cB_r(\cU)$ denote the planar map consisting of all faces that are incident to a vertex with graph distance strictly less than $r$ from the root-vertex of $\cU$. We define the hull $\cB_r^\bullet(\cU)$ as the complement of the unique infinite component of $\cB_r(\cU)$. Thus  $\cB_r^\bullet(\cU)$ is a triangulation with a simple boundary, consisting of a finite path on the boundary of $\cU$ that includes the root-edge of $\cU$, and a path of non-boundary edges that joins the two extremes of $\cB_r(\cU)$ lying on the boundary of $\cU$. We may view  $\cB_r(\cU)$ as marked at these two boundary vertices.
The proof of the following observation is analogous to the result~\cite[Lem. 10]{zbMATH07144469} for the type I case.

\begin{lemma}
	\label{eq:lemmorty}
	Let $A$ be simple triangulation with a boundary that is marked at two distinct vertices at the boundary that differ from the root vertex. Let $\tilde{\partial}A$ denote the part of $\partial A$ given by the path between the two distinguished vertices that contain the root edge. Suppose that $\Pr{\cB_r^\bullet(\cU) = A} >0$. Let $m \ge 2$ denote the number of edges of $\tilde{\partial}A$. Let $q \ge 1$ be the number of edges of $\partial A \setminus \tilde{\partial} A$. Also let $N \ge 0$ be the number of vertices of $A$ that do not lie on $\tilde{\partial}A$. Then
	\[
		\Pr{\cB_r^\bullet(\cU) = A} = (64/9)^{q-m} (256/27)^{-n}.
	\]
\end{lemma}
\begin{proof}
	Given another triangulation $A'$ with a boundary, we write $A \sqsubset A'$ to denote that $A$ may be obtained as a subtriangulation of $A'$, with root edges coinciding such that $\tilde{\partial} A$ is part of $\partial A'$ and no other edge of $A$ is on $\partial A'$. By Proposition~\ref{pro:localconv}, it follows that if $\Pr{\cB_r^\bullet(\cU) = A} >0$ then
	\begin{align}
		\label{eq:fu}
		\Pr{\cB_r^\bullet(\cU) = A} = \lim_{p \to \infty} \Pr{ A \sqsubset \cT_\infty^{(p)} }.
	\end{align}
	Moreover, by Corollary~\ref{co:cotnptinfty} it follows that 
	\begin{align}
		\label{eq:abc}
		\Pr{ A \sqsubset \cT_\infty^{(p)} } =\lim_{n \to \infty} \Pr{ A \sqsubset \cT_n^{(p)} }.
	\end{align}
	Let $p > m$. We have $A \sqsubset \cT_n^{(p)}$ holds if and only if $\cT_n^{(p)}$ may be obtained by gluing a simple triangulation $T$ with a boundary of length $q + (p-m)$ to $A$, such that a part of length $q$ of the boundary of $T$ gets identified with $\partial A \setminus \tilde{\partial} A$. Hence, for $n$ large enough,
	\[
		\Pr{ A \sqsubset \cT_\infty^{(p)} } = \frac{\# \ndT_{n-N, p+q-m}}{\# \ndT_{n, p}}.
	\]
	Thus, by Equation~\eqref{eq:abc} and~\eqref{eq:tnpasymp}, it follows that 
	\[
		\Pr{ A \sqsubset \cT_\infty^{(p)} } = \frac{C(p+q-m)}{C(p)} \left( \frac{256}{27} \right)^{-N}.
	\]
	By Equation~\eqref{eq:fu} and~\eqref{eq:cpasymptotoo} it follows that
	\[
	\Pr{\cB_r^\bullet(\cU) = A} =   \left( \frac{64}{9} \right)^{q-m}  \left(\frac{256}{27} \right)^{-N}.
	\]
\end{proof}

We say an edge of $\partial\cB_r^\bullet(\cU)$ is \emph{internal} if it does not belong to $\partial \cU= \cU_0$. Let $E_1, \ldots, E_Q$, with $Q \ge 1$,  denote the internal edges of $\partial\cB_{r+1}^\bullet(\cU)$ in clockwise order. Let $(-L', 0)$ denote the left-most vertex of $\partial\cB_r^\bullet(\cU) \cap \partial \cU$ and $(R',0)$ the right-most vertex of $\partial\cB_r^\bullet(\cU) \cap \partial \cU$. Let $L''$ and $R''$ be defined analogously for $r+1$ instead of $r$.

For all $1 \le i \le Q$ the internal edge $E_i$ of $\partial\cB_{r+1}^\bullet(\cU) \cap \partial \cU$ links two vertices at distance $r+1$ from the root vertex of $\cU$ and is incident to a ``downward'' triangle whose third vertex $V_i$ belongs to $\partial\cB_r^\bullet(\cU) \cap \partial \cU$. We set $V_0$ to be the vertex $(-L',0)$, and $V_{Q+1}$ to be the vertex $(R',0)$. For $1 \le j \le Q+1$ let $S_j \ge 0$ denote the number of edges of $\partial\cB_r^\bullet(\cU)$ that lie between $V_{j-1}$ and $V_j$. This way, $S_1 + \ldots + S_{Q+1} = \mathbf{P}_r$ is the number of internal edges of $\partial\cB_r^\bullet(\cU)$.

We prove the next observation by following closely the arguments of the corresponding statement for type I triangulations~\cite[Prop. 11]{zbMATH07144469}. 

\begin{proposition}
	\label{pro:sha4530}
	For all integers $q \ge 1$, $k_1, k_2 \ge 0$, and $s_1, \ldots, s_q \ge 0$ we have
\begin{multline*}
	\Prb{Q = q, S_1=s_1, \ldots, S_{q+1}=s_{q+1}, L'' - L' = k_1 +1, R'' - R' = k_2 +1 \mid \cB_r^\bullet(\cU)} \\
	= \frac{1}{4} \one_{s_1 + \ldots + s_{q+1} = \mathbf{P}_r} \theta(s_1 + k_1) \theta(s_2) \cdots \theta(s_q) \theta(s_{q+1} + k_2).
\end{multline*}
\end{proposition}
\begin{proof}
	For ease of notation, we set
	\[
		\alpha = 64/9 \qquad \text{and} \qquad \rho= 256/27.
	\]
	Let $A$ denote a  simple triangulation with a boundary, with two distinct vertices marked on the boundary that also differ from the root vertex of $A$. Suppose that $\Pr{ \cB_r^\bullet(\cU) = A}$. We let $p$ denote the number of internal edges of $\partial A$. We have to show that the conditional probability
	\begin{align}
		\label{eq:ricktosstudy}
	\Prb{Q = q, S_1=s_1, \ldots, S_{q+1}=s_{q+1}, L'' - L' = k_1 +1, R'' - R' = k_2 +1 \mid \cB_r^\bullet(\cU) = A}
	\end{align}
	for non-negative integers $s_1, \ldots, s_{q+1} \ge 0$ and $k_1, k_2 \ge 0$ such that $s_1 + \ldots + s_{q+1} = p$, is given by
	\[
	\frac{1}{4} \theta(s_1 + k_1) \theta(s_2) \cdots \theta(s_q) \theta(s_{q+1} + k_2).
	\]
	Recall the notation $\sqsubset$ from Lemma~\ref{eq:lemmorty}. The probability in~\eqref{eq:ricktosstudy} equals
	\begin{align}
		\label{eq:exprtod08f}
		\sum_{A'} \Prb{ \cB_{r+1}^\bullet(\cU) = A' \mid \cB_r^\bullet(\cU) = A},
	\end{align}
	with the sum index $A'$ ranging over all simple triangulations with a boundary and two marked vertices on the boundary (distinct from each other, and distinct from the root vertex of $A'$), with the following properties:
	\begin{enumerate}[\qquad i)]
		\item $A \sqsubset A'$.
		\item $\tilde{\partial} A \subset \tilde{\partial} A'$.
		\item There are $k_1+1$ boundary edges between the left-most vertex of $\tilde{\partial}A$ and the left-most vertex of $\tilde{\partial} A'$.
		\item There are $k_2+1$ boundary edges between the right-most vertex of $\tilde{\partial}A$ and the right-most vertex of $\tilde{\partial} A'$.
		\item $\partial A' \setminus \tilde{\partial}A'$ has $q$ edges. Each of these is incident to a ``downward'' triangle with the third vertex on $\partial A \setminus \tilde{\partial} A$, and the configuration of downward triangles is characterized by $s_1, \ldots, s_{q+1}$.
	\end{enumerate}
	Given such a triangulation $A'$, let $N$ denote the number of vertices in $A'$ that are not vertices of $A$ or $\partial \cU$.  It follows by Lemma~\ref{eq:lemmorty} that
	\begin{align*}
		\Pr{ \cB_{r+1}^\bullet(\cU) = A' \mid \cB_r^\bullet(\cU) = A} &= \frac{\Pr{\cB_{r+1}^\bullet(\cU) = A'}}{\Pr{\cB_r^\bullet(\cU) = A}} \\
		&= \alpha^{q-p} \alpha^{-(k_1 + k_2 +2)} \rho^{-N}. 
	\end{align*}

The triangulation $A'$ is determined by $A$, the preceding properties, and simple triangulations  glued into the $q+1$ slots left by the downward triangles. Here we include the slots of degree $2$, which are always merged into a single edge. For $2 \le i  \le q$, the $i$th slot in clockwise direction has degree $s_i + 2$. For $i=1$, the slot has degree $s_1 + k_1 + 2$, and for $i=q+1$ it has degree $s_{q+1} + k_2 + 2$. Let $\cM_i$ denote the triangulation with a boundary that we glue into the $i$th slot, with $\cM_i$ set to a single edge if the slot has degree $2$. Let $\Inn(\cM_i) \ge 0$ denote the number of inner vertices. Then
\[
	N = q-1 + \sum_{i=1}^{q+1} \Inn(\cM_i). 
\]
Let $\tilde{s}_i = s_i$ for $2 \le i \le q$, and $\tilde{s}_1 = s_1 + k_1$, and $\tilde{s}_{q+1} = s_{q+1}+k_2$. Note that $\sum_{i=1}^{q+1}(\tilde{s}_i -1) = p - (q+1) + k_1 + k_2$. By Equation~\eqref{eq:deftheta}, $\theta(k) = \frac{1}{\rho} \alpha^{-k+1} Z(k+2)$, with $Z(2) = 1$ by convention. Hence
\begin{align*}
	 \alpha^{q-p} \alpha^{-(k_1 + k_2 +2)} \rho^{-N} &= \alpha^{-3} \rho^2 \prod_{i=1}^{q+1} \left( \frac{1}{\rho} \alpha^{-(\tilde{s}_i -1)} \rho^{-\Inn(\cM_i)} \right) \\
	 &= \frac{1}{4} \prod_{i=1}^{q+1} \left( \theta(\tilde{s}_i) \frac{ \rho^{-\Inn(\cM_i)} }{Z(\tilde{s}_i +2)} \right).
\end{align*}
Summing over all possible choices of $A'$ is equivalent to summing over all possible choices of triangulations $\cM_1, \ldots, \cM_{q+1}$ with the corresponding boundary lengths. Hence, by Equation~\eqref{eq:zpdef} and the convention $Z(2)=1$, the expression in~\eqref{eq:exprtod08f} simplifies to 
\begin{align*}
	\sum_{A'}  	\Pr{ \cB_{r+1}^\bullet(\cU) = A' \mid \cB_r^\bullet(\cU) = A} &=  \frac{1}{4} \prod_{i=1}^{q+1}\theta(\tilde{s}_i) . 
\end{align*}
This completes the proof.
\end{proof}

As in Equation~\eqref{eq:htotheizzo}, we set
\begin{align*}
	h(j) = 4^{-j} \binom{2j}{j} 
\end{align*}
for all integers $j \ge 0$.

With Proposition~\ref{pro:sha4530} at hand, the proof of the following corollary is identical to the proof of the corresponding result~\cite[Cor. 12]{zbMATH07144469} for the type I case. No adaptions are necessary.
\begin{corollary}
	For all integers $c,k \ge 0$,
	\begin{align*}
		\Prb{S_{Q+1} = c, R'' - R' = k+1 \mid \cB_r^\bullet(\cU)} = \one_{c \le \mathbf{P}_r} \frac{1}{2} \left(1 + h(\mathbf{P}_r-c)\right) \theta(c+k).
	\end{align*}
In particular,
\begin{align*}
	\Prb{ R'' - R' = k+1 \mid \cB_r^\bullet(\cU)} \le \theta([k, \infty[).
\end{align*}
\end{corollary}

\subsection{Distances along the boundary of the type III UHPT}

For each integer $r \ge 1$  let $(-L_r^\cU, 0)$ and $(0, R_r^\cU)$ denote the left-most and right-most vertex in $\partial \cB_r^\bullet(\cU) \cap \partial \cU$. 

\begin{proposition}
	\label{pro:audiom}
	For each $\epsilon>0$ there exists $K > 0$ such that
	\[
		\sup_{r \ge 1} \Pr{L_r^\cU \ge K r^2} \le \epsilon
	\]
	and
	\[
		\sup_{r \ge 1} \Pr{R_r^\cU \ge K r^2} \le \epsilon.
	\]
	For each integer $m \ge 1$ let $T_m := \min\{ r \ge 1 \mid R_r^\cU > m\}$. There exists $K'>0$ such that for all $m,j \ge 1$
	\[
		\Pr{ R_{T_m}^\cU - m > j} \le K' \sqrt{\frac{m}{m+j}}.
	\]
\end{proposition}

Having the results of Section~\ref{sec:layers} at hand, the proof of Proposition~\ref{pro:audiom} is identical to the proof of the corresponding result~\cite[Prop. 13]{zbMATH07144469} for type I triangulations.  The same goes for the following result, which is the type III version of~~\cite[Prop. 14]{zbMATH07144469}:

\begin{proposition}
	Let $\epsilon>0$ and $A>0$ be given. There exists an integer $K>0$ such that for each $r \ge 1$\[
		\Prb{ \min_{\substack{0 \le i \le Ar^2 \\ j \ge K r^2}}  d_{\cU}((i,0), (j,0)) \le r} \le \epsilon.
	\]
\end{proposition}
Here $d_\cU$ denotes the graph distance on the vertex set of the type III UHPT $\cU$.

\subsection{Distances along the boundary of the type III LHPT}

For each $i \in \ndZ$ and each integer $r \ge 1$ we define the \emph{left-most geodesic} from $(i,0)$ in $\cL$ as the infinite geodesic path $\omega$ in $\cL$ starting from $\omega(0) = (i,0)$ as follows: at each step $n \ge 0$, the path walks from the vertex $\omega(n) \in \cL_n$ to $\cL_{n+1}$ along the left-most edge that links $\omega(n)$ to $\cL_{n+1}$. Thus, the first $r$ edges on this path form a geodesic from $\omega(0)$ to $\cL_r$.  

We  define the left-most geodesic in $\cU$ from $(i,r)$  to $\partial \cU$ in the same way. This way, for $1 \le i < j$, the left-most geodesics from $(i,r)$ to $\partial \cU$ and from $(j,r)$ to $\partial \cU$ coalesce before or when hitting $\partial \cU$, if and only if none of the trees $\Gamma_{(i,r)}, \ldots, \Gamma_{(j-1,r)}$ has height $r$.

\begin{proposition}
	\label{pro:p1515}
	For each $\epsilon>0$ there exists an integer $K \ge 1$ such that for each integer $r \ge 1$
	\[
		\Prb{\min_{|j| \ge K r^2} d_{\cL}( (0,0), (j,0) ) \ge r} \ge 1- \epsilon.
	\]
	In particular, with $K' = 4K$ we have for all $r \ge 1$
	\[
			\Prb{\min_{|j| \ge 2K' r^2} \,\, \min_{-K'r^2 \le i \le K' r^2} d_{\cL}( (i,0), (j,0) ) \ge r} \ge 1- 2\epsilon.
	\]
\end{proposition}
Here $d_{\cL}$ denotes the graph distances in the type III LHPT $\cL$. With the results of the preceding subsections at hand, the proof of Proposition~\ref{pro:p1515} is identical to the proof of the corresponding result~\cite[Prop. 15]{zbMATH07144469} in the type I case. The same goes for the next proposition, whose proof in the type I case~\cite[Prop. 16]{zbMATH07144469} requires no adaption.

\begin{proposition}
	Let $\delta, \gamma >0 $. There exists an integer $A \ge 1$ such that for all sufficiently large $n$ the following statements hold with probability at least $1 - \delta$:
	\begin{enumerate}
		\item For all $i \in \{-n+1, \ldots, n\}$, the left-most geodesic starting from $(i,0)$ coalesces with the left-most geodesic starting from $(-n + \lfloor 2 \ell n / A \rfloor, 0)$, for some $0 \le \ell \le A$, before hitting $\cL_{\lfloor \gamma \sqrt{n} \rfloor}$.
		\item For all $i,j \in \{-n+1, \ldots, n\}$, with $i<j$, there exists a path from $(i,0)$ to $(j,0)$ that stays in $\cL_{[0, \lfloor \gamma \sqrt{n} \rfloor]}$ and has length smaller than
		\[
			\left( \lfloor \frac{A(j-i)}{2n} \rfloor + 2\right)(1 + 2 \gamma \sqrt{n}).
		\]
	\end{enumerate}
\end{proposition}

Recall that $L_n$ denotes the length of the top cycle of $B_n^\bullet(\cT_{\infty}^{(0)})$. We let $u_0^{(n)}$ denote a uniformly at random chosen vertex from the top cycle $\partial^* B_n^\bullet(\cT_{\infty}^{(0)})$. We let $u_1^{(n)}, \ldots, u_{L_n-1}^{(n)}$ denote the remaining vertices of $\partial^* B_n^\bullet(\cT_{\infty}^{(0)})$, listed in clock-wise order starting from $u_0^{(n)}$. It will be notationally convenient to treat the indices modulo $L_n$, so that $u_{i + L_n}^{(n)} = u_i$ for all $i \in \ndZ$.

We prove the next statement by following closely the arguments of the corresponding result~\cite[Prop. 17]{zbMATH07144469} for type I triangulations.

\begin{proposition}
	\label{pro:giveusana}
	Let $\gamma \in ]0, 1/2[$ and $\delta>0$ be given. For all integers $A \ge 1$ let $H_{n,A}$ denote the event that each left-most geodesic  from some vertex in $\partial^* B_n^\bullet(\cT_{\infty}^{(0)})$ to the root coalesces before time $\lfloor \gamma n\rfloor$ with one of the left-most geodesics to the root starting from $u_{\lfloor k n^2/A \rfloor}^{(n)}$ for an integer $k$ satisfying $0 \le k \le \lfloor n^{-2} L_n A\rfloor$. Then there exists a constant $A \ge 1$ such that for all large enough $n$
	\[
		\Pr{H_{n,A}} \ge 1 - \delta.
	\]
\end{proposition}
\begin{proof}
	We use the notation from Section~\ref{sec:comparision}. Without loss of generality we may assume that the first tree in $\tilde{\cF}^{(0)}_{n - \lfloor \gamma n\rfloor, n}$ is the one whose root vertex corresponds  to the edge from $u_0^{(n)}$ to $u_1^{(n)}$. We write $\tilde{\cF}^{(0)}_{n - \lfloor \gamma n\rfloor, n} = (\tau_1^{(n)}, \ldots, \tau_{L_n}^{(n)})$. By the discussion on coalescence of geodesics at the end of Section~\ref{sec:app}, it follows that for $1 \le i < j \le n$ the left-most geodesics to the root from the vertex $u_i^{(n)}$ and from the vertex $u_j^{(n)}$ coalesce before or at time $\lfloor \gamma n \rfloor$ if the trees $\tau_{i+1}^{(n)}, \ldots, \tau_j^{(n)}$ have height strictly smaller than $\lfloor \gamma n \rfloor$. Thus, in order for $H_{n,A}$ to hold it is sufficient that for any integer $i$ with $1 \le i \le L_n$ there exists an integer $k$ with $0 \le k \le \lfloor n^{-2} L_n A \rfloor$ such that for all integers $j$ with $\min(i, \lfloor k n^2/A \rfloor) < j \le  \max(i, \lfloor k n^2/A \rfloor)$ the tree $\tau_j^{(n)}$ has height strictly smaller than $\lfloor \gamma n \rfloor$. We let $H_{n,A}'$ denote this event. By Lemma~\ref{le:analem}, there exists $a>0$ such that
	\begin{align}
		\label{eq:jim}
		a n^2 \le L_n \le a^{-1} n^2
	\end{align}
	and
	\begin{align}
		\label{eq:cho}
		a n^2 \le L_{n- \lfloor \gamma n\rfloor} \le a^{-1} n^2
	\end{align}
hold simultaneously with probability at least $1 - \delta / 2$ for all large enough $n$. Hence, in order to verify (for a suitable choice of $A$) that $\Pr{H_{n,A}} \ge 1 - \delta$, it suffices to show that~\eqref{eq:jim},~\eqref{eq:cho}, and the complement of $H_{n,A}'$ hold with probability at most~$\delta/2$.

 By Proposition~\ref{pro:comparison} we know that for each $a' \in]0,1[$ there exists $C_1>0$ such that for all large enough integers $r$, all integers $s >r$,   all integers $p,q \in \{\lfloor a' r^2 \rfloor +1, \ldots, \lfloor a'^{-1} r^2 \rfloor \}$, and all forests $\cF \in \mathbb{F}_{p,q,s-r}''$ we have
\[
\Prb{ \widetilde{\cF}_{r,s}^{(0)} = \cF} \le C_1 \Prb{ ([\tau_1]_{s-r}, \ldots, [\tau_{N_r^{(a')}}]_{s-r}) = \cF  },
\]
with $N_r^{(a')}$  uniformly distributed over $\{\lfloor a' r^2 \rfloor +1, \ldots, \lfloor a'^{-1} r^2 \rfloor \}$, and $(\tau_i)_{i \ge 1}$ denoting i.i.d. $\theta$-Bienaym\'{e}--Galton--Watson trees, and $[ \cdot ]_{s-r}$ denoting truncation at height $s-r$.

Using $r= n - \lfloor \gamma n \rfloor$, $s = n$, $a' = a(1-\gamma)^2$, we obtain the following: In order to show that (for suitable $A$)~\eqref{eq:jim},~\eqref{eq:cho}, and the complement of the event $H_{n,A}'$ hold simultaneously with probability at most~$\delta/2$, it suffices to show that uniformly for all integers $\ell$ with $an^2 \le \ell \le a^{-1} n^2$ it holds with probability at least $1 - \delta / (2 C_1)$ that the forest $(\tau_1, \ldots, \tau_\ell)$ is ``$n$-good'' in the sense that for all  $1 \le i \le m$ there exists an integer $k$ with $0 \le k \le \lfloor n^{-2} \ell A \rfloor$ such that for all integers $j$ with $\min(i, \lfloor k n^2/A \rfloor) < j \le  \max(i, \lfloor k n^2/A \rfloor)$ the tree $\tau_j$ has  height strictly less than $\lfloor \gamma n \rfloor$.

To this end, let $U_1 < \ldots < U_m$ denote the indices $U$ in $\{1, \ldots, \ell\}$ such that the height of $\tau_U$ is at least $\lfloor \gamma n \rfloor$. Note that by Equation~\eqref{eq:trivbound00435}, the probability for a $\theta$-Bienaym\'{e}--Galton--Watson tree to have height at least $\lfloor \gamma n \rfloor$ is equal to $\frac{1}{(\lfloor \gamma n \rfloor + 1)^2}$. 

Suppose that for infinitely many $n$ there exists an integer $\ell_n$ satisfying $an^2 \le \ell_n \le a^{-1} n^2$ such that the forest $(\tau_1, \ldots, \tau_{\ell_n})$ is $n$-good with probability strictly less than $1 - \delta / (2 C_1)$. By restricting to subsequences, we may without loss of generality assume that $\ell_n \sim c n^2$ for some constant $c$ satisfying $a \le c \le a^{-1}$. For all $t \ge 0$ we set
\[
	N_t := \# \{1 \le i \le m \mid U_i \le t\}.
\]
Then $(N_{\lfloor t \ell_n / c \rfloor})_{0 \le t \le c}$ converges in distribution in the Skorokhod sense to a Poisson process with parameter $\gamma^{-2}$. Setting $U_0 = 0$ and $U_{m+1} = \ell_n$, it follows that we may choose a constant $\eta>0$ that only depends on $\delta / (2 C_1)$ and  $a$ (since $c$ satisfies $a \le c \le a^{-1}$) such that for all sufficiently large $n$
\[
	U_{i+1} - U_i > \eta n^2, \qquad \text{for all $0 \le i \le m$}
\]
holds with probability at least $1 - \delta / (2 C_1)$. But this is only possible if $2/A > \eta$.

Thus, we may choose $A$ sufficiently large such that uniformly for all integers $\ell$ with $an^2 \le \ell \le a^{-1} n^2$ it holds with probability at least $1 - \delta / (2 C_1)$ that the forest $(\tau_1, \ldots, \tau_\ell)$ is ``$n$-good''. This completes the proof.
\end{proof}

\section{First passage percolation on the type III UIPT}

Throughout this section we let $\iota$ denote a random variable with  finite exponential moments and that there exists a constant $\kappa>0$ such that $\Pr{\iota \ge \kappa} = 1$. Dividing $\iota$ (and the lower bound $\kappa$) by $2 \Ex{\iota}$, we may without loss of generality assume that
\begin{align}
	\label{eq:iotaexp}
	\Ex{\iota} = 1/2.
\end{align}
This is done purely for notational convenience. The following  deviation bound for one-dimensional random walk may be found in most textbooks on the subject.
\begin{proposition}
	Let  $\iota_1,\ldots, \iota_n$ denote independent copies of $\iota$. There exists a constant $c>0$ such that for all sufficiently small $\lambda>0$ and all $n \ge 1$ and all $x \ge n/2$
\begin{align}
	\label{eq:iotabound}
	\Pr{\iota_1 + \ldots + \iota_n \ge x} \le \exp\left(c n \lambda^2 - \lambda(x - n/2)\right).
\end{align}
\end{proposition}

\subsection{Subadditivity in the type III LHPT}

 We let $d_{\mathrm{fpp}}^\cL$ denote the $\iota$-first-passage percolation distance on the type III LHPT $\cL$. Recall that $\cL_r = \ndZ \times \{-r\}$ for all integers $r \ge 0$. We let $\rho=(0,0)$ denote the root-vertex of $\cL$.

We prove the following result analogously to~\cite[Prop. 18]{zbMATH07144469}.
\begin{proposition}
	\label{pro:kingman}
	There exists a constant $c_{\mathrm{fpp}}^\cT \in [\kappa,1] $ such that
	\[
		r^{-1} d_{\mathrm{fpp}}^\cL(\rho, \cL_r) \convas c_{\mathrm{fpp}}^\cT
	\]
	as $r \to \infty$.
\end{proposition}
\begin{proof}
	For all integers $0 \le m < n$ let $\cL_{[m,n]}$ denote the submap of $\cL$ that lies in the strip $\ndR \times [-n, -m]$.  We let $d_{\mathrm{fpp}}^{\cL_{[m,n]}}$ denote the first-passage percolation distance on the vertex set of $\cL_{[m,n]}$. That is, the length between points is the minimal sum of weights along joining paths that stay within $\cL_{[m,n]}$. This way,
	\[
		d_{\mathrm{fpp}}^\cL(u,v) \le d_{\mathrm{fpp}}^{\cL_{[m,n]}}(u,v)
	\]
	for all vertices $u$, $v$ of $\cL_{[m,n]}$. 
	
	Given integers $n,m \ge 1$, we let $x_m$ denote the left-most vertex of $\cL_m$ that satisfies
	\[
		d_{\mathrm{fpp}}^\cL(\rho,\cL_m) = d_{\mathrm{fpp}}^\cL(\rho,x_m).
	\]
	(Note that since $\iota \ge \kappa >0$ and since $\cL$ is locally finite, there actually exists a left-most vertex with that property. Otherwise, there would be an infinite number of  vertices  at $d_{\mathrm{fpp}}^\cL$-distance at most $m$ from the root. This  would entail that the $(m/\kappa)$-graph-distance neighbourhood of the root would be infinite. But locally finite graphs have the property, that any ball with finite radius has a finite number of vertices.)
	Then
	\begin{align*}
		d_{\mathrm{fpp}}^\cL(\rho,\cL_{m+n}) &\le d_{\mathrm{fpp}}^\cL(\rho,x_m) + d_{\mathrm{fpp}}^{\cL} (x_m,\cL_{m+n}) \\
		&\le d_{\mathrm{fpp}}^\cL(\rho,\cL_m) + d_{\mathrm{fpp}}^{\cL_{[m, m+n]}} (x_m,\cL_{m+n}).
	\end{align*}
	Note that $x_m$ is already determined by $\cL_{[0,m]}$, since no path starting in $\rho$ can leave $\cL_{[0,m]}$ without passing through $\cL_m$. By construction of $\cL$, in particular the independence of its layers, it follows that
	\begin{align*}
		d_{\mathrm{fpp}}^{\cL_{[m, m+n]}} (x_m,\cL_{m+n}) \eqdist d_{\mathrm{fpp}}^\cL(\rho,\cL_{n})
	\end{align*}
and that $d_{\mathrm{fpp}}^{\cL_{[m, m+n]}} (x_m,\cL_{m+n})$ is independent from $\cL_{[0,m]}$. Setting  $x_0 := \rho$, this allows us to apply Ligget's  version~\cite{zbMATH03927992} of Kingman's subadditive ergodic theorem to the triangular array $(X_{m,n})_{0 \le m <n}$ with $X_{m,n} = d_{\mathrm{fpp}}^{\cL_{[m,n]}} (x_m,\cL_{n})$ for all $0 \le m <n$. It follows that the limit
\[
	X := \lim_{n \to \infty} n^{-1}X_{0,n} = \lim_{r \to \infty} r^{-1} d_{\mathrm{fpp}}^\cL(\rho, \cL_r)
\]
exists almost surely. It is clear that $X \ge \kappa$ almost surely. Kolmogorov's zero–one law readily entails that $X$ is almost surely constant. Furthermore, Inequality~\eqref{eq:iotabound} entails that as $r$ becomes large, the sum of the weights corresponding to the $r$ edges on the left-most geodesic from the root of $\cL_r$ to $\rho$ is less than $r$ with a probability that tends exponentially fast to $1$. This entails that the constant $X$ satisfies $X \le 1$.
\end{proof}

\subsection{From the LHPT to the UIPT}

We prove the following result analogously to~\cite[Prop. 19]{zbMATH07144469}, with some adaptions since we assume $\iota$ to have finite exponential moments instead of an upper bound.

\begin{proposition}
	\label{pro:ring2ring}
	Let $0<\epsilon<1$ and $\delta>0$ be given. There exists $0 < \eta < 1/2$ such that for all large enough $n$ it holds with probability at least $1 - \delta$ that
	\[
		(1 - \epsilon) c_{\mathrm{fpp}}^\cT \eta n \le d_{\mathrm{fpp}}^{\cT_\infty^{(0)}}(v, \partial^* B_{n - \lfloor \eta n \rfloor}^\bullet(\cT_\infty^{(0)})) \le (1+ \epsilon) c_{\mathrm{fpp}}^\cT \eta n
	\]
	for all vertices $v \in \partial^*B_{n}^\bullet(\cT_\infty^{(0)})$.
\end{proposition}
\begin{proof}
	Recall that $L_n$ denotes the length of the top cycle of $B_n^\bullet(\cT_{\infty}^{(0)})$, and that $u_0^{(n)}$ denotes a uniformly at random chosen vertex from  $\partial^* B_n^\bullet(\cT_{\infty}^{(0)})$. As before, we let $u_1^{(n)}, \ldots, u_{L_n-1}^{(n)}$ denote the remaining vertices of $\partial^* B_n^\bullet(\cT_{\infty}^{(0)})$, listed in clock-wise order starting from $u_0^{(n)}$. We treat the indices modulo $L_n$, so that $u_{i + L_n}^{(n)} = u_i$ for all $i \in \ndZ$.
	
	By Lemma~\ref{le:analem}, there exists  a constant $0<a<1/2$ (depending on $\delta$) such that for all $n \ge 2$ and $0 < \eta < 1/2$, the event
	\[
		\cE_n(\eta) := \{ \lfloor a n^2 \rfloor +1 \le L_n \le \lfloor a^{-1} n^2 \rfloor \} \cap \{ \lfloor a n^2 \rfloor +1 \le L_{n- \lfloor \eta n \rfloor } \le \lfloor a^{-1} n^2 \rfloor \}
	\]
	holds with probability at least $1 - \delta / 4$.
	
	Given $0 < \eta < 1/2$ and $j \in \ndZ$, let $\cH_{n,j}(\eta)$ denote the intersection of the event $\cE_n(\eta)$ with the event that the leftmost geodesics starting from $u^{(n)}_{j- \lfloor a n^2 /4 \rfloor }$ and $u^{(n)}_{j + \lfloor a n^2 / 4 \rfloor }$ do not coalesce before hitting $B_{n- \lfloor \eta n \rfloor }^\bullet(\cT_{\infty}^{(0)})$.  By Corollary~\ref{co:jimmychoo} (and making the constants in the bounds within its proof explicit) it follows that there is a constant $C_4>0$ such that for all large enough $n$ and each $j \in \ndZ$ that
	\begin{align}
		\label{eq:likesmallyeah0}
		\Pr{ \cE_n(\eta) \cap \left(\cH_{n,j}(\eta) \right)^c} \le C_4 a^{-2} \exp(-a / (4 \eta^2)).
	\end{align}
	Indeed,  the bound in Corollary~\ref{co:jimmychoo} is given by $O(1)C_1 \exp(-a / 4 \eta^2)$, with $O(1)$ denoting a bounded term that does not depend on anything. Furthermore, the constant $C_1$ from Proposition~\ref{pro:comparison} is of the form $(1/a - a)C_0 C_2$ with $C_2 = O(1/a)$, and $C_0$ from Lemma~\ref{le:analem} not depending on anything. Hence we arrive at a bound of the form $O(1) a^{-2} \exp(-a / (4 \eta^2))$. Note that this bound becomes worse the smaller we take $a$. This makes sense, since the smaller we take $a$, the less trees are between $u^{(n)}_{j- \lfloor a n^2 /4 \rfloor }$ and $u^{(n)}_{j + \lfloor a n^2 / 4 \rfloor }$, and hence the more likely it gets that none of them reach height $\eta n$.
	
	On the event $\cH_{n,j}(\eta)$, we define $\cG_{j}^{(\eta)}$ as the subregion of $B_n^\bullet(\cT_{\infty}^{(0)}) \setminus B_{n - \lfloor \eta n \rfloor }^\bullet(\cT_{\infty}^{(0)})$ containing $u_j^{(n)}$ that is bounded on one side by the leftmost geodesic from $u^{(n)}_{j- \lfloor a n^2 /4 \rfloor }$ and on the other side by the left-most geodesic from $u^{(n)}_{j + \lfloor a n^2 / 4 \rfloor }$. Moreover, let $\partial_\ell \cG_{j}^{(\eta)}$ denote the part of the boundary of $\cG_{j}^{(\eta)}$ that is contained in the union of these two geodesics.
	
	Let $\cA_{n,j}(\eta)$ denote the intersection of $\cH_{n,j}(\eta)$ with the event, where for some integer $i$ with $j - a n^2 / 16 \le i \le j + a n^2 / 16$, there is a path from $u_i^{(n)}$ to $\partial_\ell \cG_j^{(n)}(\eta)$ that stays in $B_n^\bullet(\cT_{\infty}^{(0)}) \setminus B_{n - \lfloor \eta n \rfloor }^\bullet(\cT_{\infty}^{(0)})$ and has length smaller than $4 \eta n / \kappa$. We are going to show that if we take $\eta$ sufficiently small, then for all large enough $n$ and all $j \in \ndZ$
	\begin{align}
		\label{eq:likesmallyeah}
		\lim_{\eta \downarrow 0} \limsup_{n \to \infty} \Pr{\cA_{n,j}(\eta)} = 0.
	\end{align}
	To this end, it suffices to consider the case $j=0$, since $u_j^{(0)}$ is a uniformly selected vertex of $\partial^*B_n^\bullet(\cT_\infty^{(0)})$. For each $i \in \ndZ$ we let $\cJ_i^{(n, \lfloor \eta n \rfloor )}$ denote the tree of the skeleton of $B_n^\bullet(\cT_{\infty}^{(0)}) \setminus B_{n - \lfloor \eta n \rfloor }^\bullet(\cT_{\infty}^{(0)})$ corresponding to the edge from $u_{i-1}^{(n)}$ to $u_i^{(n)}$. On the event $\cH_{n,0}(\eta)$, the region $\cG_0^{(n)}(\eta)$ is determined as planar map by the trees $\cJ_i^{(n, \lfloor \eta n \rfloor )}$ for $- \lfloor a n^2 /4 \rfloor < i \le \lfloor a n^2 / 4 \rfloor$, and by the Boltzmann triangulations used to fill the slots of the vertices of these trees with height strictly less than $\lfloor \eta n \rfloor$. 
	
	By Proposition~\ref{pro:comparison}, it follows that the probability for the event $\cA_{n,j}(\eta)$ can also be bounded by a $O(1/a^2)$ multiple of the similar event for the type III LHPT. That is, the event that in the lower half-plane model there is a path from $(i,0)$ for some $-an^2/16 \le i \le a n^2/16$ to the left-most geodesic from $(\lfloor a n^2/4 \rfloor, 0)$ or $(-\lfloor a n^2/4 \rfloor, 0)$ with length at most $4 \eta n / \kappa$. If there is such a path, then $(i,0)$ is at graph distance at most $8 \eta n / \kappa$ from $(\lfloor a n^2/4 \rfloor, 0)$ or $(-\lfloor a n^2/4 \rfloor, 0)$. By Proposition~\ref{pro:p1515}, it follows that taking $\eta$ small (with respect to $a$) this probability can be made arbitrarily small, uniformly for all sufficiently large $n$. This verifies Equation~\eqref{eq:likesmallyeah}.
	
	Using~\eqref{eq:likesmallyeah0} and~\eqref{eq:likesmallyeah} we may take $\eta$ small enough (depending on $a$) such that for large enough $n$ the probabilities for $\cE_n(\eta) \cap \left(\cH_{n,j}(\eta)\right)^c $ and $\cA_{n,j}(\eta)$ are both smaller than $a^2 \delta / 100$, uniformly for all $j$. We set
	\[
		\cB_n(\eta) := \left( \bigcap_{k=0}^{\lfloor 9a^{-2} \rfloor} \cH_{n, k \lfloor a n^2 / 8 \rfloor } (\eta)   \right) \cap  \left( \bigcap_{k=0}^{\lfloor 9a^{-2} \rfloor} \left(\cA_{n, k \lfloor a n^2 / 8 \rfloor }(\eta)\right)^c   \right).
	\]
	Hence, for large enough $n$
	\begin{align*}
		\Pr{	\cB_n(\eta)^c}  &\le \Pr{ (\cE_n(\eta))^c} + \sum_{k=0}^{\lfloor 9 a^{-2} \rfloor} \Prb{\cE_n(\eta) \cap \left(\cH_{n,k \lfloor a n^2 / 8 \rfloor}(\eta)\right)^c } \\&\quad+ \sum_{k=0}^{\lfloor 9 a^{-2} \rfloor} \Pr{\cA_{n, k \lfloor a n^2 / 8 \rfloor }(\eta)  } \\
		&\le \delta/4 + (9 a^{-2} + 1) a^2 \delta / 100  + (9a^{-2} + 1) a^2 \delta / 100 \\
		&\le \delta / 2.
	\end{align*}

	We let $d_{\mathrm{fpp}}^{(n)}$ denote the $\iota$-first-passage percolation metric that only takes into account paths that stay in $B_n^\bullet( \cT_\infty^{(0)})$. 
	For each $i \in \ndZ$ we let $\cD_i^{(n)}$ denote the event that
	\[
		 d_{\mathrm{fpp}}^ {(n)}(u_i^{(n)}, \partial^* B_{n - \lfloor\eta n \rfloor}^\bullet(\cT_\infty^{(n)})) \in [(1-\epsilon) c_{\mathrm{fpp}}^\cT \eta n, (1+\epsilon) c_{\mathrm{fpp}}^\cT \eta n].
	\]

	Let $k$ be an integer satisfying $0 \le k \le 9a^{-2}$. Clearly,
	\begin{align}
		\label{eq:bboy1}
		\Prb{ \cB_n(\eta) \cap ( \cD_i^{(n)})^c } \le \Prb{ \cH_{n, k \lfloor a n^2 / 8 \rfloor } \cap (\cA_{n, k \lfloor a n^2 / 8 \rfloor})^c \cap (\cD_i^{(n)})^c }.
	\end{align}
 Suppose that $\left(\cA_{n, k \lfloor a n^2 / 8 \rfloor }(\eta)\right)^c $ and $\cH_{n, k \lfloor a n^2 / 8 \rfloor }(\eta)$ both hold. Let $i$ be an integer satisfying 
 \[
  k \lfloor a n^2 / 8 \rfloor - a n^2/16 \le i \le  k \lfloor a n^2 / 8 \rfloor + a	 n^2 / 16.
  \] The minimal sum of link-weights along a path from $u_i^{(n)}$ to $\partial^* B_{n -\lfloor \eta n \rfloor}^\bullet(\cT_\infty^{(0)})$ that stays in $\partial^* B_{n }^\bullet(\cT_\infty^{(0)}) \setminus \partial^* B_{n -\lfloor \eta n \rfloor}^\bullet(\cT_\infty^{(0)})$ is with high probability determined by the region $\cG^{(n)}_{k \lfloor a n^2 / 8\rfloor}(\eta)$ and the corresponding link-weights. Indeed, if such a path with minimal sum of weights would hit $\partial_{\ell} \cG_{k \lfloor a n^2 /8\rfloor}^{(n)}(\eta)$, then this sum of weights would be at least $\kappa (4 \eta n /\kappa) = 4 \eta n$. However, we can reach $\partial^* B_{n -\lfloor \eta n \rfloor}^\bullet(\cT_\infty^{(0)})$ from $u_i^{(n)}$ via the left-most geodesic of path length $\lfloor \eta n\rfloor$, and by Inequality~\eqref{eq:iotabound} the sum of weights along that path is hence also at most $\lfloor \eta n\rfloor$ for all integers $i \in [k \lfloor a n^2 / 8 \rfloor - a n^2/16 , k \lfloor a n^2 / 8 \rfloor + a n^2 / 16]$ with probability at least $1 - an^2  \exp(-c_\tau \eta n)$ for some $c_\tau > 0$. Furthermore, on the event $\cH_{n, k \lfloor a n^2 / 8 \rfloor }$ the region $\cG^{(n)}_{k \lfloor a n^2 / 8\rfloor}(\eta)$ is determined by the trees $\cJ_m^{(n, \lfloor \eta n \rfloor)}$ for  $m \in [k \lfloor a n^2 / 8 \rfloor - \lfloor a n^2 / 4 \rfloor,  k \lfloor a n^2 / 8 \rfloor + \lfloor a n^2 / 4 \rfloor[$. By Proposition~\ref{pro:comparison}, it follows that for a constant $C_1 = O(1/a^2)$  
	\begin{align}
		\label{eq:bboy2}
		\Prb{ \cH_{n, k \lfloor a n^2 / 8 \rfloor } \cap (\cA_{n, k \lfloor a n^2 / 8 \rfloor})^c \cap (\cD_i^{(n)})^c } \le an^2  \exp(-c_\tau \eta n) + C_1 \Pr{F_n},
	\end{align}
with $F_n$ denoting the event that
\[
	d_{\mathrm{fpp}}^\cL( (0,0), \cL_{\lfloor \eta n \rfloor}) \notin [(1-\epsilon)c_{\mathrm{fpp}}^\cT \eta n, (1+\epsilon)c_{\mathrm{fpp}}^\cT \eta n].
\]

We let $A$ denote the integer from Proposition~\ref{pro:giveusana} where we replace $\gamma$ by $\epsilon c_{\mathrm{fpp}}^\cT \eta n / 2$ and $\delta$ by $\delta/8$. Using~Proposition~\ref{pro:giveusana}, it follows that the event
\begin{multline}
	\label{eq:oomph0}
	\cB_n(\eta) \cap \{ d_{\mathrm{fpp}}^{(n)}(v, \partial^* B_{n - \lfloor \eta n\rfloor}^\bullet(\cT_{\infty}^{(0)})) \notin [(1- 2 \epsilon)c_{\mathrm{fpp}}^\cT \eta n, (1+ 2 \epsilon)c_{\mathrm{fpp}}^\cT \eta n] \\ \text{ for some $v \in \partial^*B_n^\bullet(\cT_\infty^{(0)})$ }  \}
\end{multline}
is contained in the event
\begin{multline}
	\label{eq:oomph}
		\cB_n(\eta) \cap \{ d_{\mathrm{fpp}}^{(n)}(u^{(n)}_{\lfloor j n^2 / A \rfloor}, \partial^* B_{n - \lfloor \eta n\rfloor}^\bullet(\cT_{\infty}^{(0)})) \notin [(1-  \epsilon)c_{\mathrm{fpp}}^\cT \eta n, (1+  \epsilon)c_{\mathrm{fpp}}^\cT \eta n] \\ \text{ for some integer $0 \le j \le a^{-1}A$ }  \}
\end{multline}
except possibly on an event with probability at most $\delta/4$. To see this, note that if $\cB_n(\eta)$ holds (and hence $\cE_n(\eta)$ holds) and we discard the set set of probability at most $\delta/8$ considered in Proposition~\ref{pro:giveusana}, then for any vertex $v$ in $\partial^*B_n^\bullet(\cT_\infty^{(0)})$ we may walk along its left-most geodesic until it coalesces, after at most $\epsilon c_{\mathrm{fpp}}^\cT \eta n/2$ steps, with the left-most geodesic from a vertex of the form $u^{(n)}_{\lfloor j n^2 / A \rfloor}$ for an integer $0 \le j \le a^{-1}A$. Hence $v$ is at graph distance at most $\epsilon c_{\mathrm{fpp}}^\cT \eta n$ from  $u^{(n)}_{\lfloor j n^2 / A \rfloor}$, but we may additionally use Inequality~\eqref{eq:iotabound} to bound the probability that the sums of weights along any initial segment with length in $\{1, \ldots, \lfloor \epsilon c_{\mathrm{fpp}}^\cT \eta n/2 \rfloor\}$ of any of the at most $a^{-1}n^2$  left-most geodesics become becomes larger than   $\epsilon c_{\mathrm{fpp}}^\cT \eta n/2$. Specifically, with $\iota_1, \iota_2, \ldots$ denoting independent copies of $\iota$,  the bound
\[
	a^{-1} n^2 \sum_{\ell=1}^{\lfloor \epsilon c_{\mathrm{fpp}}^{\cT} \eta n / 2 \rfloor} \Pr{\iota_1 + \ldots +\iota_\ell \ge \epsilon c_{\mathrm{fpp}}^\cT \eta n/2}
\]
tends to zero as $n \to \infty$ by Inequality~\eqref{eq:iotabound}. Hence, we additionally have that $v$ has $d_{\mathrm{fpp}}^{(n)}$-distance smaller  than $u^{(n)}_{\lfloor j n^2 / A \rfloor}$ except on an event with probability at most $\delta/8$ for large enough $n$. 

Now, by Inequality~\eqref{eq:bboy1}, Inequality~\eqref{eq:bboy2}, Proposition~\ref{pro:kingman}, it follows that for sufficiently large $n$ it holds for all $i \in \{0, \ldots, \lfloor a^{-1} n^2 \rfloor \}$ that
\[
	\Prb{ \cB_n(\eta) \cap ( \cD_i^{(n)})^c } \le a \delta/ (4(A+1)).
\]
Hence, for large enough $n$, the probability of the event~\eqref{eq:oomph} is bounded by
\[
	\sum_{j=0}^{\lfloor a^{-1} A \rfloor} \Prb{ \cB_n(\eta) \cap ( \cD_{\lfloor j n^2 / A \rfloor}^{(n)})^c } \le ( \lfloor a^{-1} A \rfloor + 1) \frac{a \delta}{4(A+1)} \le \delta/4.
\]
Since  $\Pr{ \cB_n(\eta)^c} \le \delta / 2$,  and since the event~\eqref{eq:oomph0} is contained in the event~\eqref{eq:oomph} except possibly on an event with probability at most $\delta/4$, it follows that
\begin{multline*}
	\Pr{ d_{\mathrm{fpp}}^{(n)}(v, \partial^* B_{n - \lfloor \eta n\rfloor}^\bullet(\cT_{\infty}^{(0)})) \in [(1- 2 \epsilon)c_{\mathrm{fpp}}^\cT \eta n, (1+ 2 \epsilon)c_{\mathrm{fpp}}^\cT \eta n] \\ \text{ for all $v \in \partial^*B_n^\bullet(\cT_\infty^{(0)})$ }} \ge 1 - \delta.
\end{multline*}
We may replace $d_{\mathrm{fpp}}^{(n)}$ by $d_{\mathrm{fpp}}^{\cT_\infty^{(0)}}$ in the last bound, since $d_{\mathrm{fpp}}^{\cT_\infty^{(0)}} \le d_{\mathrm{fpp}}^{(n)}$  and for all $v \in \partial^* B_n^\bullet(\cT_\infty^{(0)})$
\[
	d_{\mathrm{fpp}}^{\cT_\infty^{(0)}}(v, \partial^* B_{n - \lfloor \eta n \rfloor }^\bullet(\cT_\infty^{(0)})) \ge \min_{v' \in \partial^* B_n^\bullet(\cT_\infty^{(0)})}  d_{\mathrm{fpp}}^{(n)}(v', \partial^* B_{n - \lfloor \eta n \rfloor }^\bullet(\cT_\infty^{(0)})).
\]
Hence the proof is complete.
\end{proof}

We prove the next result following closely the arguments of~\cite[Prop. 20]{zbMATH07144469}, with some adaptions due to  the class of link-weights under consideration.

\begin{proposition}
	\label{pro:yoyo}
	Given $0 < \epsilon < 1$, we have
	\[
		\Pr{ (c_{\mathrm{fpp}}^\cT - \epsilon) n  \le  d_{\mathrm{fpp}}^{\cT_\infty^{(0)}}(\rho, v) \le (c_{\mathrm{fpp}}^\cT + \epsilon) n   \,\,\text{ for all $v \in \partial^* B_n^\bullet(\cT_\infty^{(0)}) $} } \to 1 
	\]
	as $n \to \infty$. 
\end{proposition}
\begin{proof}
	Let $0< \delta < \epsilon / (4 | \log(\epsilon / 16)|)$ be given. By Proposition~\ref{pro:ring2ring}, there exists $0 < \eta < 1/4$ such that for large enough $n$ it holds with probability at least $1- \delta^2$ that
		\[
	(c_{\mathrm{fpp}}^\cT - \epsilon/2) \lfloor \eta n \rfloor \le d_{\mathrm{fpp}}^{\cT_\infty^{(0)}}(v, \partial^* B_{n - \lfloor \eta n \rfloor}^\bullet(\cT_\infty^{(0)})) \le (c_{\mathrm{fpp}}^\cT + \epsilon/2) \lfloor \eta n \rfloor  
	\]
	for all $v \in \partial^* B_n^\bullet(\cT_\infty^{(0)}) $. We let $\cG_n$ denote this event.  We set $n_0 = n$, $n_1 = n - \lfloor \eta n \rfloor$, and inductively $n_i = n_{i-1} - \lfloor \eta n_{i-1} \rfloor$ for all $i \ge 1$. Set
	\[
		q = \left\lfloor \frac{\log(\epsilon/ 16)}{\log(1 - \eta)} \right\rfloor.
	\]
	Hence for large enough $n$ we have $n_q \le \epsilon n / 4$. Furthermore, for $n$ large enough,
	\[
		\Exb{\sum_{j=0}^{q-1} \one_{\cG_{n_j}^c} } \le \delta^2 q.
	\]
	By Markov's inequality,
	\[
	\Prb{\sum_{j=0}^{q-1} \one_{\cG_{n_j}^c} > \delta q} \le \delta.
	\]
	We let $\cH_n$ denote the event that $\sum_{j=0}^{q-1} \one_{\cG_{n_j}^c} \le \delta q$.
	
	Using Lemma~\ref{le:analem} and Inequality~\eqref{eq:iotabound} to bound the sums of weights along the left-most geodesics from points on $\partial^* B_{n_q}^\bullet(\cT_\infty^{(0)})$ to $\rho$, it follows that for sufficiently large $n$ with probability at least $1 - \delta$ all of these sums are less than $n_q$. We let $\tilde{\cH}_n$ denote the intersection of $\cH_n$ with this event and note that
	\[
		\Pr{\tilde{\cH}_n} \ge 1 - 2 \delta.
	\]

	 Suppose that $\tilde{\cH}_n$ holds. Let $v \in \partial^* B_n^\bullet(\cT_\infty^{(0)})$. We are going to inductively construct vertices $v_{(j)} \in \partial^* B_{n_j}^\bullet(\cT_\infty^{(0)})$ for $0 \le j \le q$, starting with $v_{(0)} = v$. Having construct $v_{(0)}, \ldots, v_{(j)}$ for some $j < q$, we define $v_{(j+1)}$ as follows. If $\cG_{n_j}$ holds, we set $v_{(j+1)}$ to some point in $\partial^* B_{n_{j+1}}^\bullet(\cT_\infty^{(0)})$ satisfying \[
	d_{\mathrm{fpp}}^{\cT_\infty^{(0)}}(v_{(j)}, v_{(j+1)}) = d_{\mathrm{fpp}}^{\cT_\infty^{(0)}}(v_{(j)}, \partial^* B_{n_{j+1}}^\bullet(\cT_\infty^{(0)})).
	\]
	Otherwise, we we set $v_{(j+1)}$ to some point in $\partial^* B_{n_{j+1}}^\bullet(\cT_\infty^{(0)})$ satisfying \[
	d_{\cT_\infty^{(0)}}(v_{(j)}, v_{(j+1)}) = n_j - n_{j+1}.
	\]
Since $v_{(q)} \in \partial^* B_{n_q}^\bullet(\cT_\infty^{(0)}) $ it follows by definition of $\tilde{\cH}_n$ that
	\begin{align}
		\label{eq:boundevent}
		d_{\mathrm{fpp}}^{\cT_\infty^{(0)}}(\rho, v_{(q)}) \le n_q \le  \epsilon n / 4.
	\end{align}
	Hence, if the event $\cH_n$ and the event~\eqref{eq:boundevent} simultaneously hold, we have
	\begin{align*}
		d_{\mathrm{fpp}}^{\cT_\infty^{(0)}}(\rho, v) &\le \frac{\epsilon n}{4} + \sum_{j=0}^{q-1} d_{\mathrm{fpp}}^{\cT_\infty^{(0)}}(v_{(j)}, v_{(j+1)}) \\
		&\le \frac{\epsilon n}{4} + \left(c_{\mathrm{fpp}}^\cT + \frac{\epsilon}{2}\right) \sum_{j=0}^{q-1}(n_j - n_{j+1}) + \delta q \max_{0 \le i < q}(n_i - n_{i+1}) \\
		&\le  \frac{\epsilon n}{4} + \left(c_{\mathrm{fpp}}^\cT + \frac{\epsilon}{2}\right)n + \delta q \eta n \\
		& \le \left(c_{\mathrm{fpp}}^\cT + {\epsilon}\right)n.
	\end{align*}
Here we used $\delta q \eta \le \epsilon/4$  in the last line since $\eta \in ]0,1[$ entails  $\eta \le | \log(1-\eta)|$. 

Furthermore, if $\omega$ is a path from $v$ to the root $\rho$, then for each integer $j \in \{0, \ldots, q\}$ we let $w_{(j)}$ denote the last point of $\omega$ that belongs to $\partial^*B_{n_j}^\bullet(\cT_\infty^{(0)})$. Using $c_{\mathrm{fpp}}^\cT \le 1$, $n_q \le \epsilon n /4$, and $\delta q \eta \le \epsilon / 4$, it follows that on the event $\cH_n$ for large enough $n$ the sum of weights of the edges of the path is bounded from below by
\begin{align*}
	\sum_{j=0}^{q-1} d_{\mathrm{fpp}}( w_{(j)}, \partial^*B_{n_j}^\bullet(\cT_\infty^{(0)})) &\ge (c_{\mathrm{fpp}}^\cT - \epsilon/ 2)(n_0 - n_q) - \delta q c_{\mathrm{fpp}}^\cT \max_{0 \le i < q}(n_i - n_{i+1}) \\
	&\ge n(c_{\mathrm{fpp}}^\cT - \epsilon/ 2)(1 - \epsilon/4) - \delta q \eta n \\
	& \ge (c_{\mathrm{fpp}}^\cT - \epsilon) n.
\end{align*}
Hence
\[
d_{\mathrm{fpp}}^{\cT_\infty^{(0)}}(\rho, v) \ge \left(c_{\mathrm{fpp}}^\cT - {\epsilon}\right)n.
\]
Since the event $\tilde{\cH}_n$  holds with probability at least $1 - 2\delta$ this completes the proof. 
\end{proof}

\section{The global shape of random simple triangulations}

Let $\mu_n$ denote the uniform measure on the vertex set of $\cT_n$. It was shown in~\cite{zbMATH06812193} that
\begin{align}
	\label{eq:convsimpleghp}
	(\cT_n, (3/8)^{1/4} d_{\cT_n}, \mu_n) \convdis (\mathbf{M}, d_{\mathbf{M}}, \mu_{\mathbf{M}})
\end{align}
in the Gromov--Hausdorff--Prokhorov sense, with $(\mathbf{M}, d_{\mathbf{M}}, \mu_{\mathbf{M}})$ denoting the Brownian map. This entails that if $(o_n^i)_{i \ge 1}$ denote independent uniform vertices of $\cT_n$, and $(u_i)_{i \ge 1}$ independent $\mu_{\mathbf{M}}$-samples of $\mathbf{M}$, then
\begin{align}
	\label{eq:bm1}
(3/8)^{1/4} n^{-1/4} d_{\cT_n}(o_n^1, o_n^2) \convdis d_{\mathbf{M}}(u_1, u_2).
\end{align}

Let $\tilde{\mu}_n$ denote the degree-biased measure on the vertex set of $\cT_n$. That is, $ \tilde{\mu}_n$ describes a vertex that this selected with probability proportional to its degree. The result~\cite[Cor. 3.2, Eq. (3.4)]{stufler2022scaling} used~\eqref{eq:convsimpleghp} (and employed similar arguments as for quadrangulations~\cite{zbMATH06847066}) to deduce that
\begin{align}
	d_{\mathrm{P}}(\mu_n, \tilde{\mu}_n) \convp 0,
\end{align}
with $d_{\mathrm{P}}$ denoting the Prokhorov-distance, and $\mu_n, \tilde{\mu}_n$ interpreted as measures on the rescaled triangulation $(\cT_n, (3/8)^{1/4} d_{\cT_n})$.  By~\cite[Cor. 11.6.4]{dudley_2002}, it follows that for each $n$ there exist a coupling between a uniformly selected vertex $o_n$ of $\cT_n$ and a $\tilde{\mu}_n$-distributed vertex $\tilde{o}_n$ such that
\begin{align}
	\label{eq:bm2}
	n^{-1/4} d_{\cT_n}(o_n, \tilde{o}_n) \convp 0.
\end{align}
Let $\rho_n$ denote the origin of the root-edge of $\cT_n$. The uniform simple triangulation $\cT_n$ is distributionally invariant under re-rooting at a uniformly selected corner, and the origin of the randomly selected and oriented new root-edge follows the degree-biased distribution $\tilde{\mu}_n$. Combining this with~\eqref{eq:bm1} and~\eqref{eq:bm2} it follows that 
\begin{align}
	\label{eq:yup1}
(3/8)^{1/4} n^{-1/4} d_{\cT_n}(\rho_n, o_n) \convdis d_{\mathbf{M}}(u_1, u_2).
\end{align}
It follows from~\cite[Prop. 10]{MR2571957} and~\eqref{eq:convsimpleghp} that for any fixed $k \ge 1$
\begin{align}
		\left(\cT_n, (3/8)^{1/4} d_{\cT_n}, (o_n^1,\ldots, o_n^k))\right) \convdis \left(\mathbf{M}, d_{\mathbf{M}}, (u_1, \ldots, u_k)\right)
\end{align}
with respect to the $k$-pointed Gromov--Hausdorff metric. Since $\mu_{\mathbf{M}}$ almost surely has full-support, it follows from~\cite[Lem. 4.1]{stufler2022mass} that
\begin{align}
	\label{eq:yup2}
	d_{\mathrm{H}}( \{u_1, \ldots, u_k\}, \mathbf{M}) \convas 0
\end{align}
as $k \to \infty$, with $d_{\mathrm{H}}$ denoting the Hausdorff metric. Combining~\eqref{eq:yup1} and~\eqref{eq:yup2}, it follows that for all $\epsilon, \delta>0$ there exists a number $K$ such that for all $n \ge 2$
\begin{align}
	\label{eq:approxmesh}
	\Prb{ \sup_{v \in \V(\cT_n)} \inf_{1 \le i \le K} d_{\cT_n}(v, o_n^i) > \epsilon n^{1/4}  } < \delta.
\end{align}
To be precise, it follows from~\eqref{eq:yup1} and~\eqref{eq:yup2} that~\eqref{eq:approxmesh} holds for sufficiently large $n$. However, increasing $K$ allows us to easily treat any finite number of indices $n$.

For any vertex $x \in \V(\cT_n)$ and any number $b>0$ we let $\mfB_{a}(\cT_n, x)$ denote the open ball with radius $b$ centred at $x$.
Using again that  $\mu_{\mathbf{M}}$ almost surely has full-support, it follows from the convergence~\eqref{eq:convsimpleghp} and~\cite[Cor. 3.4]{stufler2022mass} that for all $\epsilon, \delta>0$ there exists $b>0$ such that for all large enough $n$
\begin{align}
	\label{eq:ballmass}
	\Prb{\inf_{x \in \V(\cT_n)} \# \mfB_{\epsilon n^{1/4} }(\cT_n, x)  < b n   } < \delta.
\end{align}

\section{First-passage percolation on random finite simple triangulations}

Recall that $\cT_n$ denote the uniform simple triangulation with $n+1$ vertices. It will be notationally convenient to treat the root vertex $\rho_n$ of $\cT_n$ as a bottom-cycle of length $0$, making $\cT_n = \cT_n^{(0)}$ a triangulation of the ``$0$-gon''. We let $\overline{\cT}_n$ denote the result of marking a uniformly selected non-root vertex $o_n$ of $\cT_n$. For any integer $r \ge 1$, the hull $B_r^\bullet(\overline{T}_n)$ makes sense if $d_{\cT_n}(\rho_n, o_n) > r$. If $d_{\cT_n}(\rho_n, o_n) \le r$ we set $B_r^\bullet(\overline{T}_n)$ to some place-holder value. By abuse of  notation, we let $\ndC_{0,r}$ denote collection of triangulations of the cylinder with height $r$ and bottom-cycle length $0$. For any $t \in \ndC_{0,r}$ we let $N(t)$ denote the number of  non-root vertices.

Compare the following lemma with~\cite[Lem. 22]{zbMATH07144469}.

\begin{lemma}
	\label{le:finbod9}
	There is a constant $\bar{c}>0$ such that for all  $r \ge 1$, $t \in \ndC_{0,r}$, and $n > N(t)$ we have
	\begin{align}
		\Pr{ B_r^\bullet(\overline{\cT}_n) = t} \le \bar{c} \left( \frac{n}{n - N(t)} \right)^{3/2} \Pr{ B_r^\bullet({\cT}_\infty^{(0)} ) = t}.
	\end{align}
\end{lemma}
\begin{proof}
	We let $p \ge 3$ denote the length of the top cycle $\partial^* t$ of $t$.
	Analogously as for Equation~\eqref{eq:rapdasarmas} we have
	\begin{align}
	\label{eq:barab1}
		\Pr{ B_r^\bullet(\overline{\cT}_n) = t} = \frac{\#\ndT_{n - N(t),p}}{\#\ndT_{n-2}} \left( 1 - \frac{N(t)}{n} \right).
	\end{align}
	By Equation~\eqref{eq:tn0} and Equation~\eqref{eq:tnp} and the local convergence $\cT_n \convdis \cT_\infty$ from Proposition~\ref{pro:convvvv} it follows that
	\begin{align}
		\label{eq:bruhatord}
		\Pr{B_r^\bullet({\cT}_\infty^{(0)}) = t}  = \frac{4096}{27} \sqrt{\frac{2 \pi }{3}} C(p) \left( \frac{27}{256} \right)^{N(t)}.
	\end{align}
	Equation~\eqref{eq:barab1} and Lemma~\ref{le:eet} also entail that there exists a constant $c^*>0$ such that
	\begin{align}
		\Pr{ B_r^\bullet(\overline{\cT}_n) = t} \le c^* C(p) \left( \frac{n}{n - N(t)} \right)^{3/2} \left( \frac{27}{256} \right)^{N(t)}.
	\end{align}
	Hence
	\begin{align}
		\Pr{ B_r^\bullet(\overline{\cT}_n) = t} \le c^* \frac{4096}{27} \sqrt{\frac{2 \pi }{3}} \left( \frac{n}{n - N(t)} \right)^{3/2} \Pr{ B_r^\bullet({\cT}_\infty^{(0)} ) = t}.
	\end{align}
\end{proof}

We will need the next bound in order to deal with the class of link-weights under consideration.

\begin{proposition}
	\label{pro:roughtnbound}
	Let $0 < \epsilon  < 1/4$ be given.
	Then
	\begin{align}
		\label{eq:fffbound1}
		\lim_{n \to\infty} \Prb{ d_{\mathrm{fpp}}^{\cT_n} (x,y) \le d_{\cT_n}(x,y) \text{ for all $x,y \in \V(\cT_n)$ with $d_{\cT_n}(x,y) \ge  n^{\epsilon}$ }  } = 1
	\end{align}
	and
	\begin{align}
				\label{eq:fffbound2}
			\lim_{n \to\infty} \Prb{ d_{\mathrm{fpp}}^{\cT_n} (x,y) \le n^{\epsilon} \text{ for all $x,y \in \V(\cT_n)$ with $d_{\cT_n}(x,y) \le n^\epsilon$ }  } = 1.
	\end{align}
\end{proposition}
\begin{proof}
	Let $(\iota_i)_{i \ge 1}$ denote independent copies of $\iota$. Recall that $\Ex{\iota}=1/2$, and that $\iota$ has finite exponential moments. 
	
	There are at most $(n+1)^2$ pairs of vertices with graph distance at least $n^\epsilon$ from each other. Note that the maximal graph distance between any two points is at most $n$. Conditional on $\cT_n$, we may fix a geodesic between each such pair. The sum of weights along these marked geodesics is an upper bound for their first-passage percolation distance. Hence, conditional on $\cT_n$, the probability that there exists $x,y \in \cT_n$ with $d_{\cT_n}(x,y) \ge  n^{\epsilon}$ but $d_{\mathrm{fpp}}^{\cT_n} (x,y) > d_{\cT_n}(x,y)$ is bounded by
	\[
		n^2 \sup_{n^{\epsilon} \le k \le n} \Pr{\iota_1 + \ldots + \iota_k > k}.
	\]
	This bound does not depend on $\cT_n$, hence it also holds without conditioning on~$\cT_n$. Furthermore, it tends to zero by Inequality~\eqref{eq:iotabound}. This verifies Equation~\eqref{eq:fffbound1}.
	
	Similarly, the probability that there exist vertices $x,y \in \V(\cT_n)$ with $d_{\cT_n}(x,y) \le n^\epsilon$ but $d_{\mathrm{fpp}}^{\cT_n} (x,y) > n^{\epsilon}$ is bounded by
	\[
		n^2 \sup_{1 \le k \le n^{\epsilon}} \Pr{\iota_1 + \ldots + \iota_k > n^{\epsilon}}.
	\]
	This bound also tends to zero by  Inequality~\eqref{eq:iotabound}. This completes the proof.
\end{proof}

We prove the next result following  the arguments of the proof of~\cite[Lem. 22]{zbMATH07144469}.
\begin{proposition}
	\label{pro:protypical848}
	For each $\epsilon>0$,
	\[
		\Prb{ | d_{\mathrm{fpp}}^{\cT_n} (\rho_n, o_n) - c_{\mathrm{fpp}}^{\cT}  d_{\cT_n}(\rho_n, o_n) | > \epsilon n^{1/4} } \to 0
	\]
	as $n \to \infty$.
\end{proposition}
\begin{proof}
	By~\eqref{eq:yup1} we know that $n^{-1/4} d_{\cT_n}(\rho_n, o_n)$ is stochastically bounded. Hence it suffices to prove that for all $\epsilon>0$ and $0 < \nu < 1$ we have 
	\begin{align}
		\label{eq:yesyesyes}
		\Prb{ \Bigg \vert \frac{d_{\mathrm{fpp}}^{\cT_n} (\rho_n, o_n)}{d_{\cT_n}(\rho_n, o_n)} - c_{\mathrm{fpp}}^{\cT}  \Bigg \vert > 2 \epsilon } < \nu
	\end{align}
for large enough $n$. 

Let $r \ge 1$. Given $t \in \ndC_{0,r}$ with root vertex $\rho$, we let $a_\epsilon(t)$ denote the random variable with $a_\epsilon(t)=1$ if
\[
	\sup_{x \in \partial^* t} \left |  \frac{d_{\mathrm{fpp}}^t(\rho, x) }{d_t(\rho, x)} -  c_{\mathrm{fpp}}^{\cT}  \right | \ge \epsilon
\]
and $a_\epsilon(t)=0$ otherwise.

Let $0<b<1$. Using~Lemma~\ref{le:finbod9}, it follows that
\begin{align*}
	&\Prb{ a_\epsilon( B_r^\bullet(\overline{\cT}_n)) = 1,  \# B_r^\bullet(\overline{\cT}_n) \le (1-b)n} \\
	 &\qquad= \sum_{\substack{ t \in \ndC_{0,r} \\ N(t) + 1 \le (1-b)n  }} \Prb{  B_r^\bullet(\overline{\cT}_n) = t  }\Prb{ a_\epsilon(t) = 1} \\
	 &\qquad\le \bar{c} b^{-3/2} \sum_{\substack{ t \in \ndC_{0,r} }} \Prb{  B_r^\bullet({\cT}_\infty^{(0)}) = t  }\Prb{ a_\epsilon(t) = 1} \\
	 &\qquad= \bar{c} b^{-3/2} \Prb{ a_\epsilon( B_r^\bullet({\cT}_\infty^{(0)})) = 1}.
\end{align*}
By Proposition~\ref{pro:yoyo}, it follows that
\begin{align}
	\label{eq:saens}
	\lim_{r \to \infty} \sup_{n \ge 1} \Prb{ a_\epsilon( B_r^\bullet(\overline{\cT}_n)) = 1,  \# B_r^\bullet(\overline{\cT}_n) \le (1-b)n} = 0.
\end{align}

Let $0< \alpha < \beta < \gamma$. We let $\mfB_r(\cT_n, o_n)$ denote the ball of radius $r$ centred at $o_n$ in $\cT_n$. We also define the event
\[
	D_{\beta, \gamma, n} = \{\beta n^{1/4} < d_{\cT_n}(\rho_n, o_n) \le \gamma n^{1/4}\}.
\]
Note that if $D_{\beta, \gamma, n}$ holds, then the ball $\mfB_{(\beta - \alpha)n^{1/4}}(\cT_n, o_n)$ is contained in the complement of the hull $B^\bullet_{\lfloor \alpha n^{1/4} \rfloor}(\overline{\cT}_n)$. Hence
\[
	D_{\beta, \gamma, n}  \cap \{ \# \mfB_{(\beta - \alpha)n^{1/4}}(\cT_n, o_n) > bn\} \subset \{ B^\bullet_{\lfloor \alpha n^{1/4} \rfloor}(\overline{\cT}_n)  \le (1-b)n\}.
\]
Using Equation~\eqref{eq:saens}, it follows that 
\begin{align}
	\lim_{n \to \infty} \Prb{a_\epsilon( B_{\lfloor \alpha n^{1/4} \rfloor}^\bullet(\overline{\cT}_n)) = 1, D_{\beta, \gamma, n},  \# \mfB_{(\beta - \alpha)n^{1/4}}(\cT_n, o_n) > bn} = 0.
\end{align}

Given $0<y<1$, it follows from Inequality~\eqref{eq:ballmass} that there exists $0<b<1$ such that
\begin{align}
	\liminf_{n \to \infty} \Prb{ \# \mfB_{(\beta - \alpha)n^{1/4}}(\cT_n, o_n) > bn  } \ge y.
\end{align}
Since we may take $y$ arbitrarily close to $1$, it follows that
\begin{align}
	\label{eq:usem3}
		\lim_{n \to \infty} \Prb{a_\epsilon( B_{\lfloor \alpha n^{1/4} \rfloor}^\bullet(\overline{\cT}_n)) = 1, D_{\beta, \gamma, n}} = 0.
\end{align}

Now, suppose that $0< \epsilon < 1$ and $0 < \delta < \epsilon/2$. For all integer $j$ with \begin{align}
	\label{eq:deltarange}
\lfloor \delta^{-1} \rfloor < j \le \lfloor \delta^{-3} \rfloor
\end{align}
 we set
\begin{align*}
	\alpha_j &= j \delta^2, \\
	\beta_j &= (j+1)\delta^2, \\
	\gamma_j &= (j+2) \delta^2.
\end{align*}
Hence
\begin{multline*}
	\Prb{ \bigcup_{j= \lfloor \delta^{-1} \rfloor +1}^{\lfloor \delta^{-3} \rfloor} D_{\beta_j, \gamma_j, n}} \\ = \Prb{ (\lfloor \delta^{-1} \rfloor + 2)\delta^2 n^{1/4} < d_{\cT_n}(\rho_n, o_n) \le (\lfloor \delta^{-3} \rfloor + 2)\delta^2 n^{1/4} }.
\end{multline*}
Using~\eqref{eq:yup1}, it follows that we may take $\delta$ sufficiently small such that for all large enough $n$
\begin{align*}
	\Prb{ \bigcup_{j= \lfloor \delta^{-1} \rfloor +1}^{\lfloor \delta^{-3} \rfloor} D_{\beta_j, \gamma_j, n}} \ge 1 - \frac{\nu}{4}.
\end{align*}
Using~\eqref{eq:usem3} for the values $(\alpha_j, \beta_j, \gamma_j)_j$ under consideration, it follows that
\begin{align}
	 \Prb{a_\epsilon( B_{\lfloor \alpha n^{1/4} \rfloor}^\bullet(\overline{\cT}_n)) = 0, \bigcup_{j= \lfloor \delta^{-1} \rfloor +1}^{\lfloor \delta^{-3} \rfloor} D_{\beta_j, \gamma_j, n}} \ge 1 - \nu/2
\end{align}
for sufficiently large $n$. Let $F_n$, $F_n'$ denote the events considered in~\eqref{eq:fffbound1} and~\eqref{eq:fffbound2} for $\epsilon= 1/8$. By Proposition~\ref{pro:roughtnbound} it follows that 
\begin{align}
	\label{eq:yesthisevent}
	\Prb{a_\epsilon( B_{\lfloor \alpha n^{1/4} \rfloor}^\bullet(\overline{\cT}_n)) = 0, \bigcup_{j= \lfloor \delta^{-1} \rfloor +1}^{\lfloor \delta^{-3} \rfloor} D_{\beta_j, \gamma_j, n}, F_n, F_n'} \ge 1 - \nu.
\end{align}

In order to prove Inequality~\eqref{eq:yesyesyes} it hence suffices to verify that
\[
\Bigg \vert \frac{d_{\mathrm{fpp}}^{\cT_n} (\rho_n, o_n)}{d_{\cT_n}(\rho_n, o_n)} - c_{\mathrm{fpp}}^{\cT}  \Bigg \vert > 2 \epsilon
\]
holds on the event considered in~\eqref{eq:yesthisevent}.  To this end, let $j$ be an integer satisfying~\eqref{eq:deltarange}, and suppose that  $a_\epsilon( B_r^\bullet(\overline{\cT}_n)) = 0$ and $D_{\beta_j, \gamma_j, n}$ hold. Then
\begin{align*}
	d_{\mathrm{fpp}}^{\cT_n}(\rho_n, o_n) &\ge  \min\{ d_{\mathrm{fpp}}^{\cT_n}(\rho_n, x) \mid x \in \partial^* B_{\lfloor \alpha n^{1/4} \rfloor}^\bullet(\overline{\cT}_n)  \} \\
	&\ge (c_{\mathrm{fpp}}^{\cT}   - \epsilon  )\lfloor \alpha_j n^{1/4} \rfloor.
\end{align*}
Since
\[
	\frac{\alpha_j}{\gamma_j} = \frac{j}{j+2} \ge 1 - \frac{2}{j} > 1 - 2 \delta  > 1 - \epsilon, 
\]
it follows that
\begin{align*}
	\frac{d_{\mathrm{fpp}}^{\cT_n} (\rho_n, o_n)}{d_{\cT_n}(\rho_n, o_n)}   &\ge \frac{(c_{\mathrm{fpp}}^{\cT}   - \epsilon ) \lfloor \alpha_j n^{1/4} \rfloor  }{ \gamma_j n^{1/4} } \\
	&\ge c_{\mathrm{fpp}}^{\cT}   - 2\epsilon.  
\end{align*}
At the same time
\begin{align*}
	d_{\mathrm{fpp}}^{\cT_n} (\rho_n, o_n) &\le  \max\left\{ d_{\mathrm{fpp}}^{\cT_n}(\rho_n, x) \mid x \in \partial^* B_{\lfloor \alpha n^{1/4} \rfloor}^\bullet(\overline{\cT}_n) \right\} + \max( \lfloor \gamma_j n^{1/4} \rfloor - \lfloor \alpha_j n^{1/4} \rfloor, n^{1/8}) \\
	&\le (c_{\mathrm{fpp}}^{\cT}   + \epsilon ) \lfloor \alpha_j n^{1/4} \rfloor +  \max( \lfloor \gamma_j n^{1/4} \rfloor - \lfloor \alpha_j n^{1/4} \rfloor, n^{1/8}).
\end{align*}
Consequently, for sufficiently large $n$
\begin{align*}
	\frac{d_{\mathrm{fpp}}^{\cT_n}(\rho_n, o_n)}{d_{\cT_n}(\rho_n, o_n)} &\le \frac{(c_{\mathrm{fpp}}^{\cT}   + \epsilon ) \lfloor \alpha_j n^{1/4} \rfloor +  \max( \lfloor \gamma_j n^{1/4} \rfloor - \lfloor \alpha_j n^{1/4} \rfloor, n^{1/8}) }{\beta_j n^{1/4}} \\
	&\le c_\mathrm{fpp}^\cT + 2 \epsilon.
\end{align*}
This completes the proof.
\end{proof}

We are now ready to finalize the proof of our first main theorem. With all preparations done, the final steps are analogous to the proof of~\cite[Thm. 1]{zbMATH07144469}, with small adjustments to treat unbounded link-weights.

\begin{proof}[Proof of Theorem~\ref{te:main0}]
	The random triangulation $\cT_n$ is stochastically invariant under re-rooting at a uniformly selected oriented root-edge. Furthermore, $\cT_n$ has $n+1$ vertices and hence $3(n-1)$ edges. Hence it follows from Proposition~\ref{pro:protypical848} that
	\begin{align}
		\label{eq:expconv}
		\frac{1}{n+1} \frac{1}{6(n-1)} \Exb{ \sum_{v \in \V(\cT_n)} \sum_{e \in \overrightarrow{\E}(\cT_n)}  \one_{|d_{\mathrm{fpp}}^{\cT_n}(e_*, v) - c_{\mathrm{fpp}}^{\cT_n} d_{\cT_n}(e_*, v)| > \epsilon n^{1/4} }  } \to 0
	\end{align}
	as $n \to \infty$. Here $\overrightarrow{\E}(\cT_n)$ denotes the $6(n-1)$-element set of oriented edges of $\cT_n$, and for any $e \in \overrightarrow{\E}(\cT_n)$ we let $e^*$ denote its origin. It follows from~\eqref{eq:expconv} that the following lower bound
		\begin{align}
			\label{eq:rerootdfdf}
		 \frac{1}{6(n+1)^2} \Exb{ \sum_{v, \tilde{v} \in \V(\cT_n)} \one_{|d_{\mathrm{fpp}}^{\cT_n}(v, \tilde{v}) - c_{\mathrm{fpp}}^{\cT_n} d_{\cT_n}(v, \tilde{v})| > \epsilon n^{1/4} }  }
	\end{align}
 also tends to zero. In other words, if $(o_n^i)_{i \ge 1}$ denote independent uniform vertices of $\cT_n$, then
 \begin{align}
 	\label{eq:kan}
 			\Prb{ | d_{\mathrm{fpp}}^{\cT_n} (o_n^1, o_n^2) - c_{\mathrm{fpp}}^{\cT}  d_{\cT_n}(o_n^1, o_n^2) | > \epsilon n^{1/4} } \to 0
 \end{align}
as $n \to \infty$. 
By Inequality~\eqref{eq:approxmesh},  for all $\epsilon, \delta>0$ there exists a number $K$ such that for all $n \ge 2$
\begin{align}
	\label{eq:tik}
	\Prb{ \sup_{v \in \V(\cT_n)} \inf_{1 \le i \le K} d_{\cT_n}(v, o_n^i) < \epsilon n^{1/4}  } > 1- \delta.
\end{align}
Since $K$ is constant,~\eqref{eq:kan} entails that
\begin{align}
	\label{eq:trick}
	 \Prb{ | d_{\mathrm{fpp}}^{\cT_n} (o_n^i, o_n^j) - c_{\mathrm{fpp}}^{\cT}  d_{\cT_n}(o_n^i, o_n^j) | < \epsilon n^{1/4} \text{ for all $1 \le i,j \le K$} } \to 1.
\end{align}
By Proposition~\ref{pro:roughtnbound}, we additionally have
\begin{align}
	\label{eq:track}
\Prb{ d_{\mathrm{fpp}}^{\cT_n} (x,y) \le \max(d_{\cT_n}(x,y), n^{1/8}) \text{ for all $x,y \in \V(\cT_n)$  }  } \to 1.
\end{align}
Suppose that the events under consideration in~\eqref{eq:tik},~\eqref{eq:trick}, and~\eqref{eq:track} hold. Then for any two vertices $x,y \in \V(\cT_n)$ we can find indices $1 \le i,j \le K$ with 
\[
d_{\cT_n}(x,o_n^i) < \epsilon n^{1/4} \qquad \text{and} \qquad d_{\cT_n}(y,o_n^j) < \epsilon n^{1/4}.
\]
Using the triangle inequality twice, this entail
\[
	|d_{\cT_n}(x,y) - d_{\cT_n}(o_n^i, o_n^j)| \le 2 \epsilon n^{1/4}.
\]
Furthermore, using $n^{1/8} < \epsilon n^{1/4}$, we have
\[
d^{\cT_n}_{\mathrm{fpp}}(x,o_n^i) < \epsilon n^{1/4} \qquad \text{and} \qquad d^{\cT_n}_{\mathrm{fpp}}(y,o_n^j) < \epsilon n^{1/4}.
\]
Since $|d^{\cT_n}_{\mathrm{fpp}}(o_n^i,o_n^j) - c_{\mathrm{fpp}}^{\cT} d_{\cT_n}(o_n^i,o_n^j)| < \epsilon n^{1/4}$, applying the triangle inequality twice yields 
\[
	|d^{\cT_n}_{\mathrm{fpp}}(x,y) - c_{\mathrm{fpp}}^{\cT} d_{\cT_n}(o_n^i, o_n^j)| \le 3 \epsilon n^{1/4}.
\]
Hence
\[
	|d^{\cT_n}_{\mathrm{fpp}}(x,y) - c_{\mathrm{fpp}}^{\cT} d_{\cT_n}(x, y)| \le 5 \epsilon n^{1/4}.
\]
For large enough $n$, the  events under consideration in~\eqref{eq:tik},~\eqref{eq:trick}, and~\eqref{eq:track} hold jointly with probability at least $1 - 2 \delta$. Since $\delta$ was arbitrary, the proof is complete.
\end{proof}

We may also adapt~\cite[Thm. 2]{zbMATH07144469} to the type III case:

\begin{proposition}
	\label{pro:fppontinfinity}
	Let $\epsilon>0$ be given. Then
	\[
			\lim_{r \to \infty} \Prb{ \sup_{x,y \in\V(B_r(\cT_{\infty}))} \left| d_{\mathrm{fpp}}^{\cT_\infty}(x,y) - c_{\mathrm{fpp}}^{\cT} d_{\cT_\infty}(x,y)  \right| > \epsilon r} = 0.  
	\]
	Letting $B_r^{\mathrm{fpp}}$ denote the ball of radius $r$ with respect to $d_{\mathrm{fpp}}^{\cT_\infty}$, we have
	\[
	\lim_{r \to \infty}	\Prb{ B_{(1- \epsilon)r/  c_{\mathrm{fpp}}^{\cT}}(\cT_\infty) \subset B_r^{\mathrm{fpp}}(\cT_\infty)  \subset B_{(1- \epsilon)r/  c_{\mathrm{fpp}}^{\cT}}(\cT_\infty) } = 1.
	\]
\end{proposition}
\begin{proof}
	Clearly the first part of the statement implies the second. 
	
	Recall that in our notation $\cT_\infty= \cT_\infty^{(0)}$, so that we may use notation introduced for triangulations of polygons.
	 Given $0< \delta <1$ we let $A_r^{(\delta)}$ denote the collection of rooted planar maps $M$ satisfying the Inequality
	\begin{multline*}
		\Pr{ |d_{\mathrm{fpp}}^M(x,y) - c_{\mathrm{fpp}}^\cT d_M(x,y)| \le \epsilon r \text{ for all $x,y \in \V(M)$ with $d_M(\rho,x), d_M(\rho,y)\le r$}} \\ \ge 1- \delta.
	\end{multline*}
	Here $\rho$ denotes the root vertex of $M$.
	
	In order to prove the first part of the statement, it suffices to show that for each integer $K \ge 1$ we have
	\[
		\Prb{ B_{K r}^\bullet(\cT_\infty^{(0)}) \in A_r^{(\delta)} } \ge 1 - \delta
	\]
	for all sufficiently large integers $r$. Reason being that, since $\iota\ge \kappa>0$, if $K$ is sufficiently large, then first-passage percolation distance and graph distance between  any points $x$ and $y$ from $B_r(\cT_\infty^{(0)})$ is determined by the graph $B^\bullet_{Kr}(\cT_\infty^{(0)})$.
	
	Hence, let $K \ge 1$ be a given integer. Recall that $N(t) + 1$ denotes the number of vertices of  a triangulation $t \in \ndC_{0,r}$. It follows from~\cite[Thm. 2]{zbMATH06701703} that there exist integers $\alpha, \beta \ge 2$ such that
	\[
		D_r := \{\ndC_{0, Kr} \mid N(t) \ge \alpha r^4 \text{ or } N(t) \le \alpha^{-1} r^4 \text{ or } |\partial^* t| \ge \beta r^2 \}
	\]
	satisfies
	\[
		\Prb{B_{Kr}^\bullet(\cT_{\infty}^{(0)}) \in D_r} \le \delta / 2.
	\]
	To be precise,~\cite[Thm. 2]{zbMATH06701703} treats the type II UIPT. However,~\cite[Sec. 6]{zbMATH06701703} explains how the proof may easily be adapted to related models.
	
	Let $t \in \ndC_{0, Kr} \setminus D_r$. Using~\eqref{eq:barab1},~\eqref{eq:bruhatord} and Lemma~\ref{le:eet}, it follows that there are constants $c,c'>0$ that do not depend on $t$ such that
	\begin{align*}
		\Prb{ B_{Kr}^\bullet(\overline{\cT}_{2 \alpha r^4} ) = t} &\ge c C(|\partial^*t|) (256/27)^{-N(t)} \\
		&\ge c' \Prb{B_{Kr}^\bullet(\cT_{\infty}^{(0)}) = t}. 
	\end{align*}

	Hence, 
	\begin{align*}
		&\Prb{ B_{K r}^\bullet(\cT_\infty^{(0)}) \notin A_r^{(\delta)} } \\
		&\qquad\le \Prb{ B_{K r}^\bullet(\cT_\infty^{(0)}) \in D_r } + \sum_{t \in (A_r^{(\delta)})^c \cap (\ndC_{1, Kr} \setminus D_r)}  	\Prb{ B_{Kr}^\bullet( \cT_\infty^{(0)}) = t} \\
		&\qquad\le \delta/2 + (c')^{-1} \sum_{t \in (A_r^{(\delta)})^c \cap (\ndC_{1, Kr} \setminus D_r)}  	\Prb{ B_{Kr}^\bullet(\overline{\cT}_{2 \alpha r^4} ) = t} \\
		&\qquad\le \delta/2 + (c')^{-1}\Prb{ B_{Kr}^\bullet(\overline{\cT}_{2 \alpha r^4} ) \notin A_r^{(\delta)}}.
	\end{align*}
	By Theorem~\ref{te:main0} entails that $\Prb{ B_{Kr}^\bullet(\overline{\cT}_{2 \alpha r^4} ) \notin A_r^{(\delta)}} \to 0$ as $r \to \infty$. Hence
	\[
		\Prb{ B_{K r}^\bullet(\cT_\infty^{(0)}) \notin A_r^{(\delta)} } < \delta
	\]
	for all sufficiently large integers $r$.
\end{proof}

\section{First-passage percolation on  dual maps}

As before we let $\iota$ denote a random variable with  finite exponential moments such that there exists a constant $\kappa>0$ such that $\Pr{\iota \ge \kappa} = 1$. Without loss of generality we may assume that $\Ex{\kappa}  = 1/2$. This is done purely for notational convenience.

\subsection{Degree bounds}

For any triangulation $t$ we let $d(t)$ denote the degree of the root vertex of $t$, that is, of the origin of the root-edge.

\begin{proposition}[{\cite[Lem. 4.1]{MR2013797}}]
	\label{prop:rootdegsyay}
	For each $\epsilon>0$ there exists a constant $C$ such that
	\begin{align}
		\label{eq:boundtnkdeg}
		\Pr{d(\cT_n) = k} \le C (3/4 + \epsilon)^k
	\end{align} 
for all $n \ge 2$ and $k \ge 0$. Furthermore, for $k$ and $n$ large enough,
\begin{align}
	\label{eq:bound304593}
	\frac{\Pr{d(\cT_n) = k+1}}{\Pr{d(\cT_n) = k}} \le 3/4 + \epsilon.
\end{align}
\end{proposition}
To be precise, the first bound is in the statement of \cite[Lem. 4.1]{MR2013797}, the second is verified in its proof. From this we may deduce a similar bound for triangulations of the $p$-gon:

\begin{corollary}
	\label{co:degbound345}
	For each $0<\epsilon<1/4$ there exists a constant $C>0$ such that
	\begin{align}
		\Pr{d(\cT_n^{(p)}) = k} \le C (3/4 + \epsilon)^k
	\end{align} 
	for all integers $p \ge 3$,  $n \ge 0$, $k \ge 2$.
\end{corollary}
\begin{proof}	
We may order the edges incident to the root edge of $\cT_n$ in clock-wise order, starting from the root-edge. The result of deleting the first $p-3$ edges after the root-edge (but not the root-edge itself)  in the conditioned map $(\cT_n \mid d(\cT_n) \ge p-1)$ is distributed like a uniform simple triangulation of the $p$-gon with $n+1$ vertices. Thus, the probability for the root vertex of the $p$-gon to have degree $k$ is equal to
\[
q_{k,p,n} := \frac{\Pr{d(\cT_n) = k+p-3}}{\Pr{d(\cT_n) \ge p-1}}.
\]
Using~\eqref{eq:bound304593}, it follows that there exist $p_0$ and $k_0$ such that for all $p \ge p_0$ and $k \ge k_0$ we have
\begin{align*}
	q_{k,p,n} &\le \frac{\Pr{d(\cT_n) = k+p-3}}{\Pr{d(\cT_n) = p-1}} \\
	&\le (3/4 + \epsilon)^{k+2}.
\end{align*}
Furthermore, for all $3 \le p \le p_0$ it follows by Inequality~\eqref{eq:boundtnkdeg} and the local convergence $\cT_n \convdis \cT_\infty$ from Proposition~\ref{pro:convvvv} that
\begin{align*}
	q_{k,p,n} &\le C \frac{(3/4 + \epsilon)^{k+p-3}}{\Pr{d(\cT_n) \ge p_0}} \\
	&\le C(1 + o(1)) \frac{(3/4 + \epsilon)^{k-3}}{\Pr{d(\cT_\infty) \ge p_0}}.
\end{align*}
This completes the proof.
\end{proof}

\begin{corollary}
	Let $\cT^{(p)}$ denote the Boltzmann triangulation of the $p$-gon for $p \ge 3$. For each $0<\epsilon<1/4$ there exists a constant $C>0$ that does not depend on~$p$ such that
	\[
		\Pr{d(\cT^{(p)}) \ge k} \le C (\epsilon + 3/4)^k
	\]
	holds for all $k \ge 0$.
\end{corollary}
\begin{proof}
	The Boltzmann triangulation $\cT^{(p)}$ is a mixture of $(\cT_n^{(p)})_{n \ge 0}$. Hence the statement follows immediately from Corollary~\ref{co:degbound345}.
\end{proof}

We may also bound the maximum degree:
\begin{proposition}
	\label{pro:maxdeg}
	For each $0< \epsilon<1/4$ there exists a constant $C>0$ such that the maximum degree $\mathrm{MD}(\cT_n)$ of $\cT_n$ satisfies
	\[
		\Pr{ \mathrm{MD}(\cT_n) \ge k} \le C n (3/4 + \epsilon)^k.
	\]
\end{proposition}
\begin{proof}
	The triangulation $\cT_n$ is invariant under re-rooting at a uniformly selected and oriented root-edge. Hence the degree of the root of $\cT_n$ is distributed like the degree of a uniformly selected end of a uniformly selected edge. Since $\cT_n$ has $n+1$ vertices it has $3(n-1)$ edges. Using Inequality~\eqref{eq:boundtnkdeg}, it follows that
	\[
		\Pr{\mathrm{MD}(\cT_n) \ge k} \le 6(n-1) \Pr{d(\cT_n) \ge k}
	\]
	The statement now follows immediately from Inequality~\eqref{eq:boundtnkdeg}. 
\end{proof}

\subsection{Vertex degrees along paths}

Recall that $\F(\cT_{\infty}^{(0)})$ denotes the collection of faces of $\cT_{\infty}^{(0)}$.  For each integer $r \ge 1$ we let $\F_r(\cT_{\infty}^{(0)})$ denote the collection of downward triangles at height $r$ in $\cT_{\infty}^{(0)}$. For any face $f_0 \in \F_r(\cT_\infty^{(0)})$ we construct a \emph{downward path}, which is a specific dual path $\omega$ starting at $f_0$ and ending at the face to the right of the root-edge. Its construction is as follows. Let $v_0$ be the unique boundary vertex of $f_0$ that lies on $\F_{r-1} (\cT_\infty^{(0)})$. Since $\cT_{\infty}^{(0)}$ is a simple triangulation, we may order the faces incident to $v_0$ in a counter-clockwise way. If $r=1$ then $v_0$ is the root vertex of $\cT_\infty^{(0)}$ and the  downward path proceeds from $f_0$ in counter-clockwise order along these faces until it reaches the face that lies to the right of the oriented rooted. If $r \ge 2$, then the downward path proceeds from $f_0$ in counter-clockwise order along these faces until it reaches a downward triangle $f' \in \F_{r-1}(\cT_\infty^{(0)})$. From there the construction continues inductively. 

Note that the downward path has an initial segment $f_0, f_1, \ldots, f_N$ of downward triangles. The corresponding edges $e_0, \ldots, e_N$   on $\partial^*(B_r^\bullet(\cT_\infty^{(0)}))$ form a counter-clockwise path on the top cycle of $B_r^\bullet(\cT_\infty^{(0)})$, with $e_N$ being the first one with offspring in the skeleton of $\cT_\infty^{(0)}$. Then the downward path  ``crosses'' the slot of boundary size $c_{e_N}+2$, and arrives at $f'$. The time needed for this crossing is exactly equal to the degree of the root of the triangulation filling the hole. The degree bound for simple Boltzmann triangulations from Corollary~\ref{co:degbound345} hence provides an exponential bound for this time that does not depend on the perimeter of the slot.

We define downward paths in the same way for the type III LHPT $\cL$. For all integers $i \in \ndZ$ and $r \ge 0$ we let $f_{(i,r)}$ denote the unique downward triangle of $\cL$ that is incident to the edge between $(i-1,-r)$ and $(i, -r)$.
 The next lemma is analogous to~\cite[Lem. 24]{zbMATH07144469}.

\begin{lemma}
	\label{le:omegar}
	Let $\omega_r$ denote the downward path in $\cL$ connecting $f_{(0,0)}$ to a downward triangle incident to $\cL_r$. Let $|\omega_r|$ denote the length of $\omega_r$. There exist two constants $\mu>0$ and $K >1$ such that for all integers $r \ge 1$\[
		\Exb{\exp(\mu |\omega_r|)} \le K^r.
	\]
\end{lemma}
\begin{proof}
	Since the layers of $\cL$ are i.i.d., it suffices to show that there exists $\mu>0$ with
	\begin{align}
		\label{eq:gloCK}
		\Ex{\exp(\mu |\omega_1|)} < \infty.
	\end{align}
	By construction, the downward path starting at $f_{(0,0)}$ first visits  downward  triangles $f_{(0,0)}, \ldots,  f_{(-N,0)}$, with $N+1$ the first integer $i$ such that the tree $\cJ_{-i}$ has more than one vertex. Thus
	\begin{align}
		\label{eq:gloA}
		\Pr{N \ge k} = \theta(0)^k
	\end{align}
	for all integers $k \ge 0$. The slot associated to $f_{(-N,0)}$ is filled by an independent simple Boltzmann triangulation of the $(d+2)$-gon, with $d$ the number of children of the root of the tree $\cJ_{-N-1}$. Letting $\cD_{d+2}$ denote the degree of the root of this Boltzmann triangulation, it follows that
	\begin{align}
		\label{eq:gloB}
		|\omega_1| = N + \cD_{d+2}.
	\end{align}
	Corollary~\ref{co:degbound345} entails that there exists a constant $C>0$ that does not depend on $d$ such that
	\begin{align}
		\label{eq:C}
		\Pr{\cD_{d+2} \ge k} \le C (4/5)^k
	\end{align}
	for all integers $k \ge 0$. Combining Equation~\eqref{eq:gloA}, Equation~\eqref{eq:gloB}, and Inequality~\eqref{eq:C}, it follows that there exists $\mu>0$ small enough such that~\eqref{eq:gloCK} holds.
\end{proof}

We let $d_{\mathrm{fpp}}^\dagger$ denote the $\iota$-first-passage percolation distance on the dual. We let $\cL_r^\dagger$ denotes the collection of downward triangles incident to edges of $\cL_r$.

Since $\iota$ as finite exponential moments, we may take $\mu>0$ small enough so that the same statement of Lemma~\ref{le:omegar} holds for the sum of weights along the downward path.  Using  $\Pr{\iota>\kappa} = 1$, we obtain the next observation by fully analogous arguments as for Proposition~\ref{pro:kingman}.

\begin{proposition}
	\label{pro:king345}
	There exists a constant $c_{\mathrm{fpp}}^\dagger \ge \kappa$ such that
	\[
	r^{-1} d_{\mathrm{fpp}}^\dagger(f_{(0,0)}, \cL_r^\dagger) \convas c_{\mathrm{fpp}}^\cT
	\]
	as $r \to \infty$.
\end{proposition}

\subsection{Bounds on distances in the dual}

\label{sec:bodidu}

In the following, for each integer $r \ge 1$ we let $f_r$ denote a uniformly selected downward triangle at height $r$ in $\F_r(\cT_\infty^{(0)})$.

We prove the next lemma following closely the arguments for~\cite[Lem. 26]{zbMATH07144469}.

\begin{lemma}
	\label{le:preclema}
	There exist constants $K, \alpha, \beta >0$ such that for all integers $0 \le r < s$ we have
	\[
		\Prb{ d_{\mathrm{fpp}}^\dagger (f_s, B_r^\bullet(\cT_\infty^{(0)})) >  \alpha(s-r)  } \le K \exp(- \beta(s-r)).
	\]
	Here $ B_0^\bullet(\cT_\infty^{(0)})$ only contains the root vertex.
\end{lemma}
\begin{proof}
	It suffices to prove this for the dual distance $d^\dagger_{\cT_\infty^{(0)}}$, since we may always bound the sum of weights along a $d^\dagger_{\cT_\infty^{(0)}}$-geodesic to get an upper bound for the $d_{\mathrm{fpp}}^\dagger$-distance.  Furthermore, we may assume that $r \ge 1$, since any downward triangle of height $1$ is incident to the root-vertex (and the total number of triangles incident to the root-vertex has an exponential tail by~\eqref{eq:infp}).
	
	Recall that $\widetilde{\cF}_{r,s}^{(0)}$ denotes the skeleton  of $B_s^\bullet(\cT_\infty^{(0)}) \setminus B_r(\cT_\infty^{(0)})$ after forgetting the marked vertex of $\cF_{r,s}^{(0)}$    and applying a uniform random cyclic permutation to the $L_s$ trees of~$\cF^{(0)}_{r,s}$. This way, conditional on $L_r=p$ and $L_s=q$, $\widetilde{\cF}_{r,s}^{(0)}$ is uniformly distributed over $\ndF_{p,q,s-r}''$. We may assume that $f_s$ is the downward triangle that corresponds to the root of the last tree in the forest $\widetilde{\cF}_{r,s}^{(0)}$.
	
		By Equations~\eqref{eq:slim} and~\eqref{eq:shady} it holds for all $p,q \ge 3$ and  $\cF \in \ndF_{p,q,s-r}''$	that
	\[
	\Prb{\widetilde{\cF}_{r,s}^{(0)} = \cF \mid L_r = p } = O(\sqrt{p/q})\prod_{v \in \cF^*} \theta(c_v). 
	\]
	Hence
	\[
	\Prb{\widetilde{\cF}_{r,s}^{(0)} = \cF  } = O(\sqrt{p/q})\Pr{ L_r = p} \prod_{v \in \cF^*} \theta(c_v). 
	\]
	Using Inequality~\eqref{eq:fromthis453}, it follows that
	\begin{align}
		\Prb{\widetilde{\cF}_{r,s}^{(0)} = \cF  } &= O\left(\frac{1}{\sqrt{q}}\right) \frac{p}{r^3}  \exp\left(-\frac{p}{4r^2} \right)\prod_{v \in \cF^*} \theta(c_v) \\
		&= O\left(\frac{1}{r \sqrt{q}}\right)\prod_{v \in \cF^*} \theta(c_v). \nonumber
	\end{align}
	Hence the law of $\widetilde{\cF}_{r,s}^{(0)}$ under $\Pr{ \cdot \cap \{L_s = q\}}$ is dominated by $O\left(\frac{1}{r \sqrt{q}}\right)$ times the law of a forest of $q$ independent $\theta$--Bienaym\'e--Galton--Watson trees, truncated at height $[s-r]$, and we may restrict the latter law to the event that the truncated forest reaches that height. The length of a downward path from $f_s$ to $B_r^\bullet(\cT_\infty^{(0)})$ is fully determined by the forest $\widetilde{\cF}_{r,s}^{(0)}$ and the triangulations filling the slots associated to vertices of that forest with height strictly less than $s-r$. Thus, the law of this length under $\Pr{ \cdot \cap \{L_s = q\}}$  is dominated by $O\left(\frac{1}{r \sqrt{q}}\right)$ times the law of the length of the downward path for triangulation of the cylinder of height $s-r$ whose cyclically permuted skeleton is a forest of $q$ independent $\theta$--Bienaym\'e--Galton--Watson trees truncated at height $s-r$ (restricted to the event that the truncated forest reaches that height), and whose slots are filled with independent type III Boltzmann triangulations of the corresponding perimeters. Thus, the law of the length of the downward path from $f_s$ to $B_r^\bullet(\cT_\infty^{(0)})$ under $\Pr{ \cdot \cap \{L_s = q\}}$  is dominated by $C/(\sqrt{q}(r+1))$ times the law of the downward path from $f_{(0,0)}$ to $\cL_{s-r}$ in the LHPT model. Using Lemma~\ref{le:omegar}, it follows that
	\begin{align*}
		\Prb{L_s = q, d_{\cT_\infty^{(0)}}^\dagger (f_s, B_r^\bullet(\cT_\infty^{(0)}))> \alpha(s-r)  }  &=O\left( \frac{1}{r\sqrt{q}}\right) \Pr{|\omega_{s-r}|>\alpha(s-r)} \\
		&= O\left( \frac{1}{r\sqrt{q}}\right) \exp(- \mu \alpha(s-r)) K^{s-r}.
	\end{align*}
Let us fix $\alpha>0$ and $\beta'>0$ with $\exp(- \mu \alpha) K \le \exp(- \beta')$. Then
\begin{align}
	\label{eq:11pac}
	\Prb{L_s = q, d_{\cT_\infty^{(0)}}^\dagger (f_s, B_r^\bullet(\cT_\infty^{(0)}))> \alpha(s-r)  } &=  O\left( \frac{1}{r\sqrt{q}}\right) \exp(- \beta'(s-r)).
\end{align}
Summing over $q$ and using~\eqref{eq:11pac} for $q \le (s-r)s^2$ and~\eqref{eq:2pac} for $q > (s-r)s^2$ we get
\begin{align*}
	&\Prb{d_{\cT_\infty^{(0)}}^\dagger (f_s, B_r^\bullet(\cT_\infty^{(0)}))> \alpha(s-r)  } \\
	&\qquad \le \Pr{L_s > (s-r)s^2} + O(1)\sum_{q=1}^{(s-r)s^2} \frac{1}{r \sqrt{q}} \exp(-\beta'(s-r)) \\
		&\qquad \le C_0 \exp(-(s-r)/5) + O(1) \frac{\sqrt{(s-r)s^2}}{r} \exp(-\beta'(s-r)).
\end{align*}
Choosing any $\beta$ with $0 < \beta  < \min(\beta', 1/5)$ it follows that the last display is bounded from above by $C \exp(- \beta(s-r))$ for some constant $C>0$.
\end{proof}

We deduce the following corollary analogously to~\cite[Cor. 27]{zbMATH07144469}.

\begin{corollary}
	\label{co:cobetatilde}
	Let $\alpha$ be as in Lemma~\ref{le:preclema}. Let $\delta>0$. For each integer $R \ge 1$, let $A_R(\delta)$ denote the event that
	\[
		d^\dagger_{\mathrm{fpp}}(f, B_r^\bullet(\cT_\infty^{(0)})) \le \alpha(s-r)
	\]
	holds for all $0 \le r < s \le R$ with $s-r \ge \delta R$, and each downward triangle $f$ at height~$s$. Then there exists a constant $\tilde{\beta}>0$ such that for large enough $R$
	\[
		\Prb{A_R(\delta)} \ge 1 - \exp(- \tilde{\beta}R).
	\]
\end{corollary}
\begin{proof}
	Let $r$ and $s$  satisfy $0 \le r < s \le R$ and $s-r \ge \delta R$. Let $f_{(1)}$ denote uniformly selected element of $\F_s(\cT_\infty^{(0)})$. Let $f_{(1)}, f_{(2)}, \ldots $ denote the successive downward triangles at height $s$ when walking around $\partial B_s^\bullet(\cT_\infty^{(0)})$ in clockwise order, starting from $f_{(1)}$. By Lemma~\ref{le:preclema}, we have
	\[
		\Prb{ d_{\mathrm{fpp}}^\dagger (f_{(j)}, B_r^\bullet(\cT_\infty^{(0)})) >  \alpha(s-r)  } \le K \exp(- \beta(s-r))
\]
for each $j$. Using Inequality~\eqref{eq:2pac}, it follows that
\begin{align*}
	&\Prb{d_{\mathrm{fpp}}^\dagger (f, B_r^\bullet(\cT_\infty^{(0)})) >  \alpha(s-r) \text{ for some $f \in F_s(\cT_\infty^{(0)})$ }} \\
	&\qquad \le \sum_{j=1}^{R s^2} \Prb{d_{\mathrm{fpp}}^\dagger (f_{(j)}, B_r^\bullet(\cT_\infty^{(0)})) >  \alpha(s-r)} + \Pr{L_s > Rs^2} \\
	&\qquad \le K R^3 \exp(- \beta \delta R) + C_0 \exp(-R/5).
\end{align*}
Summing over all possible values of $r$ and $s$ yields the statement.
\end{proof}

The next observation corresponds to the result~\cite[Lem. 28]{zbMATH07144469} for type I triangulations, and the proof is analogous.

\begin{lemma}
	\label{le:stateofthemind}
By abuse of notation, we also let $d_{\mathrm{fpp}}^\dagger$ denote the $\iota$-first-passage percolation distance on  the dual of $\cT_n$. 
Let $\alpha>0$ be as in Lemma~\ref{le:preclema}.	Let $0 < \epsilon < 1/4$. For each $n \ge 1$, let $E_n$ denote the event that
\[
	d_{\mathrm{fpp}}^\dagger(f,g) \le \alpha d_{\cT_n}(x,y)+ n^{\epsilon}
\]
holds for all $x,y \in \V(\cT_n)$ and $f,g \in \F(\cT_n)$ satisfying $x\triangleleft f$ and $y \triangleleft g$. Then $E_n$ holds with probability tending to $1$ as $n \to \infty$. 
\end{lemma}
\begin{proof}
Since (without loss of generality) we assumed that $\Ex{\iota}=1/2$, an analogous statement as in Proposition~\ref{pro:roughtnbound}	also holds for first-passage percolation on the dual of $\cT_n$. It follows that it suffices to prove the statement of Lemma~\ref{le:stateofthemind} for the dual graph distance $d_{\cT_n}^\dagger$.

As we before, we let $\rho_n$ denote the root vertex of $\cT_n$, and $\overline{\cT}_n$ the pointed version of $\cT_n$ obtained by marking a uniformly selected vertex $o_n$.  Recall also that we view $\cT_n = \cT_n^{(0)}$ as a triangulation of the  ``$0$-gon''. The hull $B_r^\bullet(\overline{\cT}_n)$ is well-defined provided that $0<r<  d_{\cT_n}(\rho_n, o_n)$. We let $\cE$ denote the collection of all triangulations $t$ with $t \in \ndC_{0,r}$ for some $r >1$, with the property that there is a face $f$ incident to $\partial^*t$ with dual graph distance from a face incident to $\rho_n$. 

By Lemma~\ref{le:finbod9} and Corollary~\ref{co:cobetatilde}, 
\begin{align*}
	\Prb{d_{\cT_n}(\rho_n, o_n) > r, B_r^\bullet(\overline{\cT}_n)\in \cE} &\le  \sum_{t \in \ndC_{0,r}} \one_{\cE}(t) 	\Prb{d_{\cT_n}(\rho_n, o_n) > r, B_r^\bullet(\overline{\cT}_n) = t} \\
	&\le \bar{c} n^{3/2} \sum_{t \in \ndC_{0,r}} \one_{\cE}(t) 	\Prb{ B_r^\bullet(\overline{\cT}_\infty^{(0)}) = t} \\
	&\le \bar{c} n^{3/2} \exp(- \tilde{b} r).
\end{align*}
Summing over all $r \ge \lfloor n^\epsilon \rfloor$, we get
\[
	\Exb{ \sum_{r=\lfloor n^\epsilon \rfloor}^\infty \one_{d_{\cT_n}(\rho_n, o_n) > r, B_r^\bullet(\overline{\cT}_n)\in \cE}  } \le \tilde{c} \exp(-a n^\epsilon)
\]
for some constants $\tilde{c},a > 0$. Hence
\begin{align}
	\Prb{d_{\cT_n}(\rho_n, o_n) > n^\epsilon, B_{d_{\cT_n}(\rho_n, o_n) -1 }^\bullet(\overline{\cT}_n)\in \cE} \le \tilde{c} \exp(-a n^\epsilon).
\end{align}

The vertex $o_n$ is adjacent to a vertex $v_0$ from $\partial^* B_{d_{\cT_n}(\rho_n, o_n)-1}^\bullet(\overline{\cT}_n)$. Any face $g$ incident to $o_n$ is at dual graph distance at most $\mathrm{MD}(\cT_n)$ from a face $g'$ incident to the edge between $o_n$ and $v_0$. Likewise, $g'$ is at dual graph distance at most $\mathrm{MD}(\cT_n)$ from a downward triangle at height $d_{\cT_n}(\rho_n, o_n)-1$. 
Hence, if $d_{\cT_n}(\rho_n, o_n) > n^\epsilon$ and  $B_{d_{\cT_n}(\rho_n, o_n) -1 }^\bullet(\overline{\cT}_n) \notin \cE$, then  we have
\[
	d_{\cT_n}^\dagger(f,g) \le \alpha d_{\cT_n}(\rho_n, o_n) + 2\mathrm{MD}(\cT_n)
\]
for all $f,g \in \F(\cT_n)$ with $\rho_n \triangleleft f$ and $o_n \triangleleft g$.

Using the stochastic re-rooting invariance of $\cT_n$ similarly as for~\eqref{eq:expconv} and~\eqref{eq:rerootdfdf}, it follows that
\begin{multline*}
	\Exb{ \sum_{v,v' \in \V(\cT_n)} \one_{d_{\cT_n}(v,v') > n^\epsilon} \one_{ d_{\cT_n}^\dagger(f,g) > \alpha d_{\cT_n}(v,v') + 2 \mathrm{MD}(\cT_n) \text{ whenever $v \triangleleft f$ and $v' \triangleleft g$} } } \\
	\le 6 (n+1)^2 \tilde{c} \exp(-a n^{\epsilon}).
\end{multline*}
Using Proposition~\ref{pro:maxdeg}  follows that with probability tending to $1$ as $n \to \infty$ we have
\[
	d_{\cT_n}^\dagger(f,g) \le \alpha d_{\cT_n}(v,v') + (\log n)^2
\]
whenever $v,v' \in \V(\cT_n)$, $d_{\cT_n}(v,v') > n^\epsilon$, and $v \triangleleft f$, $v' \triangleleft g$. The easy bound
\[
	d_{\cT_n}^\dagger(f,g) \le \mathrm{MD}(\cT_n)(d_{\cT_n}(v,v') +1)
\]
additionally implies that with probability tending to $1$ as $n \to \infty$ we have
\[
d_{\cT_n}^\dagger(f,g) \le (\log n)^2 (n^\epsilon+1)
\]
whenever $d_{\cT_n}(v,v') \le  n^\epsilon$, and $v \triangleleft f$, $v' \triangleleft g$. Since $\epsilon>0$ was arbitrary this completes the proof.
\end{proof}

We prove the next observation following closely the arguments of~\cite[Prop. 29]{zbMATH07144469}.

\begin{proposition}
	\label{pro:primal}
	Let $\epsilon, \delta >0$. There exists $0 < \eta < 1/2$ such that for sufficiently large $n$ it holds with probability at least $1 - \delta$ that for all $f \in \F_n(\cT_{\infty}^{(0)})$
	\[
		(1 - \epsilon) c_{\mathrm{fpp}}^\cT \lfloor \eta n \rfloor \le d_{\mathrm{fpp}}^\dagger(f, B_{n - \lfloor \eta n \rfloor}^\bullet (\cT_\infty^{(0)})) \le (1 + \epsilon) c_{\mathrm{fpp}}^\cT \lfloor \eta n \rfloor.
	\]
	Moreover, letting $f_*$ denote the face to the right of the root-edge of $\cT_{\infty}^{(0)}$,
	\[
	 \lim_{n \to \infty} \Prb{ (c_{\mathrm{fpp}}^\cT  - \epsilon) n \le d_{\mathrm{fpp}}^\cT(f_*, f) \le (c_{\mathrm{fpp}}^\cT  + \epsilon) n) \text{ for all $f \in \F_n(\cT_\infty^{(0)})$}  }.
	\]
\end{proposition}
\begin{proof}
		Let $f$ be uniformly selected from $\F_n(\cT_\infty^{(0)})$. The proof of the first assertion is analogous to the proof of Proposition~\ref{pro:ring2ring}, using Proposition~\ref{pro:king345} instead of Proposition~\ref{pro:kingman}, and noting that Corollary~\ref{co:cobetatilde} already gives the bound $d_{\cT_{\infty}^{(0)}}(f, B_{n - \lfloor \eta n \rfloor}^\bullet(\cT_{\infty}^{(0)}) ) \le \alpha \lfloor \eta n \rfloor$ with sufficient probability.
		
		We describe the adaptions. Recall the notation from the proof of Proposition~\ref{pro:ring2ring}. For each $i \in \ndZ$ we write $f_i^{(n)} \in \F_n(\cT_\infty^{(0)})$ for the  downward triangle corresponding to the edge on $\partial B_n^\bullet(\cT_\infty^{(0)})$ from $u_i^{(n)}$ to $u_{i+1}^{(n)}$. Let $j$ be an integer with $f = f_j^{(n)}$. Analogous to Equation~\eqref{eq:likesmallyeah} we  have to bound the probability that there exists an integer $i$ with $j - a n^2/16 \le i \le j + an^2 / 16$ such that there is a dual path from $f_i^{(n)}$ to $\partial B_{n - \lfloor \eta n\rfloor }^\bullet(\cT_\infty^{(0)})$ with length smaller than $4 \alpha \eta n / \kappa$ that stays in $B_n^\bullet(\cT_\infty^{(0)})$ and leaves the region $\cG_j^{(n)}(\eta)$ before hitting $ B_{n - \lfloor \eta n\rfloor }^\bullet(\cT_\infty^{(0)})$. It is easy to see that if there is such a dual path then there is a path in $\cT_\infty^{(0)}$ from $u_i^{(n)}$ to $\partial_\ell \cG_j^{(n)}(\eta)$ in $B_n^\bullet(\cT_\infty^{(0)}) \setminus B_{n - \lfloor \eta n \rfloor}^\bullet(\cT_\infty^{(0)})$ with length at most $4\alpha \eta n / \kappa +1$. As argued in the proof of Proposition~\ref{pro:ring2ring}, the existence of such a path is unlikely.
		
		In order to adapt the final part of the proof of Proposition~\ref{pro:ring2ring}, we need to verify that for each $\beta>0$ we can find a sufficiently large integer $A$ such that with high probability any downward triangle at height $n$ is connected to one of the downward triangles $f_j^{(n)}$ for $0 \le j \le \lfloor a^{-1} A \rfloor$ by a dual path in $B_n^\bullet(\cT_\infty^{(0)})$ with sum of weights at most $\beta c_\mathrm{fpp}^\dagger \eta n$. In order to do so we may also use Proposition~\ref{pro:giveusana}. If $f = f_j^{(n)}$ and $f' = f_{j'}^{(n)}$ are two downward triangles at height $n$, then the left-most geodesics from $u_j^{(n)}$ and $u_{j'}^{(n)}$ coalesce above height $n' < n$ and hence the same property holds for the downward paths from $f$ and $f'$. Using the bounds on the lengths of downward paths in the proofs in Section~\ref{sec:bodidu}, we get the desired control on the sum of weights of the dual path from $f$ to $f'$ by concatenating the respective downward paths up to their coalescence time.
		
		We may prove the second assertion similarly to Proposition~\ref{pro:yoyo}. Using the notation from the proof of Proposition~\ref{pro:yoyo}, Corollary~\ref{co:cobetatilde} allows us to take care of the ``bad'' values of $i$ for which the bound
		\[
			(1 - \epsilon) c_\mathrm{fpp}^\dagger (n_i - n_{i+1}) \le d_{\mathrm{fpp}}^\dagger (f, B_{n_{i+1}}^\bullet(\cT_\infty^{(0)})) \le (1+ \epsilon) c_\mathrm{fpp}^\dagger (n_i - n_{i+1})
		\]
		fails for some $f \in \F_{n_i}(\cT_\infty^{(0)})$. The rest of the proof is the same.
\end{proof}

\subsection{Concentration of first-passage percolation distances in the dual}

We may carry out the proof of Theorem~\ref{te:main1} in the same way as the proof of~\cite[Thm. 3]{zbMATH07144469}.

\begin{proof}[Proof of Theorem~\ref{te:main1}]
	Our first step is to adapt Proposition~\ref{pro:protypical848}. Let $f_*$ denote the face to the right of the root-edge of $\cT_n$. Let $o_n$ denote a uniformly selected vertex of $\cT_n$ and let $f_n$ be a face incident to $o_n$. Then for each $\epsilon>0$
	\begin{align}
		\label{eq:aathand}
		\lim_{n \to\infty} \Prb{| d_{\mathrm{fpp}}^\dagger(f_*, f_n) - c_{\mathrm{fpp}}^\dagger d_{\cT_n}(\rho_n, o_n) | > \epsilon n^{1/4} } =0.
	\end{align}
The proof is analogous to the proof of Proposition~\ref{pro:protypical848}, using the same absolute continuity argument via Lemma~\ref{le:finbod9}, but using Prop.~\ref{pro:primal} instead of Proposition~\ref{pro:yoyo}. The only difference is that at the end of the proof where on the event $D_{\beta_j, \gamma_j, n}$  we use Lemma~\ref{le:stateofthemind} to bound $d_{\mathrm{fpp}}^\dagger(f_n, B^\bullet_{\lfloor \alpha_j n^{1/4}\rfloor }(\cT_n))$. 

With Equation~\eqref{eq:aathand} at hand, the proof of Theorem~\ref{te:main1} is analogous to the proof of Theorem~\ref{te:main0}, using Lemma~\ref{le:stateofthemind} to approximate the dual map by sufficiently large number of random faces with respect to $d_{\mathrm{fpp}}^\dagger$.
\end{proof}

A similar statement also holds for the first-passage percolation distance on $\cT_\infty$, compare with~\cite[Thm. 4]{zbMATH07144469}.

\begin{proposition}
	Let $0<\epsilon<1$. Then 
	\[
	\lim_{r \to \infty} \Prb{ \sup_{\substack{u,v \in \V(B_r(\cT_\infty))\\ u\triangleleft f, v\triangleleft g }} \left | c_{\mathrm{fpp}}^\dagger d_{\cT_\infty}(u,v) -  d_{\mathrm{fpp}}^\dagger(f,g)\right | > \epsilon r} = 0
	\]
	and
	\[
		\lim_{r \to \infty} \Prb{B_{(1-\epsilon)r/c_{\mathrm{fpp}}^\dagger}(\cT_\infty) \subset B_r^{\dagger, \mathrm{fpp}}(\cT_\infty) \subset B_{(1+\epsilon)r/c_{\mathrm{fpp}}^\dagger}(\cT_\infty) }  = 1.
	\]
\end{proposition}

The proof is by a straight-forward adaption of the proof of Proposition~\ref{pro:fppontinfinity}.

\bibliographystyle{abbrv}
\bibliography{bmcubic}

\begin{thebibliography}{10}

\bibitem{zbMATH06812193}
L.~{Addario-Berry} and M.~{Albenque}.
\newblock {The scaling limit of random simple triangulations and random simple
  quadrangulations}.
\newblock {\em {Ann. Probab.}}, 45(5):2767--2825, 2017.

\bibitem{zbMATH06847066}
L.~{Addario-Berry} and Y.~{Wen}.
\newblock {Joint convergence of random quadrangulations and their cores}.
\newblock {\em {Ann. Inst. Henri Poincar\'e, Probab. Stat.}}, 53(4):1890--1920,
  2017.

\bibitem{emt2022}
M.~Albenque, E.~Fusy, and T.~Leh\'{e}ricy.
\newblock Random cubic planar graphs converge to the {B}rownian map.
\newblock {\em In preparation}.

\bibitem{zbMATH06556653}
J.~{Ambj{\o}rn} and T.~G. {Budd}.
\newblock {Multi-point functions of weighted cubic maps}.
\newblock {\em {Ann. Inst. Henri Poincar\'e D, Comb. Phys. Interact. (AIHPD)}},
  3(1):1--44, 2016.

\bibitem{angel2005scaling}
O.~Angel.
\newblock Scaling of percolation on infinite planar maps, i, 2005.

\bibitem{MR2013797}
O.~Angel and O.~Schramm.
\newblock Uniform infinite planar triangulations.
\newblock {\em Comm. Math. Phys.}, 241(2-3):191--213, 2003.

\bibitem{zbMATH05122852}
M.~{Bodirsky}, M.~{Kang}, M.~{L\"offler}, and C.~{McDiarmid}.
\newblock {Random cubic planar graphs}.
\newblock {\em {Random Struct. Algorithms}}, 30(1-2):78--94, 2007.

\bibitem{zbMATH03217678}
W.~G. {Brown}.
\newblock {Enumeration of triangulations of the disk}.
\newblock {\em {Proc. Lond. Math. Soc. (3)}}, 14:746--768, 1964.

\bibitem{zbMATH06701703}
N.~{Curien} and J.-F. {Le Gall}.
\newblock {Scaling limits for the peeling process on random maps}.
\newblock {\em {Ann. Inst. Henri Poincar\'e, Probab. Stat.}}, 53(1):322--357,
  2017.

\bibitem{zbMATH07144469}
N.~{Curien} and J.-F. {Le Gall}.
\newblock {First-passage percolation and local modifications of distances in
  random triangulations}.
\newblock {\em {Ann. Sci. \'Ec. Norm. Sup\'er. (4)}}, 52(3):631--701, 2019.

\bibitem{MR2440928}
D.~Denisov, A.~B. Dieker, and V.~Shneer.
\newblock Large deviations for random walks under subexponentiality: the
  big-jump domain.
\newblock {\em Ann. Probab.}, 36(5):1946--1991, 2008.

\bibitem{coreprep}
M.~Drmota, M.~Noy, C.~Requil\'{e}, and J.~Ru\'{e}.
\newblock Enumeration and limit laws of rooted cubic planar maps.
\newblock {\em In preparation}.

\bibitem{dudley_2002}
R.~M. Dudley.
\newblock {\em Real Analysis and Probability}.
\newblock Cambridge Studies in Advanced Mathematics. Cambridge University
  Press, 2 edition, 2002.

\bibitem{zbMATH06841874}
W.~{Fang}, M.~{Kang}, M.~{Mo{\ss}hammer}, and P.~{Spr\"ussel}.
\newblock {Cubic graphs and related triangulations on orientable surfaces}.
\newblock {\em {Electron. J. Comb.}}, 25(1):research paper p1.30, 52, 2018.

\bibitem{MR2483235}
P.~Flajolet and R.~Sedgewick.
\newblock {\em Analytic combinatorics}.
\newblock Cambridge University Press, Cambridge, 2009.

\bibitem{2005math.....12304K}
M.~{Krikun}.
\newblock {Local structure of random quadrangulations}.
\newblock {\em ArXiv Mathematics e-prints}, Dec. 2005.

\bibitem{zbMATH02213763}
M.~A. {Krikun}.
\newblock {A uniformly distributed infinite planar triangulation and a related
  branching process}.
\newblock {\em {Zap. Nauchn. Semin. POMI}}, 307:141--174, 282--283, 2005.

\bibitem{lehericy2019firstpassage}
T.~Leh\'{e}ricy.
\newblock First-passage percolation in random planar maps and {T}utte's
  bijection.
\newblock {\em arXiv:1906.10079}, 2019.

\bibitem{zbMATH03927992}
T.~M. {Liggett}.
\newblock {An improved subadditive ergodic theorem}.
\newblock {\em {Ann. Probab.}}, 13:1279--1285, 1985.

\bibitem{MR2571957}
G.~Miermont.
\newblock Tessellations of random maps of arbitrary genus.
\newblock {\em Ann. Sci. \'Ec. Norm. Sup\'er. (4)}, 42(5):725--781, 2009.

\bibitem{zbMATH06827273}
M.~{Noy}, C.~{Requil\'e}, and J.~{Ru\'e}.
\newblock {Enumeration of labeled 4-regular planar graphs}.
\newblock In {\em Extended abstracts of the ninth European conference on
  combinatorics, graph theory and applications, EuroComb 2017, Vienna, Austria,
  August 28 -- September 1, 2017}, pages 933--939. Amsterdam: Elsevier, 2017.

\bibitem{zbMATH07213288}
M.~{Noy}, C.~{Requil\'e}, and J.~{Ru\'e}.
\newblock {Further results on random cubic planar graphs}.
\newblock {\em {Random Struct. Algorithms}}, 56(3):892--924, 2020.

\bibitem{10.5565/PUBLMAT6612213}
M.~Noy, C.~Requilé, and J.~Rué.
\newblock {On the expected number of perfect matchings in cubic planar graphs}.
\newblock {\em Publicacions Matemàtiques}, 66(1):325 -- 353, 2022.

\bibitem{zbMATH06653517}
K.~{Panagiotou}, B.~{Stufler}, and K.~{Weller}.
\newblock {Scaling limits of random graphs from subcritical classes}.
\newblock {\em {Ann. Probab.}}, 44(5):3291--3334, 2016.

\bibitem{zbMATH06639396}
C.~{Requil\'e} and J.~{Ru\'e}.
\newblock {Triangles in random cubic planar graphs}.
\newblock In {\em Extended abstracts of the eight European conference on
  combinatorics, graph theory and applications, EuroComb 2015, Bergen, Norway,
  August 31 -- September 4, 2015}, pages 383--391. Amsterdam: Elsevier, 2015.

\bibitem{zbMATH06729837}
B.~{Stufler}.
\newblock {Scaling limits of random outerplanar maps with independent
  link-weights}.
\newblock {\em {Ann. Inst. Henri Poincar\'e, Probab. Stat.}}, 53(2):900--915,
  2017.

\bibitem{stufler2022mass}
B.~Stufler.
\newblock Mass and radius of balls in {G}romov-{H}ausdorff-{P}rokhorov
  convergent sequences.
\newblock 2022.

\bibitem{stufler2022scaling}
B.~Stufler.
\newblock The scaling limit of random cubic planar graphs.
\newblock {\em arXiv-eprints}, 2022.

\bibitem{stufler2022uniform}
B.~Stufler.
\newblock The {U}niform {I}nfinite {C}ubic {P}lanar {G}raph.
\newblock {\em arXiv:2202.00592}, 2022.

\bibitem{zbMATH03169204}
W.~T. {Tutte}.
\newblock {A census of planar triangulations}.
\newblock {\em {Can. J. Math.}}, 14:21--38, 1962.

\end{thebibliography}

\end{document}